\documentclass[10pt,english,reqno]{smfart}
\usepackage{etoolbox}
\patchcmd{\thebibliography}{*}{}{}{}
\pretocmd\thebibliography{\csname c@secnumdepth\endcsname=-2 }{}{}

\usepackage{csquotes}
\usepackage[english,french]{babel}

\usepackage[margin=8em]{geometry}

\usepackage{amsmath,amsthm,amssymb,url,xspace,aliascnt,stmaryrd,mathtools}
\usepackage{enumitem} %resuming enumerate
\setlist[enumerate]{label=(\roman*)} %labelling for enumerate
\usepackage{soul} %highlight

\usepackage[nonewpage]{indextools}              % table de notations
\indexsetup{level=\subsection*}
\makeindex[columns=3, title=Index of notation]
\makeatletter
\@newctr{footnote}[page]
\makeatother

\usepackage{tikz-cd}

\usepackage{denpastyle}

\usepackage{rank-2-roots}

\makeatletter
\let\c@equation\c@subsection
\let\theequation\thesubsection
\makeatother

\addto\extrasfrancais{

}
\addto\extrasenglish{

}

\usepackage[backend=biber,doi=false,url=false,isbn=false]{biblatex}
\DeclareFieldFormat[article]{volume}{\textbf{#1}}
\DeclareFieldFormat[article]{number}{-#1}
\DeclareFieldFormat[book]{volume}{\textbf{#1}}
\renewbibmacro{in:}{%
  \ifentrytype{article}{}{\printtext{\bibstring{in}\intitlepunct}}}

\renewbibmacro*{volume+number+eid}{%
  \printfield{volume}
  \printfield[parens]{number}%
  \setunit{\addcomma\space}%
  \printfield{eid}}

\DeclareFieldFormat{postnote}{#1}
\DeclareFieldFormat{multipostnote}{#1}
\newcommand{\Db}{\rmD^{\rmb}}
\newcommand{\ubar}[1]{\underline{\smash{#1}}}
\newcommand{\pH}{\prescript{p}{}{\mathscr{H}}}
\newcommand{\pre}[1]{\prescript{#1}{}}
\newcommand{\Gz}{{G_{\ubar 0}}}
\newcommand{\Gzt}{{G_{\ubar 0, \diamond}}}
\newcommand{\Gzq}{{G_{\ubar 0,q}}}
\newcommand{\Gztq}{{G_{\ubar 0, \diamond,q}}}

\newcommand{\Hxid}{\bfH'_{\xi}}
\newcommand{\cHd}{\ha\calH'}

\newcommand{\Lztq}{L_{0,\diamond,q}}

\newcommand{\Pztq}{P_{0,\diamond,q}}
\newcommand{\Mt}{{M_{\diamond}}}
\newcommand{\Mtq}{{M_{\diamond,q}}}

\newcommand{\Mzt}{{M_{0,\diamond}}}
\newcommand{\Mztq}{{M_{0,\diamond,q}}}
\newcommand{\Ox}{{\calO_{\bfx,\eta/2m}(\Hxid)}}
\newcommand{\lambdat}{{\lambda_{\diamond}}}
\newcommand{\lambdatq}{{\lambda_{\diamond,q}}}
\newcommand{\lambdaq}{{\lambda_{q}}}
\newcommand{\xr}{{\lambda_q}}
\newcommand{\Tt}{{T_{\diamond}}}

\newcommand{\Cq}{{\bfC^{\times}_{q}}}
\newcommand{\Ct}{{\bfC^{\times}_{t}}}
\newcommand{\Ctm}{{\bfC^{\times}_{t^{1/2m}}}}

\newcommand{\autorefitem}[2]{\hyperref[#1]{\autoref*{#1}~(#2)}}
\title{Generalised Springer correspondence for $\bfZ/m$-graded Lie algebras}
\alttitle{Correspondance de Springer g\'en\'eralis\'ee pour les alg\`ebres de Lie $\bfZ/m$-gradu\'ees}
\date{\today}
\author{\sc Wille Liu\thanks{The first version of this work was realised during the author's doctoral programme of ED 386 hosted in the Institut Mathématique de Jussieu-Paris Rive Gauche. Later revisions are made during his postdoctoral fellowship in Max-Planck-Institut f\"ur Mathematik in Bonn.}}

\address{Max-Planck-Institut f\"ur Mathematik \\
Vivatsgasse 7 \\
53111 Bonn \\
Germany \\
}
\email{wille@mpim-bonn.mpg.de}
\begin{document}
\def\smfbyname{}

\begin{abstract}
Let $G$ be a simple simply connected complex algebraic group and let $\mathfrak{g}_*$ be a $\mathbf{Z}/m$-grading on its Lie algebra $\mathfrak{g}$. In a recent series of articles, G. Lusztig and Z. Yun, studied the classification of simple $G_0$-equivariant perverse sheaves on the nilpotent cone of $\mathfrak{g}_i$ for $i\in \mathbf{Z}/m$, where $G_0$ is the exponentiation of the degree zero piece $\mathfrak{g}_0$. They proved a decomposition of the equivariant derived category of $\ell$-adic sheaves on the nilpotent cone of $\mathfrak{g}_i$ into blocks, each generated by a certain cuspidal local system via {\itshape spiral inductions}. We prove a conjecture of them, which predicts the bijectivity of a map from 1) the set of simple perverse sheaves in a fixed block to 2) the set of simple modules of a block of a (trigonometric) degenerate double affine Hecke algebra (dDAHA). This is a dDAHA analogue of the Deligne--Langlands correspondence for affine Hecke algebras proven by Kazhdan--Lusztig. Our results generalise a previous work of E. Vasserot, where the perverse sheaves in the principal block were considered.
\end{abstract}
\begin{altabstract}
Soient $G$ un groupe alg\'ebrique complexe simple simplement connexe et $\mathfrak{g}_*$ une $\mathbf{Z}/m$-graduation sur $\mathfrak{g} = \Lie G$. Dans une s\'erie r\'ecente d'articles, G. Lusztig et Z. Yun ont \'etudi\'e la classification des faisceaux pervers simples $G_0$-\'equivariants sur le c\^one nilpotent de $\mathfrak{g}_i$ pour $i\in \mathbf{Z}/m$, o\`u $G_0$ est l'exponentialis\'e de la composante de degr\'e nul $\mathfrak{g}_0$. Ils ont \'etabli une d\'ecomposition en blocs de la cat\'egorie d\'eriv\'ee \'equivariante des faisceaux $\ell$-adiques sur le c\^one nilpotent de $\mathfrak{g}_i$; chacun des blocs est engendr\'e par un certain syst\`eme local cuspidal via les {\itshape inductions spirales}. Nous d\'emontrons leur conjecture qui pr\'edit la bijectivit\'e d'une application de 1) l'ensemble des faisceaux pervers simples dans un bloc donn\'e dans 2) l'ensemble des modules simples d'une alg\`ebre de Hecke doublement affine d\'eg\'en\'er\'ee. Ceci est pour les alg\`ebres de Hecke doublement affines d\'eg\'en\'er\'ees un r\'esultat analogue \`a la correspondance de Deligne--Langlands, d\'emontr\'ee par Kazhdan--Lusztig et portant sur les alg\`ebres de Hecke affines. Nos r\'esultats g\'en\'eralisent ceux d'un travail d'E. Vasserot, dans lequel seuls les faisceaux pervers dans le bloc principal \'etaient pris en compte.
\end{altabstract}
\maketitle

\section*{Introduction}
In the present article, we establish a generalised Springer correspondence for $\bfZ / m$-graded Lie algebras and (trigonometric) degenerate double affine Hecke algebras, which was conjectured by Lusztig--Yun~\cite{LYIII}. The main result is~\autoref{theo:classification}, which confirms the multiplicity-one conjecture proposed in~\cite{LYIII} and can be viewed as a generalised Springer correspondence in the sense of Lusztig~\cite{lusztig84,lusztig88}, for certain degenerate double affine Hecke algebras (dDAHAs) with possibly unequal parameters. \par

\subsection*{Generalised Springer correspondence}

G. Lusztig in~\cite{lusztig84} generalised the classical Springer correspondence on reductive groups by introducing cuspidal local systems. \par

Let $G$ be a complex connected reductive group and let $\frakg = \Lie G$ denote its Lie algebra, on which $G$ acts by the adjoint action. A {\it cuspidal pair} $(\rmO, \scrC)$ on $\frakg$ consists of a nilpotent $G$-orbit $\rmO\subset \frakg^{\nil}$ together with an irreducible $G$-equivariant local system $\scrC$ on $\rmO$ such that the following condition holds: for every strict parabolic subalgebra $\frakp\subsetneq \frakg$ with unipotent radical $\fraku\subseteq \frakp$ and for any element $x\in \rmO$, we have $\rmH^*_c((x + \fraku)\cap \rmO, \scrC) = 0$. \par

Given the reductive group $G$, the following statements are proven in~\cite{lusztig84}:
\begin{enumerate}
	\item 
		There is a partition of the set $\Irr \Perv_G(\frakg^{\nil})$ of isomorphism classes of simple $G$-equivariant perverse sheaves on the nilpotent cone $\frakg^{\nil}$ into {\it series}:
		\begin{equation*}\begin{aligned}
				\Irr \Perv_G(\frakg^{\nil}) = \bigsqcup_{\xi} \Irr \Perv_G(\frakg^{\nil})_{\xi},
		\end{aligned}\end{equation*}
		where $\xi$ runs over all triples $(M, \rmO, \scrC)$ up to conjugation, where $M\subset G$ is a Levi subgroup and $(\rmO, \scrC)$ is a cuspidal pair on the Lie algebra $\frakm$ of $M$.
	\item
		For each triple $\xi = \left( M, \rmO, \scrC \right)$, there is a crystallographic finite Coxeter group $W_{\xi}$, called {\it relative Weyl group}, with a bijection
		\begin{equation*}\begin{aligned}
				\Irr \Perv_G(\frakg^{\nil})_{\xi} \xlongrightarrow{\sim} \Irr(\bfC W_{\xi}\mof).
		\end{aligned}\end{equation*}
\end{enumerate}
The series $\Irr \Perv_G(\frakg^{\nil})_{\xi}$ is defined using the {\it Lusztig--Spaltenstein parabolic induction}. Given such a triple $\xi = (M, \rmO, \scrC)$, let $P$ be a parabolic subgroup of $G$ containing $M$ as Levi factor and let $U$ be the unipotent radical of $P$ with Lie algebra $\fraku = \Lie U$. Consider the following diagram of stacks:
\[
	[G\backslash \frakg] \xleftarrow{a}[G\backslash (G\times^P\left(\rmO\oplus\fraku\right)) ] \xrightarrow{\cong} [P\backslash(\rmO\oplus\fraku)]\xrightarrow{b} [M\backslash\frakm],
\]
where $\rmO\oplus \fraku$ is the image of the addition map $\rmO\times \fraku \xrightarrow{+} \frakp$.
The complex $\bfI^{\xi} = a_*b^*\scrC[\dim G\times^P(\rmO\oplus \fraku)]$ is a $G$-equivariant semisimple perverse sheaf and $\Irr\Perv_G\left( \frakg^{\nil} \right)_{\xi}$ consists of the simple constituents of $\bfI^{\xi}$. The bijection in statement (ii) was constructed by means of a ring isomorphism $\bfC W_{\xi} \xrightarrow{\cong} \End(\bfI^{\xi})$. 

\subsection*{Equivariant enhancements}

Let $\Db_G(\frakg^{\nil})$ be the equivariant category of Bernstein--Lunts~\cite{BL94}. In the above statements (i) and (ii), the $G$-equivariance is only a condition on perverse sheaves and higher extensions between perverse sheaves in $\Db_G(\frakg^{\nil})$ play no r\^ole there. \par
In~\cite{RR16}, higher extensions are taken into account for the statement (i). It is enhanced into the following:
\begin{enumerate}
	\item[(i\textsuperscript{+})]
		There is an orthogonal decomposition of triangulated category:
		\begin{equation*}\begin{aligned}
			\Db_G(\frakg^{\nil}) = \bigoplus_{\xi}\Db_G(\frakg^{\nil})_{\xi},
		\end{aligned}\end{equation*}
		where $\Db_G(\frakg^{\nil})_{\xi}$ is the thick triangulated subcategory generated by $\bfI^\xi$.
\end{enumerate}
For each triple $\xi$ as above, there is a {\it graded affine Hecke algebra} (graded AHA) $\underline\bfH_{\xi}$ : it is a degenerate version of the affine Hecke algebras introduced by G. Lusztig~\cite{lusztig89}, see~\autoref{subsec:rappel} for the definition of the graded AHA in our setting. Let $\Cq = \bfC^{\times}$ be a one-dimensional torus which acts linearly on the Lie algebra $\frakg$ by weight $-2$ and trivially on the group $G$. The perverse sheaf $\bfI^{\xi}$ acquires a $G\times \Cq$-equivariant structure. In~\cite{EM} and~\cite{lusztig95}, statement (ii) above is enhanced into the following:
\begin{enumerate}
	\item[(ii\textsuperscript{+})]
		There is an isomorphism of graded rings
\begin{equation}\begin{aligned}\label{equa:EML}
	\underline\bfH_{\xi}\xrightarrow{\cong} \bigoplus_{k\in \bfZ}\Hom_{\Db_{G\times \Cq}(\frakg^{\nil})}(\bfI^{\xi}, \bfI^{\xi}[k]).
\end{aligned}\end{equation}
\end{enumerate}
The summands of the right-hand side vanish except for $k \in 2\bfZ_{\ge 0}$. Taking quotient by the graded radicals of both sides of~\eqref{equa:EML}, we recover the aforementioned isomorphism $\bfC W_{\xi}\cong \End(\bfI^{\xi})$.
One can also replace $G$ with $G\times \Cq$ in statement (i\textsuperscript{+}).

\subsection*{Perverse sheaves on \texorpdfstring{$\bfZ$}{Z}-graded Lie algebras}
Keep the reductive group $G$ as above. If $\lambda\in\bfX_*(G)$  is a cocharacter, then it gives rise to a $\bfZ$-grading on $\frakg$ by 
\[
	\frakg = \bigoplus_{n\in\bfZ}\frakg_n,\quad \frakg_n = \left\{ x\in \frakg\;;\; \lambda(t)x = t^nx,\quad \forall t\in \bfC^\times\right\}.
\]
Fix $\eta\in\bfZ_{\neq 0}$. The fixed-point subgroup $G^{\lambda}$ acts on the subspace $\frakg_{\eta}$ by the adjoint action. The number of $G^{\lambda}$-orbits in $\frakg_{\eta}$ is finite. The classification of $G^{\lambda}$-equivariant simple perverse sheaves on $\frakg_{\eta}$ proven in~\cite{lusztig88,lusztig95,lusztig95b,EM} is the following:
\begin{enumerate}
	\item[(i')]
		There is a partition of the set $\Irr \Perv_{G^{\lambda}}(\frakg_{\eta})$ of isomorphism classes of simple $G^{\lambda}$-equivariant perverse sheaves on the subspace $\frakg^{\lambda}_{\eta}$ into {\it series}:
		\begin{equation*}\begin{aligned}
				\Irr \Perv_{G^{\lambda}}(\frakg_{\eta}) = \bigsqcup_{\xi} \Irr \Perv_{G^{\lambda}}(\frakg_{\eta})_{\xi},
		\end{aligned}\end{equation*}
		where $\xi$ runs over all triples $(M, \rmO, \scrC)$ up to $G^{\lambda}$-conjugation, where $M\subset G$ is a $\lambda$-stable Levi subgroup and $(\rmO, \scrC)$ is a cuspidal pair on $\frakm$ such that $\rmO\cap \frakg_{\eta}\neq \emptyset$.
	\item[(ii')]
		For each such triple $\xi = \left( M, \rmO, \scrC \right)$, there is a bijection
		\begin{equation*}\begin{aligned}
				\Irr \Perv_G(\frakg_{\eta})_{\xi}\xlongrightarrow{\sim} \Irr\underline\bfH_{\xi}\mof_{(\lambda,\eta/2)},
		\end{aligned}\end{equation*}
		where $\underline\bfH_{\xi}\mof_{(\lambda,\eta/2)}$ is a {\it block} of the category of finite dimensional $\underline\bfH_{\xi}$-modules, which depends on $\lambda$ and $\eta / 2$. 
\end{enumerate}
The series $\Irr \Perv_{G^{\lambda}}(\frakg_{\eta})_{\xi}$ is defined using a $\bfZ$-graded version of parabolic induction of the cuspidal local system $\scrC$. Statement (ii') can be derived from the isomorphism~\eqref{equa:EML} by the technique of equivariant localisation. We will review these results in~\autoref{sec:ZgrAHA} below. \par

One might want to generalise the above story to loop groups / affine Kac--Moody groups. On the algebraic side, the affinisation of graded affine Hecke algebra is the degenerate double affine Hecke algebra (dDAHA). On the geometric side, the affinisation of the Springer resolution is the affine Springer resolution. However, instead of working on the infinite-dimensional geometry of the loop groups and affine Springer resolutions, we will consider cyclic gradings on simple Lie algebras, see~\autoref{subsec:heuristic} for a heuristic explanation of the relation between $\bfZ$-gradings on loop Lie algebras and cyclic gradings on simple Lie algebras. \par

\subsection*{Perverse sheaves on \texorpdfstring{$\bfZ/m$}{Z/m}-graded Lie algebras and dDAHA}
Suppose that we are given a simply connected simple complex algebraic group $G$ together with a group homomorphism $\theta:\mu_m\to \Aut(G)$, where $m \in \bfZ_{> 0}$ and $\mu_m = \left\{ z\in \bfC^{\times}\;;\; z^m = 1 \right\}$. The map $\theta$ gives rise to a $\bfZ/m$-grading on the Lie algebra $\frakg = \Lie G$ by
\[
	\frakg = \bigoplus_{\ubar k\in \bfZ / m}\frakg_{\ubar k},\quad \frakg_{\ubar k} = \left\{x\in \frakg\;;\; \theta(\zeta)x = \zeta^k x,\quad \forall \zeta\in \mu_m  \right\}.
\]
Also let $\Gz = G^{\theta}$ denote the fixed points of $\theta$ in $G$. Fix $\eta\in \bfZ_{\neq 0}$ and let $\ubar \eta\in \bfZ / m$ be its congruence class. In~\cite{LYI}, G. Lusztig and Z. Yun addressed the problem of partitioning simple $\Gz$-equivariant perverse sheaves on the nilpotent cone $\frakg^{\nil}_{\ubar\eta} = \frakg^{\nil}\cap \frakg_{\ubar\eta}$ into {\it series} by introducing the crucial geometric tool of {\it spiral induction}. The result can be stated as follows:
\begin{enumerate}
	\item[(i'')]
		There is a partition of the set $\Irr \Perv_{\Gz}(\frakg^{\nil}_{\ubar \eta})$ of isomorphism classes of simple $\Gz$-equivariant perverse sheaves on $\frakg^{\nil}_{\ubar\eta}$ into {\it series}:
		\begin{equation*}\begin{aligned}
				\Irr \Perv_{\Gz}(\frakg^{\nil}_{\ubar\eta}) = \bigsqcup_{\xi} \Irr \Perv_{\Gz}(\frakg^{\nil}_{\ubar\eta})_{\xi},
		\end{aligned}\end{equation*}
		where $\xi$ runs over all {\it admissible systems} $(M, \frakm_*, \rmO, \scrC)$ on $\frakg_{\ubar \eta}$, introduced in~\cite[\S 3]{LYI}, see~\autoref{subsec:admsys}. Moreover, the orthogonal decomposition also holds:
		\begin{equation*}\begin{aligned}
			\Db_{\Gz}(\frakg^{\nil}_{\ubar\eta}) = \bigoplus_{\xi} \Db_{\Gz}(\frakg^{\nil}_{\ubar\eta})_{\xi}.
		\end{aligned}\end{equation*}
\end{enumerate}

In~\cite{LYIII}, the authors defined for each admissible system $\xi$ an affine root system and a dDAHA $\bfH_{\xi}$, see~\autoref{subsec:daha}. The main result of~\cite{LYIII} is the construction of an action of $\bfH_{\xi}$ on an infinite sum of the spiral induction from various spirals 
\begin{equation*}\begin{aligned}
	\bfI^{\xi} = \bigoplus_{\frakp_*}\Ind_{\frakp_{\eta}}^{\frakg_{\ubar \eta}}\scrC.
\end{aligned}\end{equation*}
They constructed an action of $\bfH_{\xi}$ on the perverse cohomology $\pH\bfI^{\xi}= \bigoplus_k\pH^k\bfI^{\xi}$ and they conjectured that the multiplicity space of each simple constituent of the perverse cohomology $\pH\bfI^{\xi}$ is a simple $\bfH_{\xi}$-module and this yields an injection
\begin{equation*}\begin{aligned}
		\Irr \Perv_{G_{\ubar 0}}(\frakg_{\ubar\eta})_{\xi}\hookrightarrow \Irr\calO(\bfH_{\xi}),
\end{aligned}\end{equation*}
where $\calO(\bfH_{\xi})$ is the category of integrable $\bfH_{\xi}$-modules, see~\autoref{subsec:OH}. The main result of this article confirms this conjecture:
\begin{enumerate}
	\item[(ii'')]
		\begin{theo*}[=\autoref{theo:classification}]
			There is a bijection 
			\[
				\Irr \Perv_{G_{\ubar 0}}(\frakg_{\ubar\eta})_{\xi}\xlongrightarrow{\sim} \Irr\calO_{\bfx, \eta/2m}(\bfH_{\xi}),
			\]
			where $\Irr\calO_{\bfx, \eta/2m}(\bfH_{\xi})$ is a block of $\calO(\bfH_{\xi})$.
		\end{theo*}
\end{enumerate}

\begin{rema*}
	It turns out that, in statement (i''), a naive $\bfZ/m$-graded version of parabolic induction is not enough to generate the series $\Irr \Perv_{\Gz}(\frakg_{\ubar\eta})_{\xi}$. This is related to the Gelfand--Kirillov dimension of simple modules of the dDAHA $\bfH_{\xi}$: only those simple perverse sheaves whose corresponding simple $\bfH_{\xi}$-modules are of maximal GK-dimension appear in the parabolic induction. This phenomenon is explained in~\cite{liu2019extension}. The {\it spiral induction} of the cuspidal local system $\scrC$ introduced in~\cite{LYI} is the right substitute. The notion of a spiral is closely related to parahoric subalgebras of loop Lie groups. We refer to~\autoref{subsec:spirals} for its definition. \par
\end{rema*}

We achieve the theorem by constructing in~\autoref{theo:Phi} a ring homomorphism
\[
	\Phi:\bfH_{\xi}\to \Ext^*_{\Gzt\times \Cq}(\bfI^{\xi}, \bfI^{\xi})_{0}=: \ha\calH,
\]
			\label{prop:surj}
			where $\Gzt$ is the semi-direct product of $\Gz$ with a {\it loop rotation torus} (\autoref{sec:strat}) acting on it. The ring $\ha\calH$ is the completion of the extension algebra of $\bfI^{\xi}$ at the augmentation ideal $\rmH^{>0}_{\Gzt\times \Cq}$ of equivariant coefficient ring (\autoref{sec:phi} for the precise statements). We then show in~\autoref{theo:density} that when $\ha\calH$ is equipped with a suitable topology, the image of $\Phi$ is dense. The technique of convolution algebras of~\cite{CG} can be applied and yields a correspondence~(\autoref{lemm:simples}) between simple smooth $\ha\calH$-modules and simple perverse sheaves. Finally, making use of the density of the image of $\Phi$, we construct in~\autoref{theo:modHH} an equivalence between the category of smooth $\ha\calH$-modules and $\calO_{\bfx, \eta/2m}(\bfH_{\xi})$, which provides an bijection of simple objects between both categories. These two bijections provide a geometric parametrisation of simple modules in $\calO_{\bfx, \eta/2m}(\bfH_{\xi})$. Our approach to construct $\Phi$ is via inducing the map~\eqref{equa:EML} from all parabolic subalgebras of $\bfH_{\xi}$, which are graded affine Hecke algebras. The proof of the density of $\Phi$ involves analysing the geometry of the Steinberg varieties. We refer to~\autoref{sec:strat} for an explanation of the strategy. \par

\subsection*{Related works}
The case where $M$ is a maximal torus and $\scrC$ is trivial has previously been studied by Vasserot in~\cite{vasserot05} using equivariant K-homology over Kashiwara's flag manifold. The double affine Hecke algebras (DAHA) obtained are those with equal parameters. \par

In~\cite{VV09}, M. Varagnolo and E. Vasserot used the action of DAHA constructed in~\cite{vasserot05} to classify the {\itshape positive slopes} for which the spherical simple module of the dDAHA (with equal parameters) is finite dimensional. The Springer fibres involved are those which are over the regular semisimple locus of the affine Lie algebra. Our classification result~\autoref{theo:classification} works essentially for {\itshape negative slopes}, where we consider torus fixed-points of Springer fibres over nilpotent elements of the affine Lie algebra. Nevertheless, with our construction of the map $\Phi$, which works equally well for positive slopes, one can hope for a generalisation of the result of~\cite{VV09} to dDAHAs at unequal parameters. We hope to study the relation between these two opposite cases in a future work. \par

In~\cite{Yun11}, Z. Yun constructed an action of the dDAHA (with equal parameters) using parabolic Hitchin moduli spaces of principal bundles, which is a global version of the affine Springer resolution. This result is later used in~\cite{OY16} to further geometrise the results of~\cite{VV09}. One can hope for a generalisation of these results with cuspidal local systems and dDAHAs with unequal parameters. \par

\subsection*{Organisation of the article}

The sections~\autoref{sec:ZgrAHA}--\autoref{sec:affineCox} are mostly a recollection of previously known results. We reformulate them in a slightly different language. \par

In~\autoref{sec:ZgrAHA}, we review the sheaf-theoretic construction of graded AHAs given in~\cite{lusztig88},~\cite{EM} and~\cite{lusztig95}. We prove in~\autoref{prop:compatibleZ} that this construction is compatible with parabolic induction. \par

In~\autoref{sec:strat}, we explain the relation between the loop Lie algebras and $\bfZ/m$-graded Lie algebra as well as the relation between affine Springer resolution and the spirals. We sketch the construction of $\Phi$ and the strategy of the proof of the density of $\Phi$. 

In~\autoref{sec:grad}, we review the geometric setup of $\bfZ/m$-graded Lie algebras introduced by Lusztig--Yun. We recall the part of their results that is important for the purpose of this article. In particular, we recall the notion of spirals and spiral induction. \par

In~\autoref{sec:affineCox}, we review the Coxeter complex, relative root system and relative affine Weyl group introduced in \cite{LYIII}. We discuss their relation with spirals. \par

In~\autoref{sec:phi}, we define the dDAHA $\bfH_{\xi}$ and the completed extension algebra $\ha\calH$. We construct also the homomorphism $\Phi:\bfH_{\xi}\to \ha\calH$. This is done by inducing the map $\Phi$ from parabolic subalgebras of $\bfH_{\xi}$ by means of spiral induction. \par

The sections~\autoref{sec:geomZ} and~\autoref{sec:convolution} are preparation for the proof of the density theorem~\autoref{theo:density}. \par

In~\autoref{sec:geomZ}, we describe our principal geometric objects: varieties $\calX^{\nu, \nu'}$, $\calT^{\nu}$ and $\calZ^{\nu,\nu'}$, which play the role of fixed point components of the product of partial flag variety, the partial affine Springer resolution and the Steinberg variety, respectively. The filtration by Bruhat order of the varieties $\calZ^{\nu,\nu'}$ will be important in the analysis of the convolution algebra. We prove in \autoref{prop:strZ} that the convolution product on $\calZ^{\nu,\nu'}$ respects the Bruhat order and we describe the good Bruhat strata. \par

In~\autoref{sec:convolution}, we interpret the algebra $\ha\calH$ as a convolution algebra. The main technical results are~\autoref{lemm:isostr} and~\autoref{prop:imagepsi}, which will play crucial r\^oles in the proof of~\autoref{theo:density} \par

In~\autoref{sec:density}, making use of the calculations of~\autoref{sec:convolution}, we show that the image of $\Phi$ is dense. \par

In~\autoref{sec:simple}, we derive some consequences of the density of $\Phi$. In particular, we establish the multiplicity-one conjecture of~\cite{LYIII}. Simple integrable $\bfH_{\xi}$-modules with prescribed eigenvalues are classified in~\autoref{theo:classification}. We also relate simple modules with the cohomology of torus fixed points of affine Springer fibres in~\autoref{theo:standard}.\par

In~\autoref{sec:exem}, we work out the correspondence stated in~\autoref{theo:classification} in a special case. \par

In~\autoref{sec:twisted}, we explain the modifications needed for {\itshape twisted} affine root systems.

\subsection*{Acknowledgement}
The author is grateful to E. Vasserot, under whose supervision this article is written. He feels also indebted to the referees, who have read this paper with great patience and have made an enormous number of corrections and precious suggestions. He would also like to thank Oya H., Tsai C.-C. and R. Walker for useful exchanges. 

\section*{Convention and notation}
All schemes and algebraic stacks concerned will be over the field of complex number $\bfC$, which can be replaced by any algebraically closed field of characteristic $0$ or of sufficiently large characteristic (depending on the type of $G$). \par
For any algebraic stack $\calX$, we denote by $\Db\left( \calX \right)$ \index{D@$\Db\left( X \right)$} the bounded derived category of constructible $\ba\bfQ_{\ell}$-sheaves on $\calX$ defined by Laszlo--Olsson~\cite{laszlo+olsson08ii}. We denote $\bfk = \ba\bfQ_{\ell}$. \par

For any algebraic group $G$ acting on $X$ (on the left), we denote by $\Db_G\left( X \right) = \Db\left( \left[G\backslash X\right] \right)$ the $G$-equivariant bounded derived category of $\bfk$-sheaves on $X$, or equivalently bounded derived category of $\bfk$-sheaves on the quotient stack $\left[ G\backslash X \right]$. We denote $\Perv_G(X)\subseteq\Db_G\left( X \right)$ \index{Perv@$\Perv_G(X)$} the subcategory of complexes whose image in $\Db(X)$ is perverse ; those are perverse sheaves on the stack $[G\backslash X]$ up to a shift. Denote by $\rmH^*_G = \rmH^*(\rmB G, \bfk)$ the coefficient ring of the $G$-equivariant cohomology. For objects $\scrF, \scrG\in \Db_G(X)$, we put $\Ext^n_G(\scrF, \scrG) = \Hom_{\Db_G(X)}(\scrF, \scrG[n])$ and $\Ext^*_G(\scrF, \scrG) = \bigoplus_{n\in \bfZ}\Ext^n_{G}(\scrF, \scrG)$. \par

On these derived categories, the six operations $\otimes, \scrHom, f_*, f^*, f_!, f^!$ will be understood as derived functors. We suppress the symbols $\rmR$ and $\rmL$ from $\rmR f_*$, $\otimes^{\rmL}$, etc. The bi-duality functor will be denoted $\bfD$, the perverse cohomology functors will be denoted $\pH^k$\index{H@$\pH^k$} for $k\in \bfZ$ and for any local system $\scrL$ supported on some locally closed subset, its intersection complex will be denoted $\IC(\scrL)\in \Perv_G(X)$\index{IC@$\IC(\scrL)$}.\par

When a capital Latin letter ($A, B, C, \ldots$) denote an algebraic group, we usually use the corresponding Fraktur small letter ($\fraka, \frakb, \frakc, \ldots$) to denote its Lie algebra. When $G$ is an algebraic group, we denote by $Z(G)$ or $Z_G$\index{Z@$Z(G), Z_G$} its centre and by $\frakz_G$ or $\frakz(\frakg)$\index{z@$\frakz_G, \frakz(\frakg)$} the centre of its Lie algebra. \par

For any algebraic group $G$, the set of one-parameter subgroups (= cocharacters) is denoted by $\bfX_*(G)$ \index{X@$\bfX_*(G)$}. We will adopt the notion of fractional cocharacters
\[
	\bfX_*(G)_{\bfQ} = \left\{ \mu / r\;;\; \mu\in \bfX_*(G),\quad r\in \bfZ_{\ge 1} \right\} / \left((\mu,r )\sim (\mu',r') \Leftrightarrow r'\mu = r\mu'\right).
\]
\index{X@$\bfX_*(G)_{\bfQ}$} We will also consider weight spaces of elements of $\bfX_*(G)_{\bfQ}$. If $\rho:G\to \Aut(V)$ is a rational representation and if $\mu\in\bfX_*(G)_{\bfQ}$, then for each $r\in \bfQ$ the weight space of $\mu$ of weight $r$ in $V$ is denoted by $\prescript{\mu}{r}V$. Namely, let $k\in \bfZ_{\ge 1}$ such that $k\mu\in \bfX_*(G)$ and $kr\in \bfZ$, then 
\begin{equation*}\begin{aligned}
	\prescript{\mu}{r}V = \left\{ v\in V\;;\;\Ad_{(k\mu)(t)}(v) = t^{kr}v, \forall t\in \bfC^{\times} \right\}.
\end{aligned}\end{equation*} \index{0@$\prescript{\mu}{r}V$}
Similarly, if $G$ acts on a space $X$, we let $X^{\mu}$ denote the fixed point of the cocharacter $k\mu\in \bfX_*(G)$.  \par

For any pair of adjoint functors $(L, R)$, we denote by $\id\to RL$ the adjunction unit and by $LR\to \id$ the adjunction co-unit.

\section{\texorpdfstring{$\bfZ$}{Z}-graded Lie algebras and graded AHAs revisited}\label{sec:ZgrAHA}

We recall the works of Lusztig~\cite{lusztig84}~\cite{lusztig95}~\cite{lusztig95b} and Evens--Mirkovi\'c~\cite{EM} on cuspidal local systems and graded affine Hecke algebras. At the end of this section, we prove in~\autoref{prop:compatibleZ} that their constructions are compatible with the parabolic induction.

\subsection{Cuspidal local systems and AHAs}\label{subsec:rappel}
Let $G$ be a connected reductive algebraic group over $\bfC$ and $M\subset G$ a Levi subgroup. Assume that there is a nilpotent $M$-orbit $\rmO\subset \frakm^{\nil}$ and a cuspidal $M$-equivariant local system $\scrC$ on~$\rmO$ in the sense of~\cite{lusztig84}. The centre $Z_M$ of $M$ may be disconnected but the neutral component $Z_M^{\circ}$ is a torus. To such a datum $\xi = (M, \rmO, \scrC)$, Lusztig associated in~\cite[2.4]{lusztig88} a {\itshape relative root system} $(\frakz_{M, \bfQ}, R_{\xi})$ on the space $\frakz_{M,\bfQ} = \bfX_*(Z_M)\otimes \bfQ$ equipped with the inner product induced from the Killing form of $G$. Let $R'_{\xi}\subset \frakz_{M, \bfQ}^*$ be the set of non-zero weights of $Z_M^{\circ}$ which appears in $\frakg$. The pair $(\frakz_{M, \bfQ}, R_{\xi}')$ is a possibly non-reduced root system. The set of roots $R_{\xi}$ is the subset of $R'_{\xi}$ consisting of indivisible roots. \par

The {\itshape relative Weyl group} is defined to be $W_{\xi} = N_G(Z_M^{\circ}) / M$. It coincides with the Weyl group of $(\frakz_{M, \bfQ}, R_{\xi})$. It is shown in~\cite[9.2]{lusztig84} that the adjoint action of $N_G(Z_M^{\circ})$ on $\frakm$ preserves the nilpotent orbit $\rmO$ and the cuspidal local system $\scrC$ has a $N_G(Z_M^{\circ})$-equivariant structure. When $M$ is a maximal torus of $G$, the relative root system $(\frakz_{M, \bfQ}, R_{\xi})$ is reduced to the usual root system of $G$. \par

Let $\Cq = \bfC^{\times}$ be the multiplicative group which acts on $\frakg$ by weight -2 and trivially on $G$. For any group $H$, we put $H_q = H\times \Cq$ unless it is defined otherwise. The $M$-equivariant local system $\scrC$ admits a unique $M_q$-equivariant enhancement. Choose any parabolic subalgebra $P\subset G$ containing $M$ as Levi factor and denote by $\pi:\frakp\to \frakm$ the projection. Let 
\[
	\dot\frakg = G\times^P \pi^{-1}(\rmO)\xrightarrow{a} \frakg^{\nil},\quad \ddot\frakg = \dot\frakg\times_{\frakg}\dot\frakg\xhookrightarrow{(p_1, p_2)} \dot\frakg\times \dot\frakg
\]
be the partial Springer resolution and the Steinberg variety. One extends $\scrC$ to a $G\times\Cq$-equivariant local system on $\dot\frakg$. The cleanness of cuspidal local systems~\cite{lusztig94} implies that $a_!\dot\scrC\cong a_*\dot\scrC$. By the Verdier duality, there is an isomorphism
\begin{equation*}\begin{aligned}
	\Ext^*_{G_q}(a_*\dot\scrC, a_*\dot\scrC) \cong \Ext^*_{G_q}(p_2^*\dot\scrC, p_1^!\dot\scrC).
\end{aligned}\end{equation*}
Using the cleanness of $\scrC$, in~\cite{EM} and~\cite{lusztig95} the authors defined a convolution product on $\Ext^*_{G_q}(p_2^*\dot\scrC, p_1^!\dot\scrC)$ which agrees with the Yoneda product of the left-hand side of the above isomorphism and they constructed a ring isomorphism 
\begin{equation}\begin{aligned}\label{equa:isomEM}
	\underline\bfH_{\xi} \xrightarrow{\cong} \Ext^*_{G_q}(p_2^*\dot\scrC, p_1^!\dot\scrC).
\end{aligned}\end{equation}
Here $\underline\bfH_{\xi}$\index{H@$\underline\bfH_{\xi}$} is the {\itshape graded affine Hecke algebra} attached to the datum $\xi = (M, \rmO, \scrC)$ introduced in~\cite{lusztig88}. \par

We recall here the definition of $\underline\bfH_{\xi}$. We pick as above a parabolic subgroup $P\subset G$ which contains $M$ as Levi factor. The parabolic $P$ yields a set of positive roots $R_{\xi}^{+, P} = \left\{\alpha\in R_{\xi}\;;\; \frakp_\alpha \neq \emptyset \right\}$ and a base $\Delta^P_{\xi}\subset R_{\xi}^{+,P}$. The relative Weyl group $W_{\xi}$ becomes a Coxeter group with the set of generators $\langle s_{\alpha} \;;\; \alpha\in \Delta^P_{\xi}\rangle$. The graded affine Hecke algebra $\underline\bfH_{\xi}^P$ associated with the datum $(\xi, \Delta^P_{\xi})$ is the associative algebra over the polynomial ring $\bfk[u]$ generated by the sets $\left\{ x^{\mu} \right\}_{\mu\in \bfX^*(Z^{\circ}_{M})}$ and $\left\{ s_{\alpha} \right\}_{\alpha\in \Delta^P_{\xi}}$ subject to the following relations:
\begin{equation*}\begin{aligned}
	x^{\mu} + x^{\nu} = x^{\mu + \nu},\quad
	\bfk[ s_{\alpha}\;;\; \alpha\in \Delta^P_{\xi}] \cong \bfk W_\xi \\
	s_{\alpha}x^{\mu} - x^{s_\alpha(\mu)}s_{\alpha} = u c_{\alpha} \langle \mu,\alpha^{\vee}\rangle. \\
\end{aligned}\end{equation*}
where $c_{\alpha}\in \bfZ_{\ge 2}$ is a constant defined in~\cite[2.10]{lusztig88}, which depends on $\xi$ (see also~\autoref{subsec:daha}). In particular, the set $\{x^{\mu}\}_{\mu\in \bfX^*(Z^{\circ}_M)}$ generates in $\underline\bfH_{\xi}$ a polynomial subalgebra $\bfk[\frakz_{M,\bfk}] = \Sym \frakz_{M,\bfk}^*$, where $\frakz_{M,\bfk} = \bfX_*(Z_M^{\circ})\otimes \bfk$, $\frakz^*_{M,\bfk} = \bfX^*(Z_M^{\circ})\otimes\bfk$. If $P$ and $P'$ are two parabolic subgroups of $G$ which have $M$ as Levi factor, then there is an element $\dot w\in N_G(Z^{\circ}_M)$ such that $\dot w P \dot w^{-1} = P'$. The image of $\dot w$ in the quotient $W_{\xi}$, denoted by $w$, then yields an isomorphism of based root system $(\frakz_{M, \bfQ}, \Delta^P_{\xi})\cong (\frakz_{M, \bfQ}, \Delta^{P'}_{\xi})$ and hence a (canonical) isomorphism of algebras $\underline\bfH^P_{\xi}\cong \underline\bfH^{P'}_{\xi}$. The graded affine Hecke algebra attached to $\xi$ is the inverse limit $\underline\bfH_{\xi} = \varprojlim_{P} \underline\bfH^P_{\xi}$, so that for each choice of parabolic $P$, there is an canonical isomorphism $\underline\bfH_{\xi}\cong \underline\bfH^{P}_{\xi}$.  \par

According to the eigenvalues of the action of the central subalgebra $\bfk[\frakz_{M,\bfk}]^{W_{\xi}}\otimes \bfk[u]\subset \underline\bfH_{\xi}$, the category of finite-dimensional $\underline\bfH_{\xi}$-modules can be decomposed into blocks:
\[
	\underline\bfH_{\xi}\mof_{\mathrm{fd}} = \bigoplus_{r\in \bfk}\;\bigoplus_{x\in \frakz_{M,\bfk} / W_{\xi}}\underline\bfH_{\xi}\mof_{(x, r)}.
\]

\subsection{Simple modules of the graded AHA}\label{subsec:simpleAHA}
Fix $\xi$ as in~\autoref{subsec:rappel}. We shall explain the Deligne--Langlands--Lusztig parametrisation of simple $\ubar\bfH_{\xi}$-modules. \par
Let $\lambda_0\in \bfX_*(G)$ and $\eta\in \bfZ$. Put $\lambdaq = (\lambda_0, \eta/2)\in \bfX_*(M_q)_{\bfQ}$. Assume that $\frakm^{\lambda_q} \cap  \rmO\neq\emptyset$. We have the following commutative square: \\
\begin{minipage}[m]{.5\linewidth}
	\[
		\begin{tikzcd}[column sep = 20pt, row sep=15pt]
			\dot\frakg^{\xr}\arrow{d}{a}\arrow[hookrightarrow]{r}{\dot\rho} & \dot\frakg \arrow{d}{a}\\
			\frakg^{\xr}\arrow[hookrightarrow]{r}{\rho} &\frakg\\
		\end{tikzcd}
	\]
\end{minipage}
\begin{minipage}[m]{.5\linewidth}
	$\dot\frakg =  G\times^P\pi^{-1}(\rmO) \xrightarrow{a} \frakg$ : partial Springer resolution  \\
	$\frakg^{\xr}, \dot\frakg^{\xr}$ : $\xr$-fixed points \\
	$\rho, \dot\rho$: inclusions of $\xr$-fixed points \\
\end{minipage}
The following equivariant localisation formula is proven in~\cite[4.10]{EM}:
\begin{equation*}\begin{aligned}\label{equa:isoExt}
	\Ext^*_{G_q}(a_*\dot\scrC, a_*\dot\scrC)_{\xr} \cong \Ext^*_{G^{\lambda_0}_q}\left(a_*\dot\rho^*\dot\scrC, a_*\dot\rho^*\dot\scrC\right)_{0},
\end{aligned}\end{equation*}
where the subscript $\xr$ means the completion of the extension algebra at $\xr$, equipped with the adic topology and the subscript $0$ means the completion at the augmentation ideal $\frako = \rmH^{>0}_{G^{\lambda_0}_q}$. We will explain in~\autoref{subsec:compl} about the completion. Combining this with~\eqref{equa:isomEM}, we obtain an injective ring map with dense image
\begin{equation}\begin{aligned}\label{equa:locEM}
	\Phi: \underline\bfH_{\xi}\to  \Ext^*_{G^{\lambda_0}_q}\left(a_*\dot\rho^*\dot\scrC, a_*\dot\rho^*\dot\scrC\right)_{0}=: \ha\calH.
\end{aligned}\end{equation}
\begin{rema}
	\begin{enumerate}
		\item
			In fact, the formula~\eqref{equa:isoExt} was stated and proven in terms of ``quotient by $\frako$'' rather than ``completion at $\frako$''. Nevertheless, its proof can be easily adapted for completion. 
		\item
			Instead of completion at $\xr$, one can take completion at $(\lambda_0/m,\eta/2m)$ for any $m\in \bfk^{\times}$ and the resulting completions will be isomorphic. In particular, we may take the completion at any fractional cocharacters.
	\end{enumerate}
\end{rema}
Choose an $\fraksl_2$-triple $\phi = (e, h, f)$ in $\frakm$ with $e\in \frakm^{\lambda_q}\cap\rmO$, $h\in\prescript{\lambda_0}{0}\frakm$ and $f\in \prescript{\lambda_0}{-\eta}\frakm$ and set 
\begin{equation}\begin{aligned}\label{equa:Mphi}
	Z^q_{M}(\phi)= \left\{ (g, q)\in M\times\Cq \;;\; \Ad_g(e) = q^2 e,\; \Ad_g(h) = h,\; \Ad_g(f) = q^{-2}f\right\}. 
\end{aligned}\end{equation}
The group $Z^q_M(\phi)$\index{Z@$Z^q_M(\phi)$} is commutative and there is an isomorphism
\begin{equation*}\begin{aligned}
	\iota:Z_M \times \bfC^{\times} &\xrightarrow{\sim} Z^q_{M}(\phi)  \\
	(g, q) &\mapsto (g q^{h}, q).
\end{aligned}\end{equation*}\footnote{This isomorphism is technical but extremely important in the calculation of the equivariant cohomology with cuspidal local systems. We will invoke it several times.}
The image of the cocharacter $\lambdaq\in \bfX_*(M_q)_{\bfQ}$ lies in $Z^q_M(\phi)$, so we can put $(\lambda^M, \eta/2) = \iota^*\lambda_q$ with $\lambda^M\in \bfX_*(Z_M)_{\bfQ}$\index{lambda@$\lambda^M$}. The element $\lambda^M$ does not depend on the choice of $\phi$.

By the general theory of extension algebras (see \cite[Ch 8]{CG}), the simple $\ha\calH$-modules are indexed by the simple constituents of the perverse cohomology $\pH a_*\dot\rho^*\dot\scrC$. On the other hand, the map $\Phi$ induces an equivalence between $\underline\bfH_{\xi}\mof_{(\lambda^M, \eta/2)}$ and the category of finite-dimensional $\ha\calH$-modules. These considerations together yield the following geometric parametrisation: 
\begin{theo}[{\cite[8.14]{lusztig95}}]\label{theo:AHA}
	The ring map $\Phi$ induces a natural bijection
	\begin{equation*}\begin{aligned}
			\Irr(\underline\bfH_{\xi}\mof_{(\lambda^M, \eta/2)})\quad\xlongrightarrow{\sim} \quad\left\{\text{simple perverse sheaf $\subseteq \bigoplus_{k\in \bfZ}\pH^k a_*\dot\rho^*\dot\scrC$}\right\}/\sim.
	\end{aligned}\end{equation*}
\end{theo}

\subsection{\texorpdfstring{$\bfZ$}{Z}-graded Lie algebras}\label{subsec:Zgr}
In~\cite{lusztig95b}, Lusztig introduced a geometric framework for studying the fixed-point variety $\dot\frakg^{\lambda_q}$ and the complex $a_*\dot\rho^*\dot\scrC$ that appears in the previous theorem. We briefly recall the basic constructions. \par

Let $G$ be a complex connected connected reductive group as above and let $\lambda_0\in \bfX_*(G)$ be a cocharacter. Then $\lambda_0$ gives rise to a $\bfZ$-grading on the Lie algebra : 
\[
	\frakg = \bigoplus_{n\in \bfZ}\frakg_n,\quad \frakg_n = \prescript{\lambda_0}{n}\frakg = \left\{ x\in \frakg\;;\; \Ad_{\lambda_0(t)}x = t^n x,\quad \forall t\in \bfC^{\times} \right\}. 
\]
Let $G_0 = G^{\lambda_0}$ denote the subgroup of $\lambda_0$-fixed points of $G$. For any subset $\Sigma\subset G$, we will write $\Sigma_0 = \Sigma\cap G_0$ and for any subset $\Sigma\subset \frakg$, we write $U_n = U\cap \frakg_{n}$.  \par

Let $\eta\in \bfZ_{\neq 0}$. We have $\frakg_{\eta}\subset \frakg^{\nil}$ and the adjoint action of $G_{0}$ on $\frakg_{\eta}$ has a finite number of orbits. Let $P \subset G$ be a $\lambda_0$-stable parabolic subgroup and let $P = L\ltimes U$ be a $\lambda_0$-stable Levi decomposition. The inclusion $\frakp_\eta\subset \frakg_\eta$ and the projection $\frakp_\eta\to \frakp_{\eta} / \fraku_{\eta} \cong\frakl_\eta$ yield a diagram of algebraic stacks
\begin{equation*}\begin{aligned}
	\left[ L_0 \backslash \frakl_{\eta} \right] \xleftarrow{b} \left[ P_0 \backslash \frakp_{\eta} \right] \xrightarrow{a} \left[ G_0 \backslash \frakg_{\eta} \right].
\end{aligned}\end{equation*}
The $\bfZ$-graded parabolic induction functor is defined by
\begin{equation*}\begin{aligned}
	\Ind^{\frakg_{\eta}}_{\frakp_{\eta}} = a_*\circ b^* : \Db_{L_{0,q}}(\frakl_{\eta})\to \Db_{G_{0,q}}(\frakg_{\eta}). 
\end{aligned}\end{equation*}\index{Ind@$\Ind^{\frakg_{\eta}}_{\frakp_{\eta}}$}
The induction functor satisfies the transitivity: for every $\lambda_0$-stable parabolic subgroup $Q \subset L$, there is an isomorphisms of functors %are isomorphisms of functors
\begin{equation*}\begin{aligned}
	\Ind^{\frakg_{\eta}}_{\frakp_{\eta}}\circ \Ind^{\frakl_{\eta}}_{\frakq_{\eta}} \cong \Ind^{\frakg_{\eta}}_{(\frakq\oplus \fraku)_{\eta}}.
\end{aligned}\end{equation*}

A {\itshape cuspidal pair} $(\rmO, \scrC)$ on $\frakg_\eta$ is a cuspidal pair on $\frakg$ such that $\rmO\cap \frakg_{\eta}\neq \emptyset$. If $(\rmO, \scrC)$ is a cuspidal pair, it happens that $\rmO_{\eta} = \rmO\cap \frakg_{\eta}$ must be the unique open $G_0$ orbit in $\frakg_{\eta}$ and for every $e\in \rmO_{\eta}$, the inclusion $Z_{G_0}(e)\subset Z_{G}(e)$ induces an isomorphism on the component groups, see~\cite[\S4]{lusztig95b}. It follows that the restriction $\scrC_{\eta} := \scrC\mid_{\rmO}$ is irreducible, and conversely, $\scrC$ is the unique $G$-equivariant extension of $\scrC_{\eta}$ on $\rmO$. An {\itshape admissible system} on $\frakg_{\eta}$ is a datum $(M, \rmO, \scrC)$ of a $\lambda_0$-stable Levi subgroup $M\subset G$ with Lie algebra $\frakm = \Lie M$ and a cuspidal pair $(\rmO, \scrC)$ on $\frakm_{\eta}$.

%We resume to the setting of ~\autoref{subsec:rappel} so that we have an a Levi $M$ and a cuspidal pair $(\rmO, \scrC)$ on it. Let $m\in \bfZ_{> 0}$ be a common denominator of $\bfx$ and $r$ so that $m\bfx\in \bfX_*(M)$ and $mr\in \bfZ$. Consider $\lambda_0 = m\bfx$, so it define a $\bfZ$-grading on $\frakg$. Let $\eta = 2mr$. Then it is obvious that $\frakg^{\xr} = \frakg^{\lambda_0}_{\eta}$ and $G^{\bfx} = G^{\lambda_0}$. 
%We shall consider the induction of the intersection complex of $\scrC$ through all possible $\lambda_0$-stable parabolic subalgebras $\frakp$ of $\frakg$ having $\frakl$ as Levi factor. %This statement will be explained in~\autoref{subsec:compsfixes} below. 
%For this purpose, we shall introduce a combinatorial language to describe those parabolic subalgebras $\frakp$. \par

%The fixed points of the partial springer resolution $a:\dot\frakg^{\lambda_{q}}\to \frakg^{\lambda_{q}}$ suggest that 

We shall analyse the extension algebra in the right-hand side of~\eqref{equa:locEM}. For this sake, we introduce some combinatorial language for the Coxeter complex of $W$, so as to handle the varieties of fixed points $\dot\frakg^{\xr}, \ddot\frakg^{\xr}$.

\subsection{Coxeter complex and canonical relative Weyl group}\label{subsec:coxZ}

Fix a datum $\xi = (M, \rmO, \scrC)$ as in~\autoref{subsec:rappel}. Choose a maximal torus $T\subset M_0$ and set $\frakt_{\bfQ} = \bfX_*(T)_\bfQ = \bfX_*(T)\otimes\bfQ$. Let $R = R(G, T)$ denote the set of roots and let $W = N_G(T) / T$ denote the Weyl group. The root hyperplanes decompose $\frakt_{\bfQ}$ into facets and the facets of maximal dimension are the Weyl chambers. Let $\frakF$ denote the set of facets (more precisely, the root hyperplanes yield a stratification of $\frakt_\bfR = \frakt_{\bfQ}\otimes_{\bfQ} \bfR$ into facets and $\frakF$ is its restriction to $\frakt_\bfQ\subset\frakt_\bfR$). The Weyl group $W$ acts on the set $\frakF$. There is a bijection:
\begin{equation*}\begin{aligned}
		\frakF  &\xlongrightarrow{\sim} & \left\{\text{parabolic subalgebra~} \frakp\subset \frakg\;;\; \frakt\subset \frakp  \right\} \\
		\sigma &\mapsto& \frakp^{\sigma} = \bigoplus_{r\ge 0}\prescript{y}{r}\frakg= \bigoplus_{\substack{\alpha\in \bfX^*(T) \\ \langle\alpha, y\rangle\ge 0}} \frakg_{\alpha}\qquad (y\in \sigma).
\end{aligned}\end{equation*}
For every facet $\sigma\in \frakF$, let $P^{\sigma}$ denote the corresponding parabolic subgroup of $G$ such that $\Lie P^{\sigma} = \frakp^{\sigma}$ and $L^{\sigma}$ its Levi factor containing $T$ with $\frakl^{\sigma} = \Lie L^{\sigma}$. \par

The subspace $\frakz_{M,\bfQ} = \bfX_*(Z_{M})_\bfQ\subset \frakt_{\bfQ}$ is a union of facets in $\frakF$ and its pointwise stabiliser coincides with the Weyl group $W_M = W(M, T)\subset W$. The vanishing locus of each of the relative roots $\alpha\in R_{\xi}$ gives a hyperplane in $\frakz_{M, \bfQ}$. The root hyperplane arrangement of $R_{\xi}$ coincides with the restriction to $\frakz_{M, \bfQ}$ of the root hyperplane arrangement of $R$. The $\frakz_{M,\bfQ}$-chambers $\nu$ are characterised by the property that $\frakl^{\nu} = \frakm$.  Let $\Xi\subset \frakF$ be the subset of $\frakz_{M,\bfQ}$-chambers. The normaliser $N_W(W_M)$ preserves $\frakz_{M,\bfQ}$ and the quotient $N_{W}(W_M) / W_M$ acts on $\frakz_{M,\bfQ}$. It is shown in~\cite[5.9]{lusztig76} that $W_{\xi}\cong N_W(W_M) / W_M$ and it acts simply transitively on $\Xi$.  Let $\Xi'$ be the set facets which are $W$-conjugate to some facet in $\Xi$. Consequently, the Weyl group $W$ acts transitively on $\Xi'$.  \par
\par

Each $\frakz_{M,\bfQ}$-chamber $\nu\in \Xi$ gives rise to a base $\Delta_{\xi}^{\nu}\subset R_{\xi}$. The base $\Delta^{\nu}_{\xi}$ gives a set of Coxeter generators $\{s_{\alpha}\}_{\alpha\in \Delta_{\xi}^{\nu}}$ for $W_{\xi}$. We define the {\itshape canonical relative Weyl group} as $(\calW_{\xi}, \Delta_{\xi}) = \varprojlim_{\nu} ( W_{\xi},\Delta^{\nu}_{\xi} )$. The elements of $\calW_{\xi}$ can written in the form $w = ( w^{\nu} )_{\nu\in \Xi}$, so that for each $y\in W_{\xi}$ and $\nu\in \Xi$ we have $w^{y\nu} = yw^{\nu}y^{-1}$. It is a Coxeter group canonically attached to the admissible system $\xi$. We define also the canonical version of the based root system $\left(\frakh_{\xi,\bfQ}, \Delta_{\xi} \right) = \varprojlim_{\nu} (\frakz_{M,\bfQ}, \Delta^{\nu}_{\xi})$, who has the Coxeter group $(\calW_{\xi}, \Delta_{\xi})$ as Weyl group. In particular, $(\calW_{\xi}, \Delta_{\xi})$ acts on $\frakh_{\xi}$ by orthogonal reflections.  \par

We define a $\calW_{\xi}$-action on $\Xi$ as follows: 
\begin{equation*}\begin{aligned}
		\calW_{\xi}\times \Xi\to \Xi,\quad  (w, \nu)\mapsto (w^{\nu})^{-1}\nu.
\end{aligned}\end{equation*}
This action is simply transitive and commutes with the $W_{\xi}$-action on $\Xi$. Recall the point $\lambda^M\in \frakz_{M,\bfQ}$. It can be shown that $\lambda^M$ is the image of $\lambda_0$ under the orthogonal projection $\frakt_{\bfQ}\to \frakz_{M,\bfQ}$ with respect to the Killing form. Let $W_{\xi,\lambda^M} = \Stab_{W_{\xi}}(\lambda^M)$. We denote $\ubar \Xi = W_{\xi,\lambda^M}\backslash \Xi$. The $\calW_{\xi}$-action on $\Xi$ descends to $\ubar\Xi$. For every $\nu\in \Xi$, let $\lambda_{\nu}\in \frakh_{\xi}$ denote the image of $\lambda^M$ under the isomorphism $(\frakh_{\xi}, \Delta_{\xi})\cong (\frakz_{M,\bfQ}, \Delta^{\nu}_{\xi})$. There is a bijection of $\calW_{\xi}$-sets:
\begin{equation}\begin{aligned}\label{equa:bijXi}
		\ubar \Xi &\xlongrightarrow{\sim} \left\{ \lambda_{\nu} \right\}_{\nu\in \Xi} \\
	\ubar\nu&\mapsto \lambda_{\nu},
\end{aligned}\end{equation}
where $\ubar\nu$ denotes the $W_{\xi,\lambda^M}$-orbit of some $\nu\in \Xi$.
\par

The graded Hecke algebra $\underline\bfH_{\xi}$ from~\autoref{subsec:rappel} can also be defined in terms of $(\frakh_{\xi}, \Delta_{\xi})$ and it contains the polynomial subalgebra $\bfS_{\xi} = \bfk[\frakh_{\xi}] = \Sym(\frakh_{\xi}\otimes \bfk)$ as well as the group ring $\bfC \calW_{\xi}$. It follows immediately that the possible $\bfS_{\xi}$-weights of modules in the block $\underline\bfH_{\xi}\mof_{(\lambda^M, \eta)}$ is $\{\lambda_{\nu}\}_{\ubar\nu\in \ubar\Xi}$, which is in bijection with $\ubar\Xi$. \par

\subsection{Description of the \texorpdfstring{$\xr$}{(λ,η/2)}-fixed components}\label{subsec:compsfixes}
Recall $\xr = (\lambda_0, \eta/2)\in \bfX_*(M\times\Cq)_{\bfQ}$. We choose a maximal torus $T\subset M_0$. It is easy to see that the $\lambda_q$-fixed point set $\frakg^{\xr}$ coincides with the $\eta$-weight space $\frakg_{\eta} = \prescript{\lambda_0}{\eta}\frakg$. \par

Consider the partial flag variety $\calP = G / P$. From \autoref{subsec:coxZ}, we know that there exists a unique $\nu_0\in \Xi$ such that $\frakp^{\nu_0} = \frakp$. For any $\nu\in \Xi'$, we have clearly $P^{\nu} = \dot w_{\nu} P\dot w^{-1}_{\nu}$ for any $\dot w_{\nu}\in N_G(T)$ such that its image $w_{\nu}\in W$ satisfies $w_{\nu}\nu_0 = \nu$. By~\cite{reeder95}, the variety of $\lambda_0$-fixed points of $\calP$ has the following description as disjoint union of connected components:
\begin{equation}\begin{aligned}\label{equa:composantesX}
	\calP^{\lambda_0} \cong \bigsqcup_{\ubar\nu\in W_{\lambda_0}\backslash \Xi'} G_0 / P^{\nu}_0,\quad g\dot w_{\nu}P\mapsfrom gP^{\nu}_0\quad \text{for $g\in G_0$},
\end{aligned}\end{equation}
where $W_{\lambda_0}$ is the stabiliser of $\lambda_0$ in $W$, $P^{\nu}_0 = (P^{\nu})^{\lambda_0}$ and $\dot w_{\nu}\in N_G(T)$ is as above. \par

For every facet $\nu\in \Xi'$, put $\rmO^{\nu} = \Ad_{\dot w_{\nu}}\rmO\subset \frakl^{\nu}$ and $\rmO^{\nu}_{\eta} = \rmO^{\nu}\cap \frakg_{\eta}$. Using the fibration $\dot\frakg\to \calP$, we derive from~\eqref{equa:composantesX} a similar description for the variety of $\xr$-fixed points of $\dot\frakg$:
\begin{equation}\begin{aligned}\label{equa:springer-decomp-i}
	\dot\frakg^{\xr} &\cong \bigsqcup_{\ubar\nu\in W_{\lambda_0}\backslash  \Xi'} \calT^\nu,\quad \calT^\nu = G_0\times^{P^{\nu}_0}(\rmO^{\nu}_\eta\oplus \fraku^{\nu}_{\eta}) \\
	(g\dot w, \Ad_{\dot w^{-1}}(z))&\mapsfrom (g, z) \quad \text{for $g\in G_0,\; z\in \rmO^{\nu}_{\eta}\oplus \fraku^{\nu}_z$}.
\end{aligned}\end{equation}

It is shown in~\cite{LYIII} that $\rmO^{\nu}_{\eta} = \emptyset$ if $\nu$ is not $W_{\lambda_0}$-conjugate to any $\frakz_{M,\bfQ}$-chamber and $\rmO^{\nu} = \rmO$ if $\nu$ is a $\frakz_{M, \bfQ}$-chamber. In the former case, we have $\calT^{\nu} = \emptyset$. Thus we are reduced to consider the subset $\Xi\subset \Xi'$ of $\frakz_{M,\bfQ}$-chambers. The formula~\eqref{equa:springer-decomp-i} can be rewritten as
\begin{equation}\begin{aligned}\label{equa:springer-decomp-i-bis}
	\dot\frakg^{\xr} \cong \bigsqcup_{\ubar\nu\in \ubar\Xi} \calT^\nu,\quad \calT^\nu = G_0\times^{P^{\nu}_0}(\rmO_\eta\oplus \fraku^{\nu}_{\eta}).
\end{aligned}\end{equation}
\par

For every facet $\sigma\in \frakF$, we abbreviate $\Ind_{\sigma} = \Ind^{\frakg_{\eta}}_{\frakp^{\sigma}_{\eta}}$ for the parabolic $\frakp^{\sigma}$ defined in~\autoref{subsec:coxZ}.  \index{Ind@$\Ind_{\sigma}$}
\begin{prop}\label{prop:indZgrad}
	The complex $a_*\dot\rho^*\scrC$ on $\frakg_{\eta}$ can be written in terms of parabolic inductions over graded Lie algebras:
	\begin{equation*}\begin{aligned}
		a_*\dot\rho^*\dot\scrC \cong \bigoplus_{\ubar\nu\in \ubar\Xi}\Ind_{\nu} \scrC_{\eta},
	\end{aligned}\end{equation*}
	where we have identified $\scrC_{\eta}$ with its direct image under the inclusion $\rmO_{\eta}\subset\frakm_{\eta} = \frakl^{\nu}_{\eta}$ for $\nu\in \Xi$. \par
\end{prop}
\begin{proof}
	It follows from~\eqref{equa:springer-decomp-i-bis} that
	\begin{equation*}\begin{aligned}
		a_*\dot\rho^*\dot\scrC \cong \bigoplus_{\ubar\nu\in \ubar\Xi}\Ind_{\nu} \left(\Ad_{\dot w^{-1}_{\nu}}^*\scrC_{\eta}\right),
	\end{aligned}\end{equation*}
	where $\Ad_{\dot w^{-1}_{\nu}}: \rmO \to \rmO$ is the adjoint action of $\dot w^{-1}_{\nu}$ on $\frakm$. As mentioned in~\autoref{subsec:rappel}, the cuspidal local system $\scrC$ has a $N_G(Z^{\circ}_M)$ equivariant structure. Since $\nu\in \Xi$, we must have $\dot w_{\nu}\in N_G(Z^{\circ}_M)$. The equivariance yields an isomorphism $\Ad^*_{\dot w^{-1}_{\nu}}\scrC \cong \scrC$, which restricts to $\Ad^*_{\dot w^{-1}_{\nu}}\scrC_{\eta} \cong \scrC_{\eta}$. 
\end{proof}

\subsection{Completion}\label{subsec:compl}
Recall that the canonical relative Weyl group $(\calW_{\xi}, \Delta_{\xi})$ acts on the vector space $\frakh_{\xi}$. We have the classical Chern--Weil isomorphisms
\begin{equation*}\begin{aligned}
	\Sym(\bfX^*(T\times \Cq)_\bfk)^{W}\cong \rmH^*_{G_q},\quad \Sym(\bfX^*(T\times \Cq)_\bfk)^{W_{\lambda_0}}\cong \rmH^*_{G_{0,q}}. \\
\end{aligned}\end{equation*}
Let $\rmH^*_{G_q,\lambda_{q}}$\index{HG@$\rmH^*_{G_q,\lambda_{q}}$} (resp. $\rmH^*_{G_{0,q},0}$\index{HG@$\rmH^*_{G_{0,q},0}$}) denote the completion of the ring $\rmH^*_{G_q}$ (resp. $\rmH^*_{G_{0,q}}$) at the maximal ideal generated by $f - f(\lambda_q)$ for $f\in\Sym(\bfX^*(T\times \Cq)_\bfk)^{W}$ (resp. at the augmentation ideal $\rmH^{> 0}_{G_{0,q}}$). There is an isomorphism $\rmH^*_{G_q,\lambda_q}\cong \rmH^*_{G_{0,q},0}$ which respects the adic topology on both sides. \par

For any coherent $\rmH^*_{G_q}$-module $\scrM$ (resp. coherent $\rmH^*_{G_{0,q}}$-module $\scrM$), denote 
\[
	\scrM_{\lambda_{q}} = \scrM\otimes_{\rmH^*_{G_q}}\rmH^*_{G_q, \lambda_q}\quad \text{(resp. $\scrM_{0} = \scrM\otimes_{\rmH^*_{G_{0,q}}}\rmH^*_{G_{0,q}, 0}$)}.
\]

By~\autoref{prop:indZgrad}, the homomorphism~\eqref{equa:locEM} can be written as
\begin{equation*}\begin{aligned}
	\Phi: \underline\bfH_{\xi}\to \bigoplus_{\ubar\nu,\ubar\nu'\in \ubar\Xi}\Ext^*_{G_{0,q}}\left(\Ind_{\nu'} \scrC_{\eta}, \Ind_{\nu} \scrC_{\eta} \right)_{0}=:\ha\calH.
\end{aligned}\end{equation*}
This homomorphism satisfies the following remarkable properties:
\begin{itemize}
	\item
		If $\scrM\in \underline\bfH_{\xi}\mof_{(\lambda^M, \eta/2)}$, then the $\underline\bfH_{\xi}$-module structure on $\scrM$ factors through $\Phi$.
	\item
		For each $\ubar\nu\in \ubar\Xi$, the composition
		\[
			\bfS_{\xi}\hookrightarrow \underline\bfH_{\xi} \xrightarrow{\Phi} \ha\calH\to \Ext^*_{G_{0,q}}\left(\Ind_{\nu} \scrC_{\eta}, \Ind_{\nu} \scrC_{\eta} \right) / \rmH^{>0}_{G_{0,q}}=: \scrN
		\]
		makes $\scrN$ into a left (resp. right) $\bfS_{\xi}$-module supported on eigenvalue $(\lambda_{\nu}, \eta/2)\in \frakh_{\xi,\bfk}\times \bfk$.
\end{itemize}
The second point gives a geometric meaning to the combinatorial bijection~\eqref{equa:bijXi}; in other words, for each $\ubar\nu\in \ubar \Xi$, the eigenvalue $(\lambda_{\nu}, \eta/2)$ of the $\bfS_{\xi}$-action is accounted for by the connected component $\calT^{\nu}\subset \dot\frakg^{\lambda_q}$.

\subsection{Compatibility with parabolic induction}
Let $J\subset \Delta_{\xi}$ be a subset. Denote by $\calW_{\xi, J}\subset \calW_{\xi}$ the parabolic subgroup generated by $\{s_{\alpha}\}_{\alpha\in J}$. For any $\calW_{\xi,J}$-orbit $S\in \Xi / \calW_{\xi,J}$, there is a unique $\frakz_{M,\bfQ}$-facet $\sigma$ such that $\ba\sigma = \bigcap_{\nu\in S} \ba\nu$.  Let $\Xi_J\subset \frakF$ denote the set of facets $\sigma$ arising in this way. Then each $\frakz_{M,\bfQ}$-facet belongs to $\Xi_J$ for some $J\subseteq\Delta_{\xi}$. Conversely, for each $\frakz_{M,\bfQ}$-facet $\sigma\in\Xi_J$, we denote by $\Xi^{\sigma}\subset \Xi$ the subset of $\Xi$ consisting of $\frakz_{M,\bfQ}$-chambers whose closure contains $\sigma$. Then $\Xi^{\sigma}$ forms a $\calW_{\xi, J}$-orbit in $\Xi$. This yields a bijection of $W_{\xi}$-sets $\Xi / \calW_{\xi,J}\cong \Xi_J$ for each $J\subset\Delta_{\xi}$. \par
As before, we denote $\ubar \Xi_J = W_{\xi,\lambda^M}\backslash\Xi_J$ and $\ubar \Xi^{\sigma} = W_{\xi,\sigma,\lambda^M}\backslash\Xi^{\sigma}$, where $W_{\xi,\sigma,\lambda^M} = \Stab_{W_{\xi}}(\sigma, \lambda^M)$ is the pointwise stabiliser of $\sigma$ and $\lambda^M$. We denote by $\ubar\sigma$ the image of $\sigma$ in $\ubar \Xi/\calW_{\xi,J}$. 

We define $P^{\nu\le \sigma} = L^{\sigma}\cap P^{\nu}$ for any $\nu\in \Xi^{\sigma}$, so that $P^{\nu\le \sigma}$ is a parabolic subgroup of $L^{\sigma}$ with Levi factor $L^{\nu} = M$. For simplifying the notation, we abbreviate \index{Ind@$\Ind^{\sigma}_{\tau}$}
\[
	\Ind^{\sigma}_{\nu} = \Ind^{\frakl^{\sigma}_{\eta}}_{\frakp^{\sigma\le\nu}_{\eta}}.
\]
The map~\eqref{equa:locEM} for the Levi subgroup $L^{\sigma}$ and the admissible system $\xi$ on $\frakl_{\eta}$ is the following:
\[
	\Phi_{\sigma}: \underline\bfH_{\xi, J}\to  \bigoplus_{\ubar\nu,\ubar\nu'\in \ubar\Xi^\sigma}\Ext^*_{L^{\sigma}_{0,q}}\left(\Ind^{\sigma}_{\nu'}\scrC_{\eta}, \Ind^{\sigma}_{\nu}\scrC_{\eta}\right)_{0}=: \ha\calH_{\sigma},
\]
where the subscript $0$ means the completion at the augmentation ideal $\rmH^{>0}_{L^{\sigma}_{0,q}}$. Taking product over all classes $\ubar\sigma\in \ubar\Xi_J$, we obtain a homomorphism:
\begin{equation*}\begin{aligned}\label{equa:PhiJ}
	\Phi_{J} = \left( \Phi_{\sigma} \right)_{\ubar\sigma\in \ubar\Xi_J}: \underline\bfH_{\xi, J}\to \prod_{\ubar\sigma\in \ubar\Xi_J}\ha\calH_{\sigma}.
\end{aligned}\end{equation*}
On the other hand, the functoriality of the functor $\Ind_{\sigma}$ and the transitivity of parabolic inductions $\Ind_{\sigma}\circ\Ind^{\sigma}_{\nu}\cong \Ind_{\nu}$ yields a map
\[
	\psi_J: \prod_{\ubar\sigma}\ha\calH_{\sigma}  \to \prod_{\ubar\sigma}\bigoplus_{\ubar\nu,\ubar\nu'\in \ubar\Xi^\sigma}\Ext^*_{G_{0,q}}\left(\Ind_{\nu'}\scrC_{\eta}, \Ind_{\nu}\scrC_{\eta}\right)_0\subset \ha\calH.
\]
\begin{prop}\label{prop:compatibleZ}
	The construction of $\Phi$ is compatible with parabolic induction in the sense that $\Phi\mid_{\underline\bfH_{\xi,J}} = \psi_J\circ\Phi_J$.
\end{prop}
\begin{proof}
	Recall the usual (ungraded) parabolic induction functor: for each parabolic subgroup $P\subseteq G$ with Levi factor $M$, the diagram
	\[
		[\frakm / M] \xleftarrow{b} [\frakp / P]\xrightarrow{a} [\frakg / G]
	\]
	gives the induction functor $\Ind^{\frakg}_{\frakp}:= a_*b^*:\Db_{M_q}(\frakm^{\nil})\to \Db_{G_q}(\frakg^{\nil})$. By the $N_G(Z^{\circ}_M)$-equivariance (\autoref{subsec:rappel}) of $\scrC$, for any other parabolic $P'$ having $M$ as Levi factor, there is a canonical choice of isomorphism $\Ind^\frakg_{\frakp'}\scrC\cong\Ind^\frakg_{\frakp}\scrC$; thus we may write $\Ind^{\frakg}_{\frakm} \scrC := \Ind^{\frakg}_{\frakp} \scrC$. Similarly, for any $\sigma\in \Xi_J$, we may write $\Ind^{\frakl^{\sigma}}_{\frakm} \scrC := \Ind^{\frakl^{\sigma}}_{\frakp^{\sigma\le \nu}} \scrC$ by choosing an arbitrary $\nu\in \Xi^{\sigma}$. \par
	Choose any $\sigma\in \Xi_J$ and $\nu\in \Xi^{\sigma}$. It follows from Lusztig's construction~\cite{lusztig88},~\cite{lusztig95} of isomorphism $\underline\bfH_{\xi}\cong \Ext^*_{G_q}(a_* \dot\scrC, a_* \dot\scrC)$ that the following diagram commutes
	\begin{equation}\begin{aligned}\label{equa:HH}
		\begin{tikzcd}
			\underline\bfH_{\xi, J} \arrow[hookrightarrow]{d} \arrow{r}{\cong} & \Ext^*_{L^{\sigma}_q}(\Ind^{\frakl^{\sigma}}_{\frakm} \scrC, \Ind^{\frakl^{\sigma}}_{\frakm} \scrC)\arrow{r}{\Ind} &  \Ext^*_{G_q}(\Ind^{\frakg}_{\frakp^{\sigma}}\Ind^{\frakl^{\sigma}}_{\frakm} \scrC, \Ind^{\frakg}_{\frakp^{\sigma}}\Ind^{\frakl^{\sigma}}_{\frakm} \scrC)\arrow{d}{\cong}\\
			\underline\bfH_{\xi}\arrow{rr}{\cong} & &\Ext^*_{G_q}(\Ind^{\frakg}_{\frakm} \scrC, \Ind^{\frakg}_{\frakm} \scrC) \\
		\end{tikzcd}.
	\end{aligned}\end{equation}
	Denote the composite map of the top-right corner by
	\begin{equation}\begin{aligned}\label{equa:Ind} 
		\Ind^{\frakg}_{\frakp^{\sigma}}: \Ext^*_{L^{\sigma}_q}(\Ind^{\frakl^{\sigma}}_{\frakm} \scrC, \Ind^{\frakl^{\sigma}}_{\frakm} \scrC)\to \Ext^*_{G_q}(\Ind^{\frakg}_{\frakm} \scrC, \Ind^{\frakg}_{\frakm} \scrC).
	\end{aligned}\end{equation}
	For each $\sigma\in \Xi_J$, applying~\autoref{lemm:loc} below and taking conjugation by Weyl group elements, we obtain 
	\[
		\Ext^*_{L^{\sigma}_q}(\Ind^{\frakl^{\sigma}}_{\frakm} \scrC, \Ind^{\frakl^{\sigma}}_{\frakm} \scrC)\otimes_{\rmH^*_{G_q}}\rmH^*_{G_q, \lambda_q}\cong \prod_{\ubar\sigma'\in \ubar\Xi_J}\Ext^*_{L^{\sigma'}_q}(\Ind^{\frakl^{\sigma'}}_{\frakm} \scrC, \Ind^{\frakl^{\sigma'}}_{\frakm} \scrC)_{\lambda_q}.
	\]
	Completing the both sides of~\eqref{equa:Ind} at $\lambda_q$ over $\rmH^*_{G_q}$ and re-inserting it into~\eqref{equa:HH}, we obtain the following commutative diagram:
	\begin{equation*}\begin{aligned}
		\begin{tikzcd}
			\underline\bfH_{\xi, J} \arrow[bend left=15]{rr}{\Phi_J} \arrow[hookrightarrow]{d} \arrow{r} &  \prod_{\ubar\sigma}\Ext^*_{L^{\sigma}_q}(\Ind^{\frakl^{\sigma}}_{\frakm} \scrC, \Ind^{\frakl^{\sigma}}_{\frakm} \scrC)_{\lambda_{q}}\arrow{d}{\left( \Ind^{\frakg}_{\frakp^{\sigma}} \right)_{\ubar\sigma}} \arrow{r}[swap]{\text{loc}} &   \prod_{\ubar\sigma}\ha\calH_{\sigma} \arrow{d}{\psi_J}\\
			\underline\bfH_{\xi}\arrow[bend right=15]{rr}{\Phi}\arrow{r} & \Ext^*_{G_q}(\Ind^{\frakg}_{\frakm} \scrC, \Ind^{\frakg}_{\frakm} \scrC)_{\lambda_q}\arrow{r}{\text{loc}} & \ha\calH \\
		\end{tikzcd},
	\end{aligned}\end{equation*}
	where the homomorphisms denoted by loc are the equivariant localisation from~\cite[4.10]{EM}. This proves the statement $\psi_J\circ\Phi_J = \Phi\mid_{\underline\bfH_{\xi, J}}$. 
\end{proof}
\begin{lemm}\label{lemm:loc}
	For each $\sigma\in \Xi_J$, there is an isomorphism 
	\[
		\Ext^*_{L^{\sigma}_q}(\Ind^{\frakl^{\sigma}}_{\frakm} \scrC, \Ind^{\frakl^{\sigma}}_{\frakm} \scrC)\otimes_{\rmH^*_{G_q}}\rmH^*_{G_q, \lambda_q}\cong \prod_{w}\Ext^*_{L^{\sigma}_q}(\Ind^{\frakl^{\sigma}}_{\frakm} \scrC, \Ind^{\frakl^{\sigma}}_{\frakm} \scrC)_{w^{-1}\lambda_q},
	\]
	where $w$ runs over a complete set of representatives for $W_{\xi,\lambda^M}\backslash W_{\xi} / W_{\xi,\sigma}$ in $W_{\xi}$.  
\end{lemm}
\begin{proof}
	Since $\rmH^*_{L^{\sigma}}\cong\rmH^*_{P^{\sigma}} \cong \rmH^*_{G}(G/P^{\sigma}, \bfk)$, applying the equivariant localisation, we obtain
	\begin{equation}\begin{aligned}\label{equa:complGL}
		\rmH^*_{L^{\sigma}_q}\otimes_{\rmH^*_{G_q}}\rmH^*_{G_q, \lambda_q} &\cong  \rmH^*_{G^{\lambda_0}_q}((G/P^{\sigma})^{\lambda_0}, \bfk)_{\lambda_q} \cong \prod_{w}\rmH^*_{G^{\lambda_0}_q}((G^{\lambda_0}/(\pre{w}P^{\sigma}))^{\lambda_0}, \bfk)_{\lambda_q} \\
		&\cong \prod_{w}\rmH^*_{(L^{w\sigma}_q)^{\lambda_0},\lambda_q} \cong \prod_{w}\rmH^*_{L^{w\sigma}_q,\lambda_q} \cong \prod_{w}\rmH^*_{L^{\sigma}_q,w^{-1}\lambda_q},
	\end{aligned}\end{equation}
	where $w$ runs over a complete set of representatives for $W_{\lambda_0}\backslash W / W_{\sigma}$ in $W = W(G,T)$. Taking tensor product of $\Ext^*_{L^{\sigma}_q}(\Ind^{\frakl^\sigma}_{\frakm}\scrC, \Ind^{\frakl^\sigma}_{\frakm}\scrC)$ with~\eqref{equa:complGL}, we obtain
	\[
		\Ext^*_{L^{\sigma}_q}(\Ind^{\frakl^{\sigma}}_{\frakm} \scrC, \Ind^{\frakl^{\sigma}}_{\frakm} \scrC)\otimes_{\rmH^*_{G_q}}\rmH^*_{G_q, \lambda_q}\cong \prod_{w}\Ext^*_{L^{\sigma}_q}(\Ind^{\frakl^{\sigma}}_{\frakm} \scrC, \Ind^{\frakl^{\sigma}}_{\frakm} \scrC)_{w^{-1}\lambda_q}.
	\]
	It remains to show that for each $w\in W$ such that $W_{\lambda_0}wW_{\sigma}\cap N_W(W_M) = \emptyset$, the corresponding factor in this product vanishes. It is known (see~\cite[4.3]{lusztig88} or~\autoref{lemm:SE} below) that, as $\rmH^*_{L^{\sigma}_q}$-module, the cohomology $\Ext^*_{L^{\sigma}_q}(\Ind^{\frakl^{\sigma}}_{\frakm} \scrC, \Ind^{\frakl^{\sigma}}_{\frakm} \scrC)$ is supported on in image of the following map
	\begin{equation}\begin{aligned}\label{equa:XH}
		\bfX_*(Z^q_{M}(\phi))_{\bfk}\to \bfX_*(T_q)_{\bfk} / W_{\sigma} \cong \Spec \rmH^*_{L^{\sigma}_q},
	\end{aligned}\end{equation}
	where $Z^q_{M}(\phi)\subset M_q$ is defined in the same way as in~\autoref{subsec:simpleAHA} for any choice of $\fraksl_2$-triple $\phi$. Thus $w^{-1}\lambda_q$ lies in this support if and only if there exists $y\in W_{\sigma}$ such that $yw^{-1}\lambda_q\in \bfX_*(Z^q_{M}(\phi))_{\bfQ}$. \par
	Suppose that $w\in W$ is such that $w^{-1}\lambda_q$ lies in the image of~\eqref{equa:XH}. Since $W_{\xi}$ acts transitively on $\Xi$, there exists $v\in W_{\xi}$ such that $v\lambda_q = yw^{-1}\lambda_q$. Since $W_{\xi}\cong N_W(W_M) / W_M$, we may choose any lifting $\dot v\in N_W(W_M)$ of $v$, we see that $y':= \dot v^{-1}yw^{-1}\in W_{\lambda_0}$. Therefore in this case $\dot v^{-1} = y'w y^{-1}\in W_{\lambda_0}w W_{\sigma}\cap N_W(W_M)\neq\emptyset$. It follows that the condition $W_{\lambda_0}wW_{\sigma}\cap N_W(W_M) = \emptyset$ implies $\Ext^*_{L^{\sigma}_q}(\Ind^{\frakl^{\sigma}}_{\frakm} \scrC, \Ind^{\frakl^{\sigma}}_{\frakm} \scrC)_{w^{-1}\lambda_q} = 0$, which concludes the proof.
\end{proof}

\section{Outline of the strategy}\label{sec:strat}
In this section, we explain the strategy to prove the main theorem. The discussion here will be informal and serve as guideline for the sections after.

\subsection{Springer correspondence for dDAHAs : heuristics}\label{subsec:heuristic}
Let $G$ be a simply connected simple algebraic group over $\bfC$ with Lie algebra $\frakg$. Consider the loop group $G_{\aff} = G(\!(\varpi)\!)$ and the loop Lie algebra $\frakg_{\aff} = \frakg(\!(\varpi)\!)$. Let $\Ct = \bfC^{\times}$\index{C@$\Ct$} be the multiplicative group which acts on the ring of formal Laurent series $\bfC(\!(\varpi)\!)$ by $(t, \varpi)\mapsto t \varpi$. It induces a $\Ct$-action on $G_{\aff}$ and on $\frakg_{\aff}$. Let $\Cq = \bfC^{\times}$\index{C@$\Cq$} be the multiplicative group which acts on $\frakg_{\aff}$ by weight $-2$ and trivially on $G_{\aff}$. For any $\Ct$-stable subgroup $G'\subset G_{\aff}$, we denote $G'_{\diamond} = G'\rtimes \Ct$ and $G'_{\diamond, q} = G'_{\diamond}\times \Cq$. Hereafter, the subscript $\diamond$ will indicate the presence of the loop-rotation torus $\Ct$.
\par

Suppose that we have a $\Ct$-stable parahoric subgroup $P\subset G_{\aff}$ with $\Ct$-stable Levi factor $M\subset P$ and an $M_{\diamond}$-equivariant cuspidal local system $\scrC$ on a nilpotent $M_{\diamond}$-orbit $\rmO\subset \frakm^{\nil}$. Denote by $\pi: \frakp\to \frakm$ the projection. One would like to perform the same constructions on the extension algebra of the induced complex $\Ind^{\frakg_{\aff}}_{\frakp} \scrC$ on the loop Lie algebra $\frakg_{\aff}$. Imitating the case of graded affine Hecke algebras, one defines the partial Springer resolution
\[
	\dot\frakg_{\aff} =  G_{\aff}\times^P \pi^{-1}(\rmO)\xrightarrow{a} \frakg_{\aff}
\]
and extends $\scrC$ to a $ G_{\aff,\diamond,q}$-equivariant local system $\dot\scrC$ on $\dot\frakg_{\aff}$. Fix such a system $\xi = (G_{\aff}, P, M, \rmO, \scrC)$. One would hope to carry out the same constructions as~\autoref{sec:ZgrAHA} in this affine context with $\xi$. \par

Let $\ddot\frakg_{\aff} = \dot\frakg_{\aff}\times_{\frakg_{\aff}}\dot\frakg_{\aff}$ and let $p_1, p_2:\ddot\frakg_{\aff}\to \dot\frakg_{\aff}$ be the projections. By the Verdier duality, the extension algebra $\Ext^*_{ G_{\aff,\diamond,q}}(\Ind^{\frakg_{\aff}}_{\frakp}\scrC, \Ind^{\frakg_{\aff}}_{\frakp}\scrC)$ would then be isomorphic to $\Ext^*_{ G_{\aff,\diamond,q}}(p_1^*\dot\scrC, p_2^!\dot\scrC)$ equipped with an appropriate convolution product. One would then need to construct a ring homomorphism from the dDAHA $\bfH_{\xi}$ (to be defined in \autoref{subsec:daha}) to the convolution algebra $\Ext^*_{ G_{\aff,\diamond,q}}(p_1^*\dot\scrC, p_2^!\dot\scrC)$. \par

Let $\phi=(e, h, f)$ be a $\fraksl_2$-triple in $\frakm$ such that $e\in \rmO$ and let $Z^q_{M_{\diamond}}(\phi)$ be the group defined in~\eqref{equa:Mphi} so that there is an isomorphism 
\begin{equation*}\begin{aligned}
	\iota: Z_{M_\diamond}\times \bfC^{\times} \cong Z^q_{M_\diamond}(\phi) \\
	(g, q) \mapsto (g q^h, q).
\end{aligned}\end{equation*}
Fix $\lambda^M\in \bfX_*(Z_{M_\diamond})$ and $\eta\in \bfZ_{\neq 0}$. Write $\lambdatq = \iota_*(\lambda^M,\eta/2)\in \bfX_*(Z^q_{M_\diamond}(\phi))_{\bfQ}$. Let $\delta\in \bfX^*(G_{\aff,\diamond})$ denote the projection $G_{\aff,\diamond}\to \Ct$ and let $m = \delta(\lambda^M)\in \bfZ$. We assume that $\eta\neq 0$ and $m>0$. \footnote{The condition $m \neq 0$ is to ensure that the fixed components are algebraic varieties and is crucial in our approach. The condition $\eta \neq 0$ is not essential but the case $\eta= 0$ cannot be treated uniformly. For this reason we exclude it.}
One considers the following diagram: \\
\begin{minipage}[m]{.5\linewidth}
	\[
		\begin{tikzcd}[column sep = 20pt, row sep=15pt]
			\dot\frakg^{\lambdatq}_{\aff}\arrow{d}{a}\arrow[hookrightarrow]{r}{\dot\rho} & \dot\frakg_{\aff} \arrow{d}{a}\\
			\frakg^{\lambdatq}_{\aff}\arrow[hookrightarrow]{r}{\rho} &\frakg_{\aff}\\
		\end{tikzcd}
	\]
\end{minipage}
\begin{minipage}[m]{.5\linewidth}
	$\dot\frakg_{\aff} =  G_{\aff}\times^P \pi^{-1}(\rmO)\xrightarrow{a} \frakg_{\aff}$ : affine partial Springer resolution  \\
	$\frakg^{\lambdatq}_{\aff}, \dot\frakg^{\lambdatq}_{\aff}$ : $\lambda_q$-fixed points \\
	$\rho, \dot\rho$: inclusions of $\lambda_q$-fixed points \\
\end{minipage}

Finally, one would hope to apply the equivariant localisation to $\Ext^*_{ G_{\aff,\diamond,q}}(p_1^*\dot\scrC, p_2^!\dot\scrC)$ to show that
\[
	\Ext^*_{ G_{\aff,\diamond,q}}(a_*\dot\scrC, a_*\dot\scrC)_{\lambda_{\diamond,q}}\cong \Ext^*_{ G^{\lambdatq}_{\aff,\diamond,q}}\left(a_*\dot\rho^*\dot\scrC, a_*\dot\rho^*\dot\scrC\right)_{0}
\]
and to obtain a injective ring homomorphism with dense image:
\begin{equation}\begin{aligned}\label{equa:locDAHA}
		\Phi: \bfH_{\xi}\to \Ext^*_{ G^{\lambdatq}_{\aff,\diamond,q}}\left(a_*\dot\rho^*\dot\scrC, a_*\dot\rho^*\dot\scrC\right)_{0}.
\end{aligned}\end{equation}
This would yield an analogue of~\autoref{theo:AHA} for the dDAHA $\bfH_{\xi}$. \par

However, we will not construct the homomorphism $\Phi$ in~\eqref{equa:locDAHA} from an isomorphism $\bfH_{\xi}\cong \Ext^*_{ G_{\aff,\diamond,q}}\left(a_*\dot\scrC, a_*\dot\scrC\right)$ by means of localisation. This is because there is no straightforward implementation of $G_{\aff,\diamond,q}$-equivariant cohomology for the reason that $G_{\aff}$ is not an algebraic group. Instead, we shall work directly with the varieties of $\lambdatq$-fixed points $\frakg^{\lambdatq}_{\aff}$ and $\dot\frakg^{\lambdatq}_{\aff}$. Under the assumptions that we made about the cocharacter $\lambdatq$, these varieties are disjoint unions of algebraic varieties of finite type, which can be described in terms of cyclic grading on $\frakg$ and \textbf{spirals} and \textbf{splittings} introduced by Lusztig--Yun in~\cite{LYI}. Moreover, the connected components of $\dot\frakg^{\lambdatq}_{\aff}$ have a description similar to the varieties $\calT^{\nu}$ appearing in~\autoref{subsec:compsfixes}. 

\subsection{\texorpdfstring{$\bfZ$}{Z}-graded loop Lie algebras versus \texorpdfstring{$\bfZ/m$}{Z/m}-graded Lie algebras}

We briefly discuss here how to reduce the problem of the study of nilpotent cone of a $\bfZ$-graded affine Lie algebra to the one of a $\bfZ/m$-graded simple Lie algebra for $m = \delta(\lambda^M)\in \bfZ_{> 0}$ as above.  \par
Choose a $\Ct$-stable maximal torus $T\subset M$ and denote by $R = R(G, T)$ the root system. Put $T_{\diamond} = T\times \Ct$.
The affine Weyl group $W$ of $G_{\aff}$ is given by $W = N_{G_{\aff,\diamond}}(T_{\diamond}) / T_{\diamond}$. It acts linearly on the cocharacter lattice $\bfX_*(T_{\diamond})$.
Write $\lambdatq = (\lambda_0, m, \eta/2) \in \bfX_*(T\times \Ct\times \Cq)$ and denote $\lambdat = (\lambda_0, m)$. Then the weight spaces of $\lambda_{\diamond}$ yields a $\bfZ$-grading on $\frakg_{\aff}$
\[
	\frakg_{\aff} = \bigoplus_{n\in \bfZ}\frakg_{\aff,n}\;,\quad \frakg_{\aff,n} := \prescript{\lambda_{\diamond}}{n}\frakg_{\aff} = \bigoplus_{k\in \bfZ}\left\{ z\otimes \varpi^k\in \frakg\otimes \varpi^k\;;\; \lambda_0(t)z = t^{n-km}z,\quad \forall t\in \bfC^\times\right\}.
\]
On the other hand, the restriction $\theta := \lambda_0\mid_{\mu_m}:\mu_m\to T$ yields a $\bfZ/m$-grading
\[
	\frakg =\bigoplus_{\ubar n\in \bfZ/m}\frakg_{\ubar n}\;,\quad \frakg_{\ubar n} := \prescript{\theta}{\ubar n}\frakg = \left\{ z\in \frakg\;;\; \lambda_0(\zeta)z = \zeta^n z,\quad \forall \zeta\in \mu_m \right\}.
\]
Then we have an isomorphism on graded pieces induced by the evaluation $\varpi\mapsto 1$:
\begin{equation}\begin{aligned}\label{equa:lacets-cyclique}
	\frakg_{\aff,n}\cong \frakg_{\ubar n},\quad \forall n\in \bfZ.
\end{aligned}\end{equation}
	Similarly, we can deduce that the evaluation $\varpi\mapsto 1$ yields
\begin{equation}\begin{aligned}\label{equa:lacets-cyclique-groupe}
	G_{\aff,0}\cong G_{\ubar 0},\quad \text{where $G_{\aff, 0}:= G_{\aff}^{\lambda_{\diamond}}$ and $G_{\ubar 0} := G^{\theta}$.}
\end{aligned}\end{equation}
Moreover for each $n\in \bfZ$, the adjoint action of $G_{\aff,0}$ on $\frakg_{\aff,n}$ is isomorphic to the adjoint action of $\Gz$ on $\frakg_{\ubar n}$. %We will suppress the index $\theta$ from the notations. 

\subsection{Parahorics and spirals}
Let $P\subset  G_{\aff}$ be a $\Ct$-stable parahoric subgroup containing $T$ and denote $\frakp = \Lie P$. We can find a cocharacter $\mu_{\diamond} = (\mu, m')\in \bfX_*(T)\oplus\bfZ_{> 0}\subset \bfX_*(\Tt)$ such that $\frakp = \prescript{\mu_{\diamond}}{\ge 0}\frakg_{\aff}$. Put $\frakp_n = \frakp\cap  \frakg_{\aff,n}$ and $\frakl_n = \frakl\cap \frakg_{\aff, n}$ for $n\in \bfZ$ and put $P_0 = P^{\lambda_{\diamond}}$ and $L_0 = L^{\lambda_{\diamond}}$. Via the isomorphism~\eqref{equa:lacets-cyclique}, we can describe $\frakp_n$ alternatively as subspace of $\frakg_{\ubar n}$:
\begin{equation*}\begin{aligned}
	\frakp_n \cong \prescript{m'\lambda_0 - m\mu}{\ge m'n}\frakg_{\ubar n}.
\end{aligned}\end{equation*}
This description gives rise to the notion of a \emph{spiral}, introduced in~\cite{LYI} as the $\bfZ/m$-graded counterpart of a parahoric subalgebra. The cocharacter $\mu_{\diamond}$ also gives rise to a $\Ct$-stable Levi factor of $P$ is denoted by $L = G^{\mu_{\diamond}}_{\aff}$. The Lie algebra $\frakl = \Lie L$ can be described alternatively by
\begin{equation*}\begin{aligned}
	\frakl_n \cong \prescript{m'\lambda_0 - m\mu}{m'n}\frakg_{\ubar n}.
\end{aligned}\end{equation*}
Similarly, one can regard $P_0, L_0$ and $L$ as algebraic subgroups of $G$ such that $\Lie P_0 = \frakp_0$, $\Lie L_0 = \frakl_0$ and $\Lie L = \frakl$. The notion of spirals will be reviewed in~\autoref{sec:grad}.

\subsection{Fixed-point components of \texorpdfstring{$\dot\frakg_{\aff}$}{gaff} and \texorpdfstring{$\ddot\frakg_{\aff}$}{dot gaff}}\label{subsec:fixedcomp}

We suppose that there is a $\Ct$-stable parahoric subgroup $P\subset G_{\aff}$ with a $\Ct$-stable Levi factor $M\subset P$ which contains $T$. Suppose also that there is a nilpotent orbit $\rmO\subset \frakm^{\nil}$ such that $\rmO \cap \frakm_{\eta}\neq \emptyset$ which carries an $M$-equivariant cuspidal local system $\scrC$. As before, the Lie algebra $\frakm = \Lie M$ is $\bfZ$-graded by the weights of $\lambda_{\diamond}$. Such a datum $\xi = (M, \frakm_*, \rmO, \scrC)$ is called an {\it admissible system} on $\frakg_{\ubar \eta}$, see~\autoref{subsec:admsys}. Given an admissible system $\xi = (M, \frakm_*, \rmO, \scrC)$ on $\frakg_{\ubar \eta}$, just as in the case of reductive group, one can define a relative affine root system (\autoref{subsec:affroot}) and a relative affine Weyl group $W_{\xi}$. \par

%Fix $\eta\in \bfZ\setminus\left\{ 0 \right\}$. We assume that $\frakm_{\eta} \cap \rmO\neq \emptyset$. Under this hypothesis, the element $\lambda_{\diamond, q} = \iota(\lambda_{\diamond}, \eta/2)\in \bfX_*(\Tt\times \Cq)$ lies in the image of the embedding 
%\[
	%\bfX_*(Z_{\Mt}\times \Cq)\hookrightarrow \bfX_*(\Tt\times \Cq).
%\]
 %We may thus regard $\lambda_{\diamond,q}\in \frakz_{M,\diamond,q}$. \par
%Recall that $\lambdatq = (\lambdat,\eta/2)\in \bfX_*(Z_{M_\diamond,q})_{\bfQ}$.
 The sets $\dot\frakg_{\aff} =  G\times^P \pi^{-1}(\rmO)\to \frakg_{\aff}$ and $\dot\frakg_{\aff}$ have a $G_{\aff,\diamond,q}$-action induced from the actions on $G_{\aff}$ and $\frakg_{\aff}$. As in the case of finite reductive groups~\autoref{subsec:compsfixes}, the $\lambdatq$-fixed points of this action in $\dot\frakg_{\aff}$ can be described as a disjoint union of connected components
\begin{equation*}\begin{aligned}%\label{equa:springer-decomp-bis}
	\dot\frakg^{\lambda_{\diamond,q}}_{\aff} \cong \coprod_{\nu\in W_{\xi, \lambda}\backslash \Xi} \calT^\nu\to \frakg_{\aff}^{\lambda_{\diamond,q}} \cong \frakg_{\ubar \eta},
\end{aligned}\end{equation*}
where $\Xi$ is the affine analogue of the set introduced in~\autoref{subsec:coxZ}, which carries a $W_{\xi}$-action and $W_{\xi,\lambda} = \Stab_{W_{\xi}}(\lambda^M)$. We will define the set $\Xi$ in~\autoref{subsec:cusp} and the variety $\calT_{\eta}$ in~\autoref{subsec:steinberg}. Put $\ubar\Xi = W_{\xi,\lambda}\backslash\Xi$. \par

Since we have $\ddot\frakg^{\lambda_{\diamond,q}}_{\aff} = \dot\frakg^{\lambda_{\diamond,q}}_{\aff}\times_{\frakg_{\aff}} \dot\frakg^{\lambda_{\diamond,q}}_{\aff}$, there is a similar description for the Steinberg variety :
\begin{equation}\begin{aligned}\label{equa:steinberg-decomp}
	\ddot\frakg^{\lambda_{\diamond,q}}_{\aff} \cong \coprod_{\nu,\nu'\in \ubar\Xi} \calZ^{\nu, \nu'},\quad \calZ^{\nu, \nu'} = \calT^{\nu}\times_{\frakg_{\ubar \eta}}\calT^{\nu'}.
\end{aligned}\end{equation}
In consequence, the fixed-point sets $\ddot\frakg_{\aff}^{\lambda_{\diamond,q}}$, $\ddot\frakg_{\aff}^{\lambda_{\diamond,q}}$ and $\frakg_{\aff}^{\lambda_{\diamond,q}}$ are disjoint unions of $\Gztq$-equivariant algebraic varieties and the formalism of six operations and perverse sheaves is at our disposition. 

\subsection{Outline of the proof}
We can now explain the strategy of the construction of the homomorphism 
\begin{equation*}\begin{aligned}
	\Phi: \bfH_{\xi}\to \prod_{\nu'\in \ubar\Xi}\bigoplus_{\nu\in \ubar \Xi} \Ext^*_{\Gztq}\left( \bfI^{\nu'}, \bfI^{\nu} \right)_{0}.
\end{aligned}\end{equation*}
\begin{enumerate}
	\item[Step 1]
		We recall in~\autoref{sec:grad} the basic constructions on a $\bfZ/m$-graded Lie algebra introduced in~\cite{LYI}, such as the spiral induction functor $\Ind^{\frakg_{\ubar\eta}}_{\frakp_{\eta}}$ and admissible systems on the nilpotent cone. It is a $\bfZ/m$-graded analogue of~\autoref{subsec:Zgr}. 
	\item[Step 2]
		The $\bfZ/m$-counterpart of~\autoref{subsec:coxZ} is introduced in~\autoref{sec:affineCox}. We recall there the affine analogue of the facet decomposition $\frakF$ of an affine space $\bbA$ and the relative affine hyperplane $\bbE^M$, introduced in~\cite{LYIII}. We can identify $\bbE^M$ with the space $\bfX_*(Z_{M_{\diamond}})_{\bfQ}$ of fractional cocharacters. There is a bijection 
		\begin{equation*}\begin{aligned}
				\Xi = \left\{ \text{$\bbE^M$-chamber} \right\} &\xlongrightarrow{\sim} \left\{\frakp_* \text{ spiral}\;;\; \frakm_* \subset \frakp_*~\text{ is a splitting} \right\} \\
			\nu &\mapsto \frakp^{\nu}_*
		\end{aligned}\end{equation*}
		Then, we define the relative affine root system $(\bbE^M, R_{\xi})$ as well as its Weyl group $W_{\xi}$. Each element $\nu\in \Xi$ yields a base $\Delta^{\nu}_{\xi}\subset R_{\xi}$. We introduce also the canonical based relative root system $(\bbE_{\xi}, \Delta_{\xi})$ and the canonical relative Weyl group $\calW_{\xi}$. These two groups $W_{\xi}$ and $\calW_{\xi}$ act simply transitively on $\Xi$ and their action commute. We define also the linearisation $\bbE_{\xi,\diamond}$ of $\bbE_{\xi}$. 
	\item[Step 4]
		We put $\ubar \Xi = W_{\xi, \lambda}\backslash \Xi$. There is a bijection between $\ubar \Xi$ and the set of connected components of $\dot \frakg_{\aff}^{\lambda_{\diamond,q}}$.
		Following~\cite{LYIII}, for each $\nu\in \Xi$, consider the corresponding spiral $\frakp^{\nu}_*$ and the spiral induction $\bfI^{\nu} = \Ind^{\frakg_{\ubar \eta}}_{\frakp^{\nu}_{\eta}}\scrC\in \Db_{\Gztq}\left( \frakg_{\ubar \eta} \right)$. We {\itshape realise} the complex $a_*\dot\rho^*\dot\scrC$ as the infinite sum $\bigoplus_{\ubar \nu\in \ubar \Xi}\bfI^{\nu}$. Put 
		\[
			\ha\calH = \prod_{\ubar \nu'\in \ubar \Xi}\bigoplus_{\ubar \nu\in \ubar \Xi}\Ext^*_\Gztq(\bfI^{\nu'}, \bfI^{\nu})_{0}.
		\]
	\item[Step 5]
		Let $J\subsetneq \Delta_{\xi}$ and consider the parabolic subalgebra $\bfH_{\xi, J} \subset \bfH_{\xi}$. For each $\sigma\in \ubar \Xi / \calW_{\xi,J}$, there is an extension algebra
		\[
			\ha\calH_{\sigma} = \bigoplus_{\ubar\nu,\ubar\nu'\in \ubar\Xi^{\sigma}}\Ext^*_{H_{0,\diamond,q}}(\Ind^{\frakh_{\eta}}_{\frakq^{\nu'}_{\eta}}\scrC, \Ind^{\frakh_{\eta}}_{\frakq^{\nu}_{\eta}}\scrC)_{0},
		\]
		where $\frakh_* = \frakl^{\sigma}_*$ is a $\bfZ$-graded pseudo-Levi subalgebra of $\frakg$ and $\frakq^{\nu}_* = \frakh_*\cap \frakp^{\nu}_*\subset \frakh_*$ is the $\bfZ$-graded parabolic subalgebra corresponding to $\nu$.  The construction of~\cite{lusztig88}~\cite{lusztig95}~\cite{EM} yields a homomorphism $\Phi_{\sigma}: \bfH_{\xi,J}\to \ha\calH_\sigma$. Taking product over $\sigma$, we obtain
		\[
			\Phi_{J} = \left( \Phi_{\sigma} \right)_{\ubar\sigma\in \ubar\Xi / \calW_{\xi,J}}: \bfH_{\xi,J}\to \ha\calH_J = \prod_{\ubar\sigma\in \ubar\Xi / \calW_{\xi,J}}\ha\calH_{\sigma}
		\]
		The functoriality of $\Ind^{\frakg_{\ubar\eta}}_{\frakp^{\sigma}_{\eta}}$ yields an homomorphism $\ha\calH_J\to \ha\calH$, which is injective.
	\item[Step 6]
		For $K\subset J\subsetneq \Delta_{\xi}$, the compatibility result~\autoref{prop:compatibleZ} implies that there is a commutative square:
		\[
			\begin{tikzcd}[row sep=5pt]
				\bfH_{\xi,K}  \arrow[hookrightarrow]{d} \arrow{r}{\Phi_K} & \ha\calH_K  \arrow{d} \\
				\bfH_{\xi,J}   \arrow{r}{\Phi_J} & \ha\calH_J  \\
			\end{tikzcd}.
		\]
		This implies that we can take the direct limit over the partially ordered set of subsets $\left(\{J\subsetneq \Delta_{\xi}\}, \subseteq\right)$ and obtain
		\begin{equation*}\begin{aligned}
			\Phi = \varinjlim_{J}\Phi_J: \varinjlim_{J}\bfH_{\xi, J} \to \varinjlim_{J}\ha\calH_{J}\to \ha\calH.
		\end{aligned}\end{equation*}
		Using the fact that $\varinjlim_{J}\bfH_{\xi,J} = \bfH_{\xi}$, we obtain the map $\Phi$. 
\end{enumerate}
In the case where $J = \emptyset$, we have the smallest parabolic subalgebra $\bfH_{\xi, \emptyset}\subset \bfH_{\xi}$, which is equal to a polynomial ring $\bfS_{\xi} = \bfk[\bbE_{\xi,\diamond}]\otimes \bfk[u]$.
The second objective~\autoref{theo:density} is to show that the image of $\Phi$ is dense in $\ha\calH$ when the latter is equipped with a suitable topology. There are two essential ingredients:
\begin{enumerate}
	\item
		There is an affine version of the correspondence $\ubar\Xi\longleftrightarrow \left\{ \lambda_{\nu} \right\}$. This will allow to show that the restriction of $\Phi$ to the polynomial part $\bfS_{\xi}$ contains all the idempotents $(\bfe_{\nu})_{\nu\in \ubar \Xi}$ in~\autoref{subsec:compl}. It is done in~\autoref{prop:dense}.
	\item
		An analysis of the Steinberg varieties $\calZ^{\nu,\nu'}:= \calT^{\nu}\times_{\frakg_{\ubar \eta}} \calT^{\nu'}$ by stratification is done in~\autoref{sec:geomZ}. 
	\item
		A filtration on $\ha\calH$ by the Bruhat order is defined in~\autoref{subsec:filtration}. The associated quotient of this filtration can be studied using the stratification on the Steinberg varieties. We deduce a surjectivity result (\autoref{lemm:isostr}) which is crucial in the proof of the density theorem (\autoref{theo:density}).
	\item
		For each $J$, the map $\ha\calH_J\to \ha\calH$ is studied determined in~\autoref{prop:imagepsi}. 
\end{enumerate}

\begin{rema}
	The module category of $\bfH_{\xi}$ in which we are interested is, by convention, the category $\calO(\bfH_{\xi,\delta=1})$ consisting of finitely generated $\bfH_{\xi}$-modules on which the polynomial part $\bfS_{\xi}$ acts locally finitely and the imaginary root $\delta\in \bbE^*_{\xi,\diamond}$ acts by $1$. However, we have $\delta(\lambda^M) = \delta(\lambdat) = m$. For this reason we need to renormalise the parameters by setting $\bfx^M = \lambda^M/m$, $\bfx = \lambda_{\diamond}/m$, $\bfx_{\nu} = \lambda_{\nu} / m$. We will consider the block $\calO_{\bfx_{\nu}, \eta / 2m }(\bfH_{\xi,\delta=1})$ consisting of modules $\scrM\in \calO(\bfH_{\xi,\delta=1})$ on which $\bfS_{\xi} / (\delta - 1)$ acts with eigenvalues in the orbit $W_{\xi}\bfx_{\nu}\subset \bbE_{\xi}$, see ~\autoref{subsec:OH}. 
\end{rema}

\section{\texorpdfstring{$\bfZ/m$}{Z/m}-grading, spirals and splittings} \label{sec:grad}

We recall in this section some of the basic constructions over $\bfZ/m$-graded Lie algebras introduced in~\cite{LYI} and~\cite{LYIII}. 
\subsection{\texorpdfstring{$\bfZ/m$}{Z/m}-grading on \texorpdfstring{$G$}{G}} \label{subsec:grad}
Let $G$ \index{G@$G$} be a connected simply connected simple algebraic group over $\bfC$. The Lie algebra is denoted $\frakg = \Lie G$. \index{g@$\frakg$}\par
We fix a positive integer $m\in \bfZ_{> 0}$. For any integer $k\in \bfZ$, we denote $\ubar k = k \mod m\in \bfZ / m$. Let
\begin{equation*}\begin{aligned}
	\frakg = \bigoplus_{\ubar i\in \bfZ / m}\frakg_{\ubar i}
\end{aligned}\end{equation*}
\index{g@$\frakg_{\ubar i}$}
be a $\bfZ/m$-grading on $\frakg$ such that $\left[ \frakg_{\ubar i}, \frakg_{\ubar j} \right]\subseteq \frakg_{\ubar i+\ubar j}$ for all $\ubar i, \ubar j\in \bfZ/m$. Let $\mu_m\subset \bfC^\times$\index{m@$\mu_m$} be the group of $m$-th roots of unity. We define a homomorphism $\theta: \mu_m\to \Aut(\frakg)$\index{theta@$\theta$} by setting 
\begin{equation*}\begin{aligned}
	\theta(\zeta)\mid_{\frakg_{\ubar j}} = \zeta^j, \quad \forall \zeta\in \mu_m\;\forall j\in \bfZ.
\end{aligned}\end{equation*}\index{m@$\mu_m$}
The simplicity and the simple connectedness of $G$ imply that $\Aut(G)\cong \Aut(\frakg)$, so we can write $\theta: \mu_m\to \Aut(G)$. The fixed-point subgroup $G_{\ubar 0} = G^{\theta}\subset G$\index{G@$G_{\ubar 0}$} is connected by the theorem of Steinberg~\cite[8.1]{steinberg68}.  \par

\par

We assume that the $\bfZ/m$-grading is inner~\footnote{This assumption is made in order to simplify the presentation below. The affine root system $R_{\aff}$ that we introduce in~\autoref{subsec:affroot} below, under this assumption, is {\itshape untwisted} (denoted $A^{(1)}_n, B^{(1)}_n, \ldots$ in Kac's notation). We indicate in~\autoref{sec:twisted} how to adjust the construction to the case where $\theta$ is not inner. } in the sense that $\theta(\mu_m)\subset G^{\ad}$, where $G^{\ad}$ is the adjoint group of $G$. Under this assumption, we may assume without loss of generality that $\theta$ admits a lifting $\mu_m\to G$ --- indeed, we may pre-compose $\theta$ with a finite cover $[m']:\mu_{m'm}\to \mu_m$, where $m' = \#Z_G$, so that the resulting homomorphism $\theta':\mu_{m'm}\to G^{\ad}$ lifts to $G$; under this change, the grading on $\frakg$ is multiplied by $m'$. There exists then a cocharacter $\lambda_0\in \bfX_*\left( G\right)$ \index{lambda@$\lambda_0$} such that
\begin{equation*}\begin{aligned}
	\frakg_{\ubar i} = \bigoplus_{\substack{k\in \bfZ \\ \ubar i = \ubar k}}\prescript{\lambda_0}{k}{\frakg}.
\end{aligned}\end{equation*}
We fix once and for all the choice of such a cocharacter $\lambda_0$.

Besides, we fix a maximal torus $T\subseteq G$ \index{T@$T$} which centralises $\lambda_0$ so that $\lambda_0\in \bfX_*\left( T \right)$. In particular, $T$ is contained in $G_{\ubar 0}$. Fix once and for all an integer $\eta\in \bfZ_{\neq 0}$. Let $\frakg_{\ubar \eta}^{\nil} = \frakg_{\ubar \eta}\cap \frakg^{\nil}$ be the closed subvariety of nilpotent elements. 

\subsection{Jacobson--Morosov theorem}\label{subsec:JM}
Recall the theorem of Jacobson--Morosov in the $\bfZ/m$-graded setting. Let $e\in \frakg_{\ubar \eta}^{\nil}$. According to~\cite[2.3]{LYI}, we can complete $e$ into an $\fraksl_2$-triple $\phi = \left( e, h, f \right)$ with $h\in \frakg_{\ubar 0}$ and $f\in \frakg_{-\ubar \eta}$. Consequently, there is a cocharacter of $\Gz$, the exponentiation of $h$, which we denote by $\exp(h): t\mapsto t^h$. Moreover, the set of such triple $\phi$ with a given $e$ forms a principal homogeneous space under the action of the unipotent part of the stabiliser of $e$ in $G_{\ubar 0}$.  Using the Jacobson--Morosov theorem, one can show that the $G_{\ubar 0}$-orbits in $\frakg^{\nil}_{\ubar \eta}$ are invariant under homothety. 

\subsection{Spirals, nilpotent radical and splittings}\label{subsec:spirals}
Let $\mu\in \bfX_*\left(G_{\ubar 0}\right)_{\bfQ}$ be a fractional cocharacter and let $\varepsilon \in \{1, -1\}$.\index{e@$\varepsilon$} We attach to it a $\bfZ$-graded Lie algebra $\prescript{\varepsilon}{}\frakp^{\mu}_* =\bigoplus_{n\in \bfZ}  \prescript{\varepsilon}{}\frakp^\mu_n$ where
\begin{equation*}\begin{aligned}
	\prescript{\varepsilon}{}\frakp^{\mu}_n = \bigoplus_{\substack{r\in \bfQ \\r \ge \varepsilon n}} \prescript{\mu}{r}\frakg_{\ubar n}.
\end{aligned}\end{equation*}
Such a $\bfZ$-graded Lie algebra $\prescript{\varepsilon}{}\frakp^{\mu}_*$ \index{p@$\frakp^{\mu}_*$} is called an  {\bf $\varepsilon$-spiral} of $\frakg$. We also define a $\bfZ$-graded Lie subalgebra of $\frakg$
\begin{equation*}\begin{aligned}
	\prescript{\varepsilon}{}\frakl^{\mu}_* = \bigoplus_{n\in \bfZ} \prescript{\varepsilon}{}\frakl^{\mu}_n, \quad \prescript{\varepsilon}{}\frakl^{\mu}_n = \prescript{\mu}{\varepsilon n}\frakg_{\ubar n}.
\end{aligned}\end{equation*}
Such a Lie subalgebra $\prescript{\varepsilon}{}\frakl^{\mu}_*$ \index{l@$\prescript{\varepsilon}{}\frakl^{\mu}_*$} is called a {\bf splitting} of the spiral $\prescript{\varepsilon}{}\frakp^{\mu}_*$. We let $\prescript{\varepsilon}{}\frakl^{\mu}$ be the same Lie algebra as $\pre{\varepsilon}\frakl^{\mu}_*$ which has the grading forgotten. \par
We define the $\bfZ$-graded Lie algebra $\prescript{\varepsilon}{}\fraku^{\mu}_*$ with
\begin{equation*}\begin{aligned}
	\prescript{\varepsilon}{}\fraku^{\mu}_* = \bigoplus_{n\in \bfZ} \prescript{\varepsilon}{}\fraku^{\mu}_n ,\quad \prescript{\varepsilon}{}\fraku^{\mu}_n = \bigoplus_{r > \varepsilon n} \prescript{\mu}{r}\frakg_{\ubar n}
\end{aligned}\end{equation*}
to\index{u@$\fraku^{\mu}_*$} be the {\bf nilpotent radical} of the spiral $\prescript{\varepsilon}{}\frakp^{\mu}_*$.  Then $\prescript{\varepsilon}{}\fraku^{\mu}_*$ forms a homogeneous ideal of the graded Lie algebra $\prescript{\varepsilon}{}\frakp^{\mu}_*$ and that for each $n\in \bfZ$ the subspace $\prescript{\varepsilon}{}\frakl^{\mu}_n\subseteq \prescript{\varepsilon}{}\frakp^{\mu}_n$ is mapped isomorphically onto the quotient $\prescript{\varepsilon}{}\frakp^{\mu}_n / \prescript{\varepsilon}{}\fraku^{\mu}_n$. Moreover, $\pre{\varepsilon}\frakl^{\mu}$ is the Lie algebra of a pseudo-Levi subgroup of $G$, denoted by $L^{\mu}$, see~\cite[2.2.5]{LYIII}. 

We will write $\frakp^{\mu}_*$, $\frakl^{\mu}_*$ and $\fraku^{\mu}_*$ instead of $\prescript{\varepsilon}{}\frakp^{\mu}_*$, $\prescript{\varepsilon}{}\frakl^{\mu}_*$ and $\prescript{\varepsilon}{}\fraku^{\mu}_*$ when $\varepsilon$ is clear from the context.

\subsection{Admissible systems}\label{subsec:admsys}
An {\bf admissible system} on $\frakg_{\ubar \eta}$, as defined in~\cite{LYI}, is a datum $\left(M, \frakm_*, \rmO, \scrC  \right)$, where $M\subset G$ is a subgroup whose Lie algebra is equipped with a $\bfZ$-grading $\frakm_* = \bigoplus_{n\in \bfZ} \frakm_n$ arising as splitting $\frakl^{\mu}_*$ of some spiral, $\rmO\subset \frakm^{\nil}$ is a $M$-orbit such that $\rmO\cap \frakm_{\eta} \neq \emptyset$ and $\scrC\in \Loc_{M}(\rmO)$ is a cuspidal local system in the sense of~\cite{lusztig84}. We denote $\rmO_{\eta} = \rmO \cap \frakm_{\eta}$ and $\scrC_{\eta} = \scrC\mid_{\rmO_{\eta}}$. It is known that $\scrC_{\eta}$ is an irreducible $M_0$-equivariant local system which is {\itshape clean}, see~\cite[\S 4]{lusztig95b}. \par 

An isomorphism of admissible systems $(M, \frakm_*, \rmO, \scrC)\cong (M', \frakm'_*, \rmO', \scrC')$ is a pair $(g, \varphi)$ of element $g\in \Gz$ such that $gMg^{-1} = M'$, $\Ad_g \frakm_n = \frakm'_n$, $\Ad_g\rmO = \rmO'$ and an isomorphism $\varphi : g^* \scrC'\cong \scrC$.  Let $\frakT\left(\frakg_{\ubar \eta}\right)$ denote the groupoid of admissible systems on $\frakg_{\ubar \eta}$ and let $\ubar\frakT\left(\frakg_{\ubar \eta}\right)$ denote the set of isomorphism classes its objects. 

\subsection{Extra torus actions}\label{subsec:extratori}

We will consider two one-dimensional tori $\Cq = \bfC^{\times}$\index{C@$\Cq$} and $\Ct = \bfC^{\times}$.\index{C@$\Ct$} We let $\Cq$ act on $\frakg_{\ubar \eta}$ by weight $-2$ and we let it acts trivially on $G_{\ubar 0}$. The torus $\Ct$ acts~\footnote{As it is said in the beginning of~\autoref{sec:strat}, the torus $\Ct$ acts morally by loop rotation. We will see in~\autoref{prop:Ppoids} why the $\Ct$-action is defined in this way.} on $\frakg_{\ubar \eta}$ by $(\eta - \lambda_0) / m$ and on $\Gz$ by $-\lambda_0 / m$. Then the product $\Gz\rtimes \Ct\times \Cq$ acts on $\frakg_{\ubar\eta}$ in a natural way. \par

Given any $\mu\in \bfX_*(G_{\ubar 0})_{\bfQ}$, consider the $\varepsilon$-spiral $\frakp^{\mu}_*$ and the splitting $\frakl^{\mu}_*$. We let $\Cq$ acts on $\frakp^{\mu}_*$ by weight $-2$ and trivially on the groups $L^{\mu}$ and $P^{\mu}_0$. For each $n\in \bfZ$, we let $\Ct$ act on $\frakp_n$ by $(n - \lambda_0) / m$ for each $n\in \bfZ$. This $\Ct$-action, when restricted to the Lie algebra $\frakl^{\mu}_*$, can be integrated to an $\Ct$-action on the group $L$. The product $L\rtimes \Ct\times \Cq$ acts on $\frakl^{\mu}$ ($\bfZ$-grading forgotten) in a natural way. Similarly, the group $L^{\mu}_0\rtimes \Ct\times \Cq$ (resp. $P^{\mu}_0\rtimes \Ct\times \Cq$) acts on the $\bfZ$-graded Lie algebra $\frakl^{\mu}_*$ (resp. $\frakp^{\mu}_*$). \par

We also want to consider the $\Ct$-equivariance of sheaves on Lie algebras. However, not every cuspidal local system admits an extra $\Ct$-equivariance. For this reason, let $\Ctm\to \Ct$\index{C@$\Ctm$} be the $2m$-fold cover. Hereafter, for any group $H$ on which $\Ct$ acts by automorphism, we denote\footnote{The use of the index $\diamond$ differs slightly from that in~\autoref{sec:strat}. We are obliged to take a finite cover of $\Ct$ so as to make the cuspidal local system equivariant. This technical point was omitted there.} $H_{\diamond} = H\rtimes \Ctm$ and for any group $K$, we denote $K_q = K\times \Cq$.\index{0@$\Gztq, \Gzt, M_q, \Mtq, \Mztq$} \par
\begin{prop}\label{prop:qtequivariance}
Given any admissible system $(M, \frakm_*, \rmO, \scrC)$ on $\frakg_{\ubar \eta}$, the cuspidal local system $\scrC$ can be enhanced to a $\Mtq$-equivariant local system on $\rmO$ in a unique way.
\end{prop}
\begin{proof}
	Pick $e\in \rmO_{\eta}$. Since we have $\Loc_{M}(\rmO)\cong \Rep \pi_0(Z_{M}(e))$, and $\Loc_{\Mtq}(\rmO)\cong \Rep \pi_0(Z_{\Mtq}(e))$ it suffices to show~\autoref{lemm:pi} below.
\end{proof}
\begin{lemm}\label{lemm:pi}
	Let $\mu\in \bfX_*(\Gz)$ and $(L^{\mu}, \frakl^{\mu}_*)$ the corresponding $\bfZ$-graded pseudo-Levi. Let $e\in \frakl^{\mu}_{\eta}$. Then the following inclusions
	\[
		Z_{L^{\mu}}(e)\subseteq Z_{L^{\mu}_{\diamond}}(e) \subseteq Z_{L^{\mu}_{\diamond, q}}(e)
	\]
	induce isomorphisms on the component groups
	\[
		\pi_0(Z_{L^{\mu}}(e)) \cong \pi_0(Z_{L^{\mu}_{\diamond} }(e))\cong \pi_0(Z_{L^{\mu}_{\diamond,q}}(e)).
	\]
\end{lemm}
\begin{proof}
	Complete $e$ into an $\fraksl_2$-triple $\phi = (e, h, f)$ with $h\in \frakl^{\mu}_{0}$ and $f\in \frakl^{\mu}_{-\eta}$, which is possible by the $\bfZ$-graded Jacobson--Morosov theorem~\cite[3.3]{lusztig95b}. Let $\varphi = \exp(h)\in \bfX_*(L_0)$. Since the map
	\begin{equation*}\begin{aligned}
		L^{\mu}\times \bfC^{\times}&\to L^{\mu}\rtimes \Ctm\\
		(g, \tau) &\mapsto (g \varphi(\tau^{-\eta}) \lambda_0(\tau^2), \tau)
	\end{aligned}\end{equation*}
	restricts to an isomorphism
	\begin{equation*}\begin{aligned}
		Z_{L^{\mu}}(\phi)\times\bfC^{\times}\cong Z_{L^{\mu}_{\diamond}}(\phi),
	\end{aligned}\end{equation*}
	we obtain isomorphisms on their component groups.
	\[
		\pi_0(Z_{L^{\mu}}(e))\cong \pi_0(Z_{L^{\mu}}(\phi))\cong\pi_0(Z_{L^{\mu}_{\diamond}}(\phi)) \cong\pi_0(Z_{L^{\mu}_{\diamond}}(e)),
	\]
	where the first (resp. the last) isomorphism is due to the fact that $Z_{L^{\mu}}(\phi)\subset Z_{L^{\mu}}(e)$ (resp. $Z_{L^{\mu}_{\diamond}}(\phi)\subset Z_{L^{\mu}_{\diamond}}(e)$) is a maximal reductive subgroup, see~\cite[2.1]{lusztig88}.  \par
	For the $\Cq$-equivariance, we make use of the group $Z^q_{L^{\mu}_{\diamond}}(\phi)$ introduced in~\eqref{equa:Mphi} and the isomorphism
	\[
		Z_{L^{\mu}_{\diamond}}(\phi)\times \bfC^{\times}\cong Z^q_{L^{\mu}_{\diamond}}(\phi)
	\]
	to show that
	\[
		\pi_0(Z_{L^{\mu}_{\diamond}}(e))\cong \pi_0(Z_{L^{\mu}_{\diamond}}(\phi))\cong\pi_0(Z^q_{L^{\mu}_{\diamond}}(\phi)) \cong\pi_0(Z_{L^{\mu}_{\diamond,q}}(e)).
	\]
\end{proof}
\begin{rema}
	As explained in~\autoref{sec:strat}, we think of $\frakg_{\ubar \eta}$ as the graded piece $\frakg_{\aff, \eta}$ of the loop group $\frakg_{\aff} = \frakg[ \varpi^{\pm 1} ]$. Under this identification, the action of $\Ct$ on $\frakg_{\ubar \eta}$ defined as above coincides with the loop rotation on $\frakg_{\aff, \eta}$.
\end{rema}

\subsection{Spiral induction}\label{subsec:indspirale}
With the datum $\left( \frakp_*, \frakl_*, \fraku_* \right) =\left( \prescript{\varepsilon}{}\frakp^{\mu}_*, \prescript{\varepsilon}{}\frakl^{\mu}_*, \prescript{\varepsilon}{}\fraku^{\mu}_* \right)$ of a spiral together with a splitting, we can define the functor of induction. Let $P_0 = \exp(\frakp_0)$ and $L_0 = \exp(\frakl_0)$.  \par

Fix a sign $\varepsilon \in \left\{ 1, -1 \right\}$. We consider the following diagram
\begin{equation*}\begin{aligned}%\label{equa:suite-indspirale}
	\frakg_{\ubar \eta} \xleftarrow{\alpha} G_{\ubar 0}\times^{P_0}\frakp_{\eta}\xleftarrow{\beta} \frakp_{\eta}\xrightarrow{\gamma}  \frakl_{\eta},
\end{aligned}\end{equation*}
where
\begin{equation*}\begin{aligned}
				\quad \alpha(g, x) = \Ad(g) x; \quad \beta(x) = (e, x); \quad \gamma(x) = x\mod{\fraku_{\eta}}.
\end{aligned}\end{equation*}
Then, the morphism $\alpha$ is proper whereas $\gamma$ is a trivial vector bundle. The above sequence gives rise to a sequence of quotient stacks:
\begin{equation*}\begin{aligned}
	[\frakg_{\ubar \eta}/ \Gztq] \xleftarrow{a} [\Gz\times^{P_{0}}\frakp_{\eta} / \Gztq]\xleftarrow{b}  [\frakp_{\eta}/\Pztq]\xrightarrow{c} [\frakl_{\eta}/\Lztq],
\end{aligned}\end{equation*}
where $b$ is an isomorphism and $a$ is proper. The spiral induction is defined as
\begin{equation}\begin{aligned}\label{equa:nil}
	\Ind^{\frakg_{\ubar \eta}}_{\frakp_{\eta}} &= (ab)_* c^*: \Db_{\Lztq}\left( \frakl_{\eta} \right)\to \Db_{\Gztq}\left( \frakg_{\ubar \eta} \right). \\ 
\end{aligned}\end{equation}\index{Ind@$\Ind^{\frakg_{\ubar \eta}}_{\frakp_{\eta}}$}
In fact, for the sign $\varepsilon = \eta / |\eta|$, it induce functors on the nilpotent cones~\cite[7.1(a)]{LYI}:
\begin{equation*}\begin{aligned}
	\Ind^{\frakg_{\ubar \eta}}_{\frakp_{\eta}}:\Db_{\Lztq}\left( \frakl^{\nil}_{\eta} \right)\to \Db_{\Gztq}\left( \frakg^{\nil}_{\ubar \eta} \right).
\end{aligned}\end{equation*}

The spiral induction functor satisfies the following transitivity property with the parabolic induction:
\begin{equation}\begin{aligned}\label{equa:transindsp}
	\Ind^{\frakg_{\ubar \eta}}_{\frakp_{\eta}} = \Ind^{\frakg_{\ubar \eta}}_{\frakq_{\eta}}\circ\Ind^{\frakh_{\eta}}_{\ba\frakp_{\eta}}%,\quad \Res^{\frakg_{\ubar \eta}}_{\frakp_{\eta}} = \Res^{\frakh_{\eta}}_{\ba\frakp_{\eta}}\circ\Res^{\frakg_{\ubar \eta}}_{\frakq_{\eta}}
\end{aligned}\end{equation}
whenever there is a spiral $\frakq_*$ with splitting $\frakh_*$ such that $\frakp_*\subset \frakq_*$ and $\frakl_*\subset \frakh_*$. Here $\ba\frakp_* = \frakp_*\cap \frakh_*$ is a parabolic subalgebra of $\frakh_*$.

\subsection{Block decomposition}\label{subsec:block}
Suppose that $\varepsilon = \eta / |\eta|$. Let $\xi = (M, \frakm_*, \rmO, \scrC)\in \ubar\frakT\left( \frakg_{\ubar \eta} \right)$ be an admissible system on $\frakg_{\ubar \eta}$. By the cleanness of cuspidal local systems~\cite[\S 4]{lusztig95b}, the $*$-extension and the $!$-extension of $\scrC_{\eta}$ are isomorphic to (a shift of) the intersection complex $\IC(\scrC_{\eta})$. Therefore, we may identify $\scrC_{\eta}$ with its direct image on $\frakm_{\eta}$. \par

Define $\Db_{\Gz}\left( \frakg_{\ubar \eta}^{\nil} \right)_{\xi}$ to be the thick triangulated subcategory of $\Db_{\Gz}\left( \frakg_{\ubar \eta}^{\nil} \right)$ generated by the constituents of the perverse cohomology of $\Ind^{\frakg_{\ubar \eta}}_{\frakp_\eta}\scrC$ for all $\varepsilon$-spiral $\frakp_*$ which has $\frakm_*$ as splitting.  We define $\Perv_{\Gz}\left( \frakg_{\ubar \eta}^{\nil} \right)_{\xi}$ to be the intersection of $\Db_{\Gz}\left( \frakg_{\ubar \eta}^{\nil} \right)_{\xi}$ with $\Perv_{\Gz}\left( \frakg_{\ubar \eta}^{\nil} \right)$. We call these subcategories and subsets the {\bf blocks} of $\xi$. \par

According to~\cite[0.6]{LYI}, there are orthogonal decompositions
\begin{equation*}\begin{aligned}\label{equa:dcpGz}
	\Perv_{\Gz}\left( \frakg^{\nil}_{\ubar \eta} \right) = \bigoplus_{\xi\in \ubar\frakT\left( \frakg_{\ubar \eta} \right)} \Perv_{\Gz}\left( \frakg^{\nil}_{\ubar \eta} \right)_{\xi}, \qquad \Db_{\Gz}\left( \frakg^{\nil}_{\ubar \eta} \right) = \bigoplus_{\xi\in \ubar\frakT\left( \frakg_{\ubar \eta} \right)} \Db_{\Gz}\left( \frakg^{\nil}_{\ubar \eta} \right)_{\xi}.
\end{aligned}\end{equation*}
Using the Hochschild--Serre spectral sequence, we can deduce the same decompositions with the extra torus actions:
\begin{equation}\begin{aligned}\label{equa:dcpGztq}
	\Perv_{\Gztq}\left( \frakg^{\nil}_{\ubar \eta} \right) = \bigoplus_{\xi\in \ubar\frakT\left( \frakg_{\ubar \eta} \right)} \Perv_{\Gztq}\left( \frakg^{\nil}_{\ubar \eta} \right)_{\xi}, \qquad \Db_{\Gztq}\left( \frakg^{\nil}_{\ubar \eta} \right) = \bigoplus_{\xi\in \ubar\frakT\left( \frakg_{\ubar \eta} \right)} \Db_{\Gztq}\left( \frakg^{\nil}_{\ubar \eta} \right)_{\xi}.
\end{aligned}\end{equation}

\section{Relative affine root system and affine Weyl group}\label{sec:affineCox}
This section is a reminder of the affine alcove complex defined in \cite{LYIII} and its relations with spirals and splittings. We retain the setting of~\autoref{sec:grad}. \par

\subsection{Affine root hyperplane arrangement}\label{subsec:affroot} 
From now on, we fix a maximal torus $T\subset G^{\lambda_0}$ so that $\lambda_0\in \bfX_*(T)$. Let $\bbA_{\diamond}$\index{A@$\bbA, \bbA_{\diamond}$} denote the vector space $\bfX_*(T_{\diamond})_\bfQ = \bfX_*(T_{\diamond})\otimes \bfQ$. Since 
\[
	\bfX_*(T_{\diamond}) = \bfX_*(T)\times\bfX_*(\Ctm) = \bfX_*(T)\times 2m\bfX_*(\Ct), 
\]
we let $(0,2m)\in \bbA_{\diamond}$ denote the inclusion of $\Ctm$ in $T_{\diamond}$ and let $\delta = (0,1)\in \bbA_{\diamond}^*$ \index{d@$\delta$} denote the character given by the obvious map $T_{\diamond}\to \Ct$ so that $\delta(\mu, 1) = 1$ for all $\mu\in \bfX_*(T)$. Notice that the defining character of $\Ctm$ is $(1/2m)\delta$. Let $\bbA = \delta^{-1}(1)\subset \bbA_{\diamond}$, which is an affine subspace whose tangent space is $\bfX_*(T)_\bfQ$.  \\
The root system of $G$ is a subset $R(G,T)\subset \bfX^*(T)$. We define $R_{\aff} = R(G, T)\oplus\bfZ \delta\subset \bbA^*_{\diamond}$ to be the set of (real) affine roots.\index{R@$R_{\aff}$} \par
Each affine root $\alpha \in R_{\aff}$ restricts to a non-constant affine function on $\bbA$, whose zero locus is a hyperplane in $\bbA$, denoted by $H_{\alpha}$ and called {\bf root hyperplane}. Let $\frakH = \left\{ H_{\alpha}\;;\; \alpha \in R_{\aff} \right\}$\index{H@$\frakH$} be the collection of root hyperplanes. These affine hyperplanes yield a stratification of $\bbA$ into {\bf facets} (i.e. subsets of $\bbA$ determined by a finite number of equations $\alpha = 0$ or $\alpha > 0$ with $\alpha\in R_{\aff}$). Let $\frakF$\index{F@$\frakF$} denote the collection of facets. The facets of maximal dimension are called {\bf alcoves}. The set of alcoves is denoted by $\frakA$\index{A@$\frakA$}. \par

The Killing form on $\frakg$ yields a Euclidean space structure on $\bbA$. The {\bf affine Weyl group} (with respect to $T$) is the group $W = N_G(T)/T\ltimes \bfX_*(T)$\index{W@$W$}, which acts on $\bbA$ by orthogonal reflections and translations. This action preserves the set $\frakH$ of root hyperplanes, inducing thus a $W$-action on $\frakF$, whose restriction to $\frakA\subseteq \frakF$ is simply transitive. Since $\bbA_{\diamond}$ is the linearisation of $\bbA$, the affine $W$-action on $\bbA$ extends in a unique way to a linear $W$-action on $\bbA_{\diamond}$. \par

Let $\frakE$ be the collection of affine subspaces of $\bbA$ which are non-empty intersection of a finite subset of $\frakH$. Elements of $\frakE$\index{E@$\frakE$} are called {\bf relevant subspaces}. \par

\subsection{Canonical affine Weyl group}\label{subsec:canweyl}

Given an alcove $\kappa\in \frakA$, let $\Delta^{\kappa}\subseteq R_{\aff}$\index{D@$\Delta^{\kappa}$} be the subset of affine simple roots $\alpha$ such that $\alpha > 0$ on $\kappa$ and $H_{\alpha} \cap \ba\kappa$ is a face of the simplex $\ba\kappa$. It is known that $\Delta^\kappa$ forms a base for the affine root system $(\bbA, R_{\aff})$. For each $\alpha\in R_{\aff}$, let $s_\alpha\in W$ denote the orthogonal reflection on $\bbA$ with respect to the root hyperplane $H_{\alpha}$. Then injective map $\Delta^{\kappa}\hookrightarrow W$ defined by $\alpha\mapsto s_\alpha$ yields a Coxeter group $\left(W, \Delta^{\kappa}\right)$. \par
Since $W$ acts simply transitively on $\frakA$, given any two alcoves $\kappa,\kappa'\in \frakA$, there is a unique $w\in W$ such that $\kappa' = w\kappa$. Then, $w$ yields an isomorphism of Coxeter groups
\begin{equation*}\begin{aligned}
	(W, \Delta^{\kappa}) &\cong ( W, \Delta^{\kappa'} ) \\
	y\mapsto wyw^{-1},&\quad  s_\alpha \mapsto w s_\alpha w^{-1} = s_{w \alpha}.
\end{aligned}\end{equation*}
Thus we can define 
\[
	\left(\calW, \Delta\right) = \varprojlim_{\kappa\in \frakA} \left( W, \Delta^\kappa \right)
\]
to be the {\bf canonical affine Weyl group}: for any $\kappa\in \frakA$, there is a canonical isomorphism $\left(\calW, \Delta\right) \cong \left(W, \Delta^{\kappa}\right)$\index{W@$\calW$}\index{D@$\Delta$}. \par 

For any subset $J\subseteq \Delta$, we denote by $\left(\calW_{J}, J\right)$ \index{W@$\calW_{J}$} the Coxeter sub-system of $\left(\calW, \Delta\right)$ generated by $J$. For any alcove $\kappa\in \frakA$, we denote by $\left(W_{J^{\kappa}}, J^{\kappa}\right)$ the Coxeter subgroup of $\left(W, \Delta^{\kappa}\right)$ which is the image of $\left( \calW_J, J \right)$ in $(W, \Delta^{\kappa})$ under the canonical isomorphism $\left(\calW, \Delta\right)\cong \left(W, \Delta^{\kappa}\right)$.  \par

We define a $\calW$-action on $\frakA$ as follows: for $w\in \calW$ and $\kappa\in \frakA$, put $w\kappa = (w^{\kappa})^{-1}\kappa$. This action is simply transitive and commutes with the $W$-action on $\frakA$.
\begin{rema}
	It is useful to keep in mind that $\calW = \Aut_W(\frakA)$ and that the $\calW$-action and the $W$-action on $\frakA$ are of different nature. The $W$-action on $\frakA$ is induced from a reflection action on $\bbA$, whereas the $\calW$-action on $\frakA$ does not extend continuously to an action on $\bbA$. Moreover, the $\calW$-action is ``local'' in the sense that whenever $s\in \calW$ is a simple reflection, $s\kappa$ and $\kappa$ are adjacent for every alcove $\kappa\in \frakA$. 
\end{rema}

\subsection{Facets and Coxeter subgroups}\label{subsec:facetcox}
We put a partial order $\le$ on the set of facets $\frakF$ by the inclusion of closure : $\sigma \le \tau$ if $\sigma\subseteq \ba\tau$. Given any $\kappa\in \frakA$ and $J\subsetneq \Delta$, we define a facet $\partial_J \kappa\in \frakF$ with $\partial_J \kappa \le \kappa$ by
\begin{equation*}
	\partial_{J}\kappa = \left\{y\in \bbA\;;\;\begin{aligned}& \alpha(y) = 0, & \forall \alpha\in J^{\kappa} \\ & \alpha(y)  > 0, & \forall \alpha\in \Delta^{\kappa}\setminus J^{\kappa}\end{aligned}\right\}. 
\end{equation*}
Then $W_{J^{\kappa}}$ coincides with the pointwise stabiliser $W_{\partial_J \kappa} = \Stab_W(\partial_J\kappa)$. 
\index{0@$\partial_{J}$}
\par

Given any facet $\sigma\in \frakF$, choose an alcove $\kappa\in \frakA$ such that $\sigma\le \kappa$. There is a unique subset $J\subsetneq \Delta$ such that $\partial_J \kappa = \sigma$ and moreover, the subset $J$ is independent of the choice of $\kappa\in \frakA$. The subset $J\subsetneq \Delta$ is called the {\bf type} of the facet $\sigma\in \frakF$. Let $\frakF_{J} = \partial_J(\frakA)\subseteq \frakF$\index{F@$\frakF_J$} denote the set of facets of type $J$. It follows that $\frakF = \bigsqcup_{J\subsetneq \Delta}\frakF_J$ and $\frakF_{\emptyset} = \frakA$. Moreover, each subset $\frakF_J\subset \frakF$ is stable by the action of $W$ and the map $\partial_J:\frakA\to \frakF_J$ induces a bijection of $W$-sets $\frakA / \calW_J \cong \frakF_J$.
\par

For any pair of strict subsets $J\subseteq K\subsetneq \Delta$, the boundary map $\partial_{K}: \frakA\to \frakF_{K}$ descends to a map $\partial^J_K:\frakF_J\to \frakF_K$ such that $\partial_K = \partial^J_K\circ\partial_J$. By abuse of notation, we will write $\partial_K = \partial^J_K$. 
\par

\subsection{Correspondence between spirals and facets} \label{subsec:spiralfacet}
We define $\bfx = (\lambda_0/m, 1)\in \bbA$.\index{x@$\bfx$} Let $\frakP_T$ be the set of spirals $\frakp_*$ such that $\frakp_* = \pre{\varepsilon}\frakp_*^{\mu}$ for some $\mu\in \bfX_*\left( T \right)_{\bfQ}$, \cfauto{subsec:spirals}. \par
Given a facet $\sigma\in \frakF$, we choose a point $y\in \sigma$. Since both $y$ and $\bfx$ lie in $\bbA$, the difference $\bfx - y$ is in $\bfX_*(T)_{\bfQ}$. Set $\mu_y = m\varepsilon\left( \bfx - y \right) \in \bfX_*\left( T \right)_{\bfQ}$. It gives rise to the spiral $\pre{\varepsilon}\frakp^{\mu_y}_*$, which does not depend on the choice of $y\in \sigma$, see \cite[3.4.4]{LYIII}. \par

We denote $\frakp^{\sigma}_* = \pre{\varepsilon}\frakp^{\mu_y}_*$\index{p@$\frakp^{\sigma}_*$} as well as $\fraku^{\sigma}_* = \pre{\varepsilon}\fraku^{\mu_y}_*$\index{u@$\fraku^{\sigma}_*$} and $\frakl^{\sigma}_* = \pre{\varepsilon}\frakl^{\mu_y}_*$\index{l@$\frakl^{\sigma}_*$} for any choice of $y\in \sigma$. We also denote $P^{\sigma}_0 = \exp\left( \frakp^{\sigma}_0 \right)$, $L^{\sigma} = \exp\left( \frakl^{\sigma} \right)$, $L^{\sigma}_0 = \exp\left( \frakl^{\sigma}_0 \right)$ and $U^{\sigma}_0 = \exp\left( \fraku^{\sigma}_0 \right)$; those are subgroups of $G$. \par

Recall that we have defined in~\autoref{subsec:extratori} a $\Ct$-action on $\frakp^{\sigma}_*$. The $T_{\diamond}$-weights appearing in $\frakp^{\sigma}_*$ can be characterised as follows: 
\begin{prop}\label{prop:Ppoids}
	An affine root $\alpha\in R_{\aff}$ appears in $\frakp^{\sigma}_*$ if and only if for some (equiv. every) $y\in \sigma$, we have
	\begin{itemize}
		\item[ ]
			$\alpha(y) \le 0$ in the case $\varepsilon = 1$;
		\item[ ]
			$\alpha(y) \ge 0$ in the case $\varepsilon = -1$.
	\end{itemize}
	In this case, the $\alpha$-weight space appears in $\frakp^{\sigma}_n$ for $n = m\,\alpha(\bfx)\in \bfZ$. \hfill\qedsymbol
\end{prop}
This proposition yields a combinatorial parametrisation of the set of $T$-stable $\varepsilon$-spirals.
\begin{coro}\label{coro:facet-spiral}
	The assignment $\sigma\mapsto \frakp^{\sigma}_*$ yields a bijection between $\frakF$ and $\frakP_T$. Moreover, $\sigma\le\tau$ holds if and only if $\frakp^{\tau}_*\subseteq\frakp^{\sigma}_*$.\hfill\qed
\end{coro}

\begin{rema}
	In other words, $\frakp^{\sigma}_*$ recovers the usual notion of parahoric subalgebra of the loop group $G(\!(\varpi^{-\varepsilon})\!)$ and $\frakp^{\sigma}_n$ is the $(n/m)$-weight space for the fractional cocharacter $\bfx = (\lambda_0 / m, 1)\in \bfX_*(T_{\diamond})_{\bfQ}$. The space $\bbA$ with the facet structure $\frakF$ is the apartment of the Bruhat--Tits building for $G(\!(\varpi^{-\varepsilon})\!)$ given by the parahoric subgroups containing $T$ and stabilised by the loop rotation $\Ct$.
\end{rema}

\subsection{Pseudo-Levi attached to relevant affine subspaces} \label{subsec:subsp-pslevi}
Let $\frakM^{\bfZ-\gr}_T$ be the set of quadruples $(M, M_0, \frakm, \frakm_*)$ with $M$ a pseudo-Levi subgroup of $G$ containing the maximal torus $T$ and $\frakm_*$ a $\bfZ$-grading on $\frakm = \Lie M$ which makes $\frakm$ a graded Lie algebra such that $T \subseteq M_0$, where $M_0 = \exp\left( \frakm_0 \right)$. There is bijective correspondence
\begin{equation*}\begin{aligned}
	\frakE\longleftrightarrow \frakM^{\bfZ-\gr}_T
\end{aligned}\end{equation*}
defined as follows: given any relevant subspace $\bbE\in\frakE$, choosing any facet $\sigma\in \frakF$ which spans $\bbE$ as affine subspace, we set 
\begin{equation*}\begin{aligned}
	\left(M^\bbE, M^\bbE_0,  \frakm^\bbE, \frakm^\bbE_*\right) = \left(L^{\sigma} , L^{\sigma}_0,\frakl^{\sigma}, \frakl^{\sigma}_* \right);
\end{aligned}\end{equation*}
then the splitting $\left(M^{\bbE}, M^{\bbE}_0, \frakm^{\bbE}, \frakm^{\bbE}_*\right)$ does not depend on the choice of $\sigma$, see \cite[3.4.7]{LYIII}. Similarly to the case of spiral, we can characterise the $T_{\diamond}$-weights appearing in $\frakm^{\bbE}_*$ as follows: 
\begin{prop}\label{prop:Mpoids}
	An affine root $\alpha\in R_{\aff}$ appears in $\frakm^{\bbE}_*$ if and only if $\alpha\mid_{\bbE} = 0$.  In this case, the $\alpha$-weight space appears in $\frakm^{\bbE}_n$ for $n = m\,\alpha(\bfx)$. \hfill\qedsymbol
\end{prop}

We can describe the root system of the reductive group $M^{\bbE}_{\diamond} = M^{\bbE} \rtimes \Ctm$ with respect to the maximal torus $T_{\diamond}$ as a subsystem of the affine root system $(\bbA, R_{\aff})$:

\begin{prop}\label{prop:centre}
	Given a relevant subspace $\bbE\in \frakE$, we put $R_{\bbE} = \{\alpha\in  R_{\aff}\;;\; \alpha\mid_{\bbE} = 0\}$. Denote $(M, \frakm_*) = (M^{\bbE}, \frakm^{\bbE}_*)$. Then the following statements hold:
	\begin{enumerate}[label=(\roman*)]
		\item\label{prop:centre-i}
			$R_{\bbE}$ coincides with the root system $R(M_{\diamond}, T_{\diamond})$ as subset of $\bfX^*(T_{\diamond})_{\bfQ}$.
		\item\label{prop:centre-ii}
			The vector subspace $\bbE_{\diamond}\subset \bbA_{\diamond}$ spanned by $\bbE$ coincides with $\bfX_*(Z(M_{\diamond}))_{\bfQ}$.
		\item\label{prop:centre-iii}
			The pointwise stabiliser $W_{\bbE} = \Stab_{W}(\bbE)$ coincides with the Weyl group $W_{M_{\diamond}} = W(M_{\diamond}, T_{\diamond})$.
		\item\label{prop:centre-iv}
			The stabiliser $W_{\bbE, \bfx} = \Stab_{W_{\bbE}}(\bfx)$  coincides with the Weyl group $W_{M_{0,\diamond}} = W(M_{0, \diamond}, T_{\diamond})$.
	\end{enumerate}
\end{prop}
\begin{proof}
	The statement~\ref{prop:centre-i} follows immediately from~\autoref{prop:Mpoids}. The statement~\ref{prop:centre-ii} follows from~\ref{prop:centre-i} by taking the intersection of the kernels of elements of $R_{\bbE} = R(M_{\diamond}, T_{\diamond})$.  \par
	Via the reflection action of $W_{M_{\diamond}}$ on $\bfX_*\left( T_{\diamond} \right)$, we may consider $W_{M_{\diamond}}$ as subgroup of $W$. Let $\kappa\in \frakA$ be an alcove such that $\ba\kappa\cap \bbE \neq\emptyset$. Then there is a subset $J^{\kappa}\subset \Delta^{\kappa}$ such that $\bbE = \left\{ y\in \bbA\;;\; \alpha(y) = 0,\; \forall \alpha\in J^{\kappa} \right\}$. Consequently, $J^{\kappa}$ is a base for the root system $R_{\bbE}$ and the subgroup $W_{\bbE} = W_{J^{\kappa}}$ is its Weyl group. We conclude that $W_{\bbE}$ coincides with $W(M_{\diamond}, T_{\diamond})$, whence~\ref{prop:centre-iii}. \par
	Now let $\bbE'$ be the smallest relevant affine subspace containing both $\bbE$ and $\bfx$. It follows that $\frakm^{\bbE'}_0 = \frakm^{\bbE}_0$ and $\frakm^{\bbE'}_n = 0$ for $n\neq 0$. Then~\ref{prop:centre-iii} applied to $(M^{\bbE'}, \frakm^{\bbE'}_*)$ yields~\ref{prop:centre-iv}.
\end{proof}

In particular, if we let $\bbE\in \frakE$ be the smallest relevant affine subspace containing $\bfx\in \bbA$, then by~\autoref{prop:Mpoids} we have $\frakm^{\bbE}_0 = \frakg_{\ubar 0}$ and $\frakm^{\bbE}_n = 0$ for $n\neq 0$. Hence $M^{\bbE}_0 = G_{\ubar 0}$ and~\autorefitem{prop:centre}{iv} yields the following:
\begin{coro}\label{coro:WGz}
	The subset $R_{\bfx} = \left\{ \alpha\in R_{\aff}\;;\; \alpha(\bfx) = 0 \right\}$ coincides with the root system $R(\Gzt, T_{\diamond})$ and the stabiliser $W_{\bfx} = \Stab_W(\bfx)$ coincides with the Weyl group $W(\Gzt,T_{\diamond})$. Moreover, the inclusion of pairs $(\Gz, T)\hookrightarrow (\Gzt, \Tt)$ induces an isomorphism $W(\Gz, T)\cong W_{\bfx}$. \hfill\qedsymbol
\end{coro}

\subsection{Relative affine root system attached to an admissible system} \label{subsec:cusp}
Let $\xi = \left(M,\frakm_*, \rmO, \scrC  \right)$\index{0@$(M, \frakm_*, \rmO, \scrC)$} be an admissible system on $\frakg_{\ubar \eta}$ as defined in~\autoref{subsec:admsys}. In particular, $\scrC$ is a $\Mtq$-equivariant cuspidal local system on $\rmO\subset \frakm$. There exists a relevant hyperplane $\bbE = \bbE^M$ \index{E@$\bbE^M$} such that $(M, \frakm_*) = (M^{\bbE}, \frakm^{\bbE}_*)$ as in~\autoref{subsec:subsp-pslevi}. We assume that $\dim \bbE^M > 0$. \par
The restriction of the affine root system $(\bbA, R_{\aff})$ to the subspace $\bbE^M$ yields an affine root system $(\bbE^M, R'_{\xi})$, where
$R'_{\xi} = \left\{ \alpha\mid_{\bbE^M_{\diamond}}\;;\; \alpha\in R_{\aff}\;;\; a\mid_{\bbE^M_{\diamond}}\neq 0 \right\}$. The root system $(\bbE^M, R'_{\xi})$ may not be reduced. Let $R_{\xi}\subset R'_{\xi}$ be the subset of indivisible roots. \par
We let $\Xi\subset \frakF$\index{Xi@$\Xi$} denote the set of facets which span $\bbE^M$ as affine subspace of $\bbA$. The elements of $\Xi$ are called $\bbE^M${\bf -alcoves}. Notice that by~\autoref{prop:Mpoids}, an $\bbE^M$-alcove $\nu\in \Xi$ is characterised by the property that $\frakl^{\nu}_* = \frakm_*$. Let $W_{M} = W(M_{\diamond}, T_{\diamond})\subset W$\index{W@$W_{M}$} be the Weyl group of $M_{\diamond}$. By~\autoref{prop:centre}, it is also the pointwise stabiliser of $\bbE^M$. As we have explained in~\autoref{subsec:fixedcomp}, the set $\Xi$ will be used to parametrise the fixed components of the partial affine Springer resolution $\dot\frakg_{\aff}$. \par

The main properties, proven in~\cite[2.4]{LYIII}, are the following:
\begin{prop}\label{prop:Wxi}
	The following statements hold:
	\begin{enumerate}
		\item\label{prop:Wxi-i}
			$R_{\xi}$ is an irreducible affine root system on $\bbE^M$. Let $W_{\xi}$ denote its Weyl group.
		\item\label{prop:Wxi-ii}
			$W_{\xi}$ is isomorphic to $N_W(W_{M}) / W_M$. 
		\item\label{prop:Wxi-iii}
			Given any $\bbE^M$-facet $\nu\in \Xi$, there is a unique base $\Delta^{\nu}_{\xi}\subset R_{\xi}$ such that the elements of $\Delta^{\nu}_{\xi}$ take positive values on $\nu$.
		\item\label{prop:Wxi-iv}
			$W_{\xi}$ acts simply transitively on $\Xi$. Consequently, all the $\bbE^M$-alcoves are of the same type.
	\end{enumerate}\hfill\qedsymbol
\end{prop}

By~\autoref{prop:Wxi}~\ref{prop:Wxi-iv}, we can denote the type of $\bbE^M$-facets by $I_\xi\subset \Delta$, so that $\Xi\subset \frakF_{I_\xi}$.  \par

\subsection{Symmetry of \texorpdfstring{$\xi$}{ξ}}\label{subsec:conjxi}

We keep the admissible system $\xi$ as in~\autoref{subsec:cusp} and the type $I_\xi\subset \Delta$. Let $\bfx^M\in \bbE^M$ denote the image of $\bfx$ under the orthogonal projection $\bbA\to \bbE^M$. \index{x@$\bfx^M$} We denote $W_{\xi, \bfx} = \Stab_{W_{\xi}}(\bfx^M)$.\index{W@$W_{\xi,\bfx}$} 

\begin{prop}\label{prop:conjxi}
	The following statements hold:
	\begin{enumerate}
		\item\label{prop:conjxi-i}
			The action of $N_{G_{\ubar 0, \diamond}}(Z_{M_\diamond})$ on $Z_{M_{\diamond}}$ by conjugation yields an isomorphism $N_{G_{\ubar 0, \diamond}}(Z_{M_\diamond}) / M_{0,\diamond} \cong W_{\xi, \bfx}$.
		\item\label{prop:conjxi-ii}
			The adjoint action of $N_{G_{\ubar 0, \diamond}}(Z_{M_\diamond})$ on $\frakm_{\eta}$ preserves the $M_{0,\diamond}$-orbit $\rmO_{\eta}\subset \frakm_{\eta}$.
		\item\label{prop:conjxi-iii}
			The cuspidal local system $\scrC$ has a $N_{G_{\ubar 0, \diamond}}(Z_{M_\diamond})$-equivariant structure.
	\end{enumerate}
\end{prop}
\begin{proof}
	Let $\bbE'\in \frakE$ be the smallest relevant affine subspace which contains $\bfx^M$. It follows that $\bbE'\subseteq\bbE$ and $(M', \frakm'_*):= (M^{\bbE'}, \frakm^{\bbE'}_*)$ defined in~\autoref{subsec:subsp-pslevi} is a $\bfZ$-graded pseudo-Levi subgroup of $G$ which contains $(M, \frakm_*)$ as graded Levi subgroup. We may view $\xi$ as an admissible system on $\frakm'_{\diamond}$ in the sense of~\cite{lusztig84}. We have the isomorphisms of groups below.
	\begin{enumerate}
		\item
			By~\autoref{coro:WGz}, the Weyl group $W(\Gzt, \Tt)$ coincides with $W_{\bfx}$.
		\item
			We can deduce from it that $N_{\Gzt}(Z_{\Mt}) / \Mzt$ is isomorphic to $W_{\xi, \bfx} = \Stab_{W_{\xi}}(\bfx^M)$.
		\item
			$W_{\sigma, \xi} = N_{M'_{\diamond}}(Z_{M_{\diamond}}) / M_{\diamond}$ coincides with the pointwise stabiliser of $\sigma$ in $W_{\xi}$;
		\item
			$W_{\sigma, \xi, \bfx} = N_{M'_{0,\diamond}}(Z_{M_{\diamond}}) / M_{0,\diamond}$ coincides with the stabiliser of $\bfx^M\in \bbE^M$ in $W_{\sigma, \xi}$.
	\end{enumerate}
	By the choice of $\sigma$, we have $W_{\sigma, \xi, \bfx} =  W_{\xi, \bfx}$. Thus we deduce $N_{G_{\ubar 0,\diamond}}(Z_{M_{\diamond}}) = N_{M'_{0,\diamond}}(Z_{M_{\diamond}})$. \par
	Since $(\Mt, \rmO, \scrC)$ is an admissible system on $\frakm'_{\diamond}$, by~\cite[9.2]{lusztig84} the action of $N_{M'_{\diamond}}(Z_{\Mt})$ by conjugation on $\frakm'_{\diamond}$ preserves $\rmO\subset\frakm_{\diamond}$ and the cuspidal local system $\scrC$ has a $N_{M'_{\diamond}}(Z_{\Mt})$-equivariant structure. The restriction of this action to the subgroup $N_{M'_{0,\diamond}}(Z_{\Mt})$ preserves $\rmO_{\eta}$ and induces an action on $\scrC_{\eta}$. Since $N_{M'_{0,\diamond}}(Z_{\Mt}) = N_{\Gzt}(Z_{\Mt})$, the results follow.
\end{proof}

\subsection{Canonical relative affine Weyl group}\label{subsec:canrelweyl}

By~\autoref{prop:Wxi}, the considerations of~\autoref{subsec:canweyl} and~\autoref{subsec:facetcox} can be applied to $\bbE^M$ with the alcove structure $\frakF(\bbE^M) = \left\{ \sigma\in \frakF\;;\; \sigma\subset \bbE^M \right\}$:
\begin{itemize}
	\item
		Each $\bbE^M$-alcove $\nu\in \Xi$ gives rise to a base $\Delta^{\nu}_{\xi}\subset R_{\xi}$ and for any pair $\nu,\nu'\in \Xi$, the $W_{\xi}$-action gives a canonical isomorphism of based affine root systems~\footnote{We omit the set of roots $R_{\xi}$ in the datum $(\bbE_{\xi,\diamond}, \Delta^{\nu}_{\xi})$ because it is determined by the Euclidean structure on $\bbE_{\xi}$ and the base $\Delta^{\nu}\subset \bbE_{\xi,\diamond}^*$.} $(\bbE^M_{\diamond},\Delta^\nu_{\xi})\cong (\bbE^M_{\diamond},\Delta^{\nu'}_{\xi})$. We define the {\bf canonical based relative affine root system} to be $(\bbE_{\xi,\diamond},\Delta_{\xi}) = \varprojlim_{\nu\in \Xi}(\bbE^M_{\diamond}, \Delta^{\nu}_{\xi})$. Set $\bbE_{\xi} = \delta^{-1}(1)\subset \bbE_{\xi,\diamond}$, which is a Euclidean affine space.\index{E@$\bbE_{\xi}, \bbE_{\xi, \diamond}$} \index{W@$\calW_{\xi}$}\index{D@$\Delta_{\xi}$} 
	\item
		We define the {\bf canonical relative affine Weyl group} $(\calW_{\xi}, \Delta_{\xi})$ to be the Weyl group of $(\bbE_{\xi,\diamond},\Delta_{\xi})$. It acts by orthogonal reflections on $\bbE_{\xi}$ and this action extends linearly to $\bbE_{\xi, \diamond}$. For each $\nu\in \Xi$, there is a canonical isomorphism $(\calW_{\xi}, \Delta_{\xi}) \cong (W_{\xi}, \Delta^{\nu}_{\xi})$.
	\item
		The $\calW$-action on $\frakA$ induces a $\calW_{\xi}$-action on $\frakF_{I_\xi}\cong \frakA/ \calW_{I_\xi}$. The latter action restricts to a simply transitive $\calW_{\xi}$-action on the subset $\Xi\subset \frakF_{I_\xi}$ which commutes with the $W_{\xi}$-action. 
	\item
		For each subset $J\subsetneq \Delta_{\xi}$, there is a boundary map $\partial_J:\Xi\to \frakF(\bbE^M)$ of {\bf relative type} $J$. Put $\Xi_J = \partial_J(\Xi)$\index{Xi@$\Xi_J$} (in particular, $\Xi = \Xi_{\emptyset}$). There is a decomposition $\frakF(\bbE^M) = \bigsqcup_{J\subsetneq \Delta_{\xi}}\Xi_J$. 
	\item
		For each subset $J\subsetneq \Delta_{\xi}$, the map $\partial_J$ induces a bijection of $W_{\xi}$-sets $\Xi / \calW_J \cong \Xi_J$.
\end{itemize}

\begin{prop}\label{prop:relW}
	Suppose that $\dim \bbE^M > 0$. Write $I = I_\xi$. The following statements hold:
	\begin{enumerate}[label=(\roman*)]
		\item
			$\calW_{\xi}$ is isomorphic to $N_{\calW}(\calW_I) / \calW_I$.
		\item
			The short exact sequence 
			\[
				1\to \calW_I\to N_{\calW}(\calW_I)\to \calW_{\xi}\to 1
			\]
			has a canonical splitting which takes any element $w\in \calW_{\xi}$ to the (unique) shortest element of the pre-image of $w$ in $N_{\calW}(\calW_I)$. We may thus regard $\calW_{\xi}$ as a subgroup of $\calW$.
		\item
			Let $\ell:\calW\to \bfN$ be the length function of the Coxeter group $(\calW, \Delta)$ and let $\ell_{\xi}: \calW_{\xi}\to \bfN$\index{l@$\ell_{\xi}$} be that of $(\calW_{\xi}, \Delta_{\xi})$. If we regard $\calW_{\xi}$ as a subgroup of $\calW$, then for $w,w'\in \calW_{\xi}$ we have
			\[
				\ell(ww') = \ell(w) + \ell(w') \Longleftrightarrow \ell_{\xi}(ww') = \ell_{\xi}(w) + \ell_{\xi}(w').
			\]
	\end{enumerate}
\end{prop}
\begin{proof}
	These statements are proven in~\cite[5.9]{lusztig76} for finite root systems. However, the same arguments work for affine root systems and Coxeter groups in general, provided that the parabolic subgroup $\calW_{I}$ is finite and satisfies the condition of~\cite[5.7.1]{lusztig76}.
\end{proof}

\begin{rema}
	A complete list of the types $I_{\xi}\subset \Delta$ with the relative affine root systems $(\bbE_{\xi}, \Delta_{\xi})$ is given in~\cite[\S 6--\S 7]{lusztig95c} for $(\bbA, \Delta)$ untwisted and~\cite[\S 11]{lusztig02} for $(\bbA, \Delta)$ twisted, where the subset $I_{\xi}$ is indicated as boxed vertices in the Dynkin diagram of $(\bbA, \Delta)$ and the relative affine root system $(\bbE_{\xi}, \Delta_{\xi})$ is indicated in the $\flat-\sharp$-diagrams.
\end{rema}

\section{Construction of \texorpdfstring{$\Phi$}{Φ}}\label{sec:phi}
We keep the assumptions of~\autoref{sec:grad} and~\autoref{sec:affineCox}. In particular, there is an admissible system $\xi = \left( M, \frakm_*, \rmO, \scrC \right)$ on $\frakg_{\ubar \eta}$ \cfauto{subsec:cusp}, a cocharacter $\lambda_0\in \bfX_*\left( T \right)$ which lifts the $\bfZ/m$-grading on $\frakg$, a sign $\varepsilon\in \left\{ 1, -1 \right\}$ and $\bfx = (\lambda_0 / m, 1)\in \bbA$.\par

Following the strategy of \cite{vasserot05} and \cite{LYIII}, we will construct a homomorphism $\Phi:\bfH_{\xi} \to \ha\calH$ from the degenerate double affine Hecke algebra (dDAHA) to the convolution algebra.

\setcounter{subsection}{-1}
\subsection{Notation for parabolic and spiral inductions}\label{subsec:notind}
In~\autoref{subsec:spiralfacet}, we have introduced a bijection between $\frakF$ and the set of spirals $\frakp_*$ such that $T\subset P_0$. We have attached to each facet $\sigma\in \frakF$ a $\varepsilon$-spiral $\frakp^{\sigma}_*$ and a splitting $\frakl^{\sigma}_*$. In the rest of the article, we will abbreviate the spiral induction (\autoref{subsec:indspirale}) by
\begin{equation*}\begin{aligned}
	\Ind_{\sigma} = \Ind^{\frakg_{\ubar \eta}}_{\frakp^\sigma_\eta}:\Db_{L^{\sigma}_{0,\diamond,q}}(\frakl^{\sigma}_{\eta})\to \Db_{\Gztq}(\frakg_{\ubar\eta}).
\end{aligned}\end{equation*} \index{Ind@$\Ind_{\sigma}$}
\par

By~\autoref{prop:centre}~\ref{prop:centre-i}, the root system $R(L^{\sigma}_{\diamond}, T_{\diamond})$ can be identified with the root subsystem of $\alpha\in R_{\aff}$ such that $\alpha\mid_{\sigma} = 0$. Given another facet $\tau\in \frakF$ such that $\sigma\le \tau$ (\autoref{subsec:facetcox}), we have $\frakp^{\tau}_*\subseteq \frakp^{\sigma}_*$ by~\autoref{coro:facet-spiral}. Set $\frakp^{\sigma\le \tau}_* = \frakp^{\tau}_*\cap\frakl^{\sigma}_*$, so that $\frakp^{\sigma\le\tau}_*$ is a parabolic subalgebra of $\frakl^{\sigma}_*$.\index{p@$\frakp^{\sigma\le\nu}_*$}  If we pick any point $y\in \tau$, then by~\autoref{prop:Ppoids}, the $T_{\diamond}$-weights of $\frakp^{\sigma\le \tau}_*$ are those roots $\alpha\in R(L^{\sigma}_{\diamond}, T_{\diamond})$ such that $\varepsilon\cdot\alpha(y) \le 0$. We will abbreviate the $\bfZ$-graded parabolic induction (\autoref{subsec:Zgr}) by
\begin{equation*}\begin{aligned}
	\Ind^{\sigma}_{\tau} = \Ind^{\frakl^{\sigma}_{\eta}}_{\frakp^{\sigma\le \tau}_\eta}:\Db_{L^{\tau}_{0,\diamond,q}}(\frakl^{\tau}_{\eta})\to \Db_{L^{\sigma}_{0,\diamond,q}}(\frakl^{\sigma}_{\eta}). 
\end{aligned}\end{equation*} \index{Ind@$\Ind^{\sigma}_{\tau}$}

\subsection{Degenerate double affine Hecke algebra}\label{subsec:daha}
Recall the canonical relative based root system $(\bbE_{\xi}, \Delta_{\xi})$ defined in~\autoref{subsec:canrelweyl}. We regard $\Delta_{\xi}$ as linear functions on the linearisation $\bbE_{\xi,\diamond}$. The relative affine Weyl group $\calW_{\xi}$ acts on $\bbE_{\xi,\diamond}$ by orthogonal reflections. It induces a $\calW_{\xi}$-action on the algebra $\bfS_{\xi}$ in which $\calW_{\xi}$ acts trivially on $u$.  \par

For each relative affine simple root $\alpha\in \Delta_{\xi}$, we define as in~\cite{LYIII} an integer $c_{\alpha}\in \bfZ_{\ge 2}$. Pick any $\bbE^M$-alcove $\nu\in \Xi$. The canonical isomorphism $(\bbE_{\xi},\Delta_{\xi})\cong (\bbE^M,\Delta^{\nu}_{\xi})$ sends $\alpha$ to a relative root $\alpha^{\nu}\in \Delta^{\nu}_{\xi}$, which cuts out a relative hyperplane $H\subset \bbE^M$. Then $H\subset \bbA$, being a relevant subspace, gives rise to a pseudo-Levi subalgebra $\frakl$ which contains $\frakm$ as Levi subalgebra. The connected component of $\bbE^M_{\bfR} \setminus H_{\bfR}$ on which $\alpha^{\nu}$ takes positive values yields a parabolic subalgebra $\frakp^+ = \frakm\oplus \fraku^+ \subset \frakl$. Pick any $e\in \rmO$. Then $c_{\alpha}$ is defined to be the biggest integer $c \ge 2$ such that $(\ad_e)^{c-2}\neq 0$ on $\fraku^+$.  \par

We define the {\bf degenerate double affine Hecke algebra} (dDAHA) $\bfH_{\xi}$ attached to the admissible system $\xi$ with the based relative affine root system $(\bbE_{\xi}, \Delta_{\xi})$. It is the associative algebra over the polynomial ring $\bfk[u]$ generated by $\left\{ x^{\mu} \right\}_{\mu\in \bbE_{\xi,\diamond}^*}$ and $\left\{ s_{\alpha} \right\}_{\alpha\in \Delta_{\xi}}$ subject to the following relations for $\mu,\nu\in \bbE_{\xi, \diamond}^*$, $r\in \bfQ$ and $\alpha\in \Delta_{\xi}$:
\begin{equation*}\begin{aligned}
	rx^{\mu} = x^{r\mu},\quad x^{\mu} + x^{\nu} = x^{\mu + \nu},\quad
	\bfk[ s_{\alpha}\;;\; \alpha\in \Delta_{\xi}] \cong \bfk \calW_{\xi} \\
	s_{\alpha}x^{\mu} - x^{s_\alpha(\mu)}s_{\alpha} = u c_{\alpha} \langle \mu,\alpha^{\vee}\rangle. \\
\end{aligned}\end{equation*}
The family $(c_{\alpha})_{\alpha\in \Delta_{\xi}}$ is usually called the {\itshape parameters} of $\bfH_{\xi}$. \par

Let $\bfk[\bbE_{\xi,\diamond}]$ be the algebra of $\bfk$-value polynomial functions on $\bbE_{\xi,\diamond}$. The subalgebra of $\bfH_{\xi}$ generated by the set $\left\{ x^{\mu} \right\}_{\mu\in \bbE_{\xi,\diamond}^*}$ is isomorphic to $\bfk[\bbE_{\xi,\diamond}]$ in the obvious way.  Therefore we may view $\bfk[\bbE_{\xi,\diamond}]$ as a subalgebra of $\bfH_{\xi}$ and denote $\mu = x^{\mu}\in \bfH_{\xi}$ for $\mu\in \bbE_{\xi,\diamond}^*$. The dDAHA $\bfH_{\xi}$\index{H@$\bfH_{\xi}$}, as vector space, can be written as a tensor product 
\[
	\bfH_{\xi} = \bfS_{\xi}\otimes \bfk\calW_{\xi}
\]
of two subalgebras: the polynomial subalgebra $\bfS_{\xi} = \bfk[\bbE_{\xi,\diamond}]\otimes \bfk[u]$\index{S@$\bfS_{\xi}$} and the group ring of the canonical affine Weyl group $\bfk \calW_{\xi}$. They satisfy the following commutation laws:
\begin{equation*}\begin{aligned}
	s_\alpha\; f - s_\alpha(f)\; s_\alpha = uc_\alpha \frac{f - s_\alpha(f)}{\alpha}, \quad \alpha\in \Delta_{\xi},\quad f\in \bfS_{\xi}
\end{aligned}\end{equation*}
The elements $u,\delta\in \bfS_{\xi}$ are central in $\bfH_{\xi}$. 

\begin{rema}
	The class of dDAHAs which can be constructed in the present setting is limited for non-simply laced root systems. Since the constants $c_\alpha$ are certain integers determined by the cuspidal pair $(\rmO, \scrC)$, only certain integral proportions between parameters can appear. In the list of G. Lusztig~\cite[\S 6--\S 7]{lusztig95c}~\cite[\S 11]{lusztig02}, the number $c_{\alpha}$ for $\alpha\in \Delta_{\xi}$ corresponds to the number $A$ for in the vertex $\sharp^{A\times B}_k$ or $\flat^{A\times B}_k$ indicated in the {\itshape $\flat-\sharp$-diagrams}.
\end{rema}

\subsection{Extension algebra \texorpdfstring{$\ha\calH$}{H}} \label{subsec:indcusp}

We define the {\bf Lusztig sheaf} for the $\bbE^M$-alcove $\nu\in \Xi$ to be 
\[
	\bfI^{\nu}=\Ind_{\nu}\scrC_{\eta} \in \Db_{\Gztq}(\frakg_{\ubar\eta}), 
\]\index{I@$\bfI^{\nu}$}
 We define the set of $W_{\xi,\bfx}$-conjugacy classes $\ubar\Xi = W_{\xi, \bfx} \backslash \Xi$\index{Xi@$\ubar\Xi$}. Where $W_{\xi, \bfx}$ is as in~\autoref{subsec:conjxi}.
 We show that $\bfI^{\nu}$ depends only on the class $\ubar\nu$ in $\ubar \Xi$. By~\autorefitem{prop:conjxi}{iii}, we may view the cuspidal local system $\scrC_{\eta}$ as $W_{\xi,\bfx}$-equivariant local system over the stack $[\rmO_{\eta} / \Mztq]$. If $\nu,\nu'\in \Xi$ are such that $\ubar \nu = \ubar \nu'$ (i.e. if they are in the same $W_{\xi,\bfx}$-orbit), let $w\in W_{\xi,\bfx}$ be such that $w\nu = \nu'$. Then $w$ gives a stack automorphism $\Ad_w:[\frakm_{\eta} / \Mztq]\cong [\frakm_{\eta} / \Mztq]$. The $W_{\xi,\bfx}$-equivariance of $\scrC_{\eta}$ gives an isomorphism $\Ad_{w}^{*}\scrC_{\eta} \cong \scrC_{\eta}$. Consequently, we have canonical isomorphisms
\[
	\bfI^{\nu'} = \Ind_{\nu'}\scrC_{\eta} \cong \Ind_{\nu}\Ad_{w}^*\scrC_{\eta} \cong \Ind_{\nu}\scrC_{\eta}= \bfI^{\nu}.
\]
\index{I@$\bfI^{\nu}$}

\begin{prop}\label{prop:BBDG}
				For $\ubar\nu\in \ubar\Xi$, the following statements hold
				\begin{enumerate}[label=(\roman*)]
								\item\label{prop:BBDG-ii}
									We have $\bfI^{\nu} \cong \bigoplus_{k\in \bfZ} \pH^k\bfI^{\nu}[-k]$ and each factor $\pH^k\bfI^{\nu}$ is a semi-simple perverse sheaf. 
								\item\label{prop:BBDG-i}
									If $\varepsilon = \eta / |\eta|$, then the complex $\bfI^{\nu}$ is supported on the nilpotent cone, i.e. $\bfI^{\nu}\in \Db_{\Gztq}\left( \frakg_{\ubar \eta}^{\nil} \right)$.
				\end{enumerate}
\end{prop}
\begin{proof}
	The statement~\ref{prop:BBDG-ii} follows from the Be\u \i linson--Bernstein--Deligne--Gabber decomposition theorem and the purity of $\scrC$~\cite[1.4]{lusztig95b}. The statement~\ref{prop:BBDG-i} follows from to the assumption that $\varepsilon = \eta / |\eta|$ and~\eqref{equa:nil}.
\end{proof}

Define the following space
\begin{equation*}\begin{aligned}
	\ha\calH = \prod_{\nu'\in \ubar\Xi}\bigoplus_{\nu\in \ubar\Xi}\Ext_{\Gztq}^*\left( \bfI^{\nu'},\bfI^{\nu} \right)_{0}.
\end{aligned}\end{equation*}\index{H@$\ha\calH$}
Recall that the index $0$ means the completion at the augmentation ideal $\rmH^{>0}_{\Gztq}$~\cfauto{subsec:compl}. The Yoneda product gives a ring structure on the space $\ha\calH$. We will introduce a topology on it in~\autoref{sec:density}.
\begin{rema}
	Although the set $\Xi$ (and hence $\ubar\Xi$) is infinite, there is only a finite number of possibilities for the subspace $\frakp^{\nu}_{\eta}$ for $\nu\in \Xi$, on which depends the functor $\Ind_{\nu}$; therefore, there is only a finite number of isomorphism classes of the complexes $\bfI^{\nu}$. However, we needed to take the sum over the infinite set $\ubar \Xi$ in order to construct an action of the dDAHA $\bfH_{\xi}$ on it. This phenomenon was first observed in~\cite[3.4]{VV09} and was used to study to the dimension of simple $\bfH_{\xi}$-modules.
\end{rema}

\subsection{Graded affine Hecke algebras and \texorpdfstring{$\Phi_{J}$}{ΦJ}}\label{subsec:affinehecke}
Let $J\subsetneq \Delta_{\xi}$. We have a parabolic subgroup $\left( \calW_{\xi, J}, J \right)$ of the canonical relative Weyl group $\left( \calW_{\xi}, \Delta_{\xi} \right)$. We define $\bfH_{\xi,J}$\index{H@$\bfH_{\xi,J}$} to the subalgebra of $\bfH_{\xi}$ generated by $\bfS_{\xi}$ and $\{s_a\}_{a\in J}$. As vector space, it admits a decomposition $\bfH_{\xi,J} \cong \bfk \calW_{\xi,J}\otimes \bfS_{\xi}$. In the case where $J = \emptyset$, it recovers the polynomial algebra $\bfH_{\xi,\emptyset} = \bfS_{\xi}$.  \par

Recall the set $\Xi_J = \partial_J \Xi \subset \frakF(\bbE^{M})$ of $\bbE^M$-facets of (relative) type $J$ introduced in~\autoref{subsec:canrelweyl}. For each $\sigma\in \Xi_{J}$, we define $\Xi^{\sigma} =  \left\{ \nu\in \Xi\;;\; \partial_{J}\nu = \sigma \right\}$\index{Xi@$\Xi^{\sigma}$}, so that there is a partition $\Xi = \bigsqcup_{\sigma\in \Xi_J}\Xi^{\sigma}$. In other words, $\Xi^{\sigma}$ is the subset of $\Xi$ consisting of those $\bbE^M$-alcoves whose closure contains $\sigma$. The stabiliser $W_{\xi, \sigma} = \Stab_{W_{\xi}}(\sigma)$ acts simply transitively on $\Xi^{\sigma}$. Let $\ubar\Xi^{\sigma} = W_{\xi,\sigma,\bfx}\backslash \Xi^{\sigma}$, where $W_{\xi, \sigma, \bfx} = W_{\xi,\sigma}\cap W_{\xi,  \bfx}$. \par

With the datum $(L^{\sigma}_{\diamond}, \frakl^{\sigma}_{\diamond, *}, \xi)$, we are in the situation of~\autoref{subsec:Zgr}. The group $\calW_{\xi, J}$ can be identified with the canonical relative Weyl group of $L^{\sigma}_{\diamond}$ with respect to $\xi$. Moreover, by~\autoref{prop:centre}~\ref{prop:centre-ii}, the vector space $\bbE^M_{\diamond}$ coincides with $\bfX_*(Z_{M_{\diamond}})_\bfQ$ and $\Xi^{\sigma}$ can be identified with the set denoted by $\Xi$ in~\autoref{subsec:coxZ}. Choosing any $\nu_0\in \Xi^\sigma$, we have an isomorphism $(\bbE_{\xi,\diamond}, \Delta_{\xi}) \cong \left( \bbE^M_{\diamond}, \Delta_{\xi}^{\nu_0} \right)$. Under this isomorphism, the based root subsystem $(\bbE_{\xi,\diamond}, J)$ is sent to a based root system $\left( \bbE^M_{\diamond}, \Delta_{\xi}^{\sigma} \right)$ with $\Delta^{\sigma}_{\xi}\subset \Delta^{\nu_0}_{\xi}$. Thus we can identify the subalgebra $\bfH_{\xi, J}\subset \bfH_{\xi}$ with the graded affine Hecke algebra (which was denoted by $\underline\bfH_{\xi}$ in~\autoref{subsec:rappel}) attached to $(\bbE^M_{\diamond}, \Delta^{\sigma}_{\xi})$ via the isomorphism $(\bbE_{\xi,\diamond}, \Delta_{\xi}) \cong ( \bbE^M_{\diamond}, \Delta_{\xi}^{\nu_0} )$. \par

For $\sigma\in \Xi_J$, define the parabolic version of $\ha\calH$ attached to $\sigma$:
\[
	\ha\calH_{\sigma} = \bigoplus_{\ubar\nu,\ubar\nu'\in \ubar\Xi^{\sigma}} \Ext^*_{L^{\sigma}_{\diamond,0,q}}\left(\Ind^{\sigma}_{\nu'}\scrC_{\eta} , \Ind^{\sigma}_{\nu}\scrC_{\eta} \right)_{0}.
\]\index{H@$\ha\calH_{\sigma}$}
Recall the element $\bfx = (\lambda_0/m, 1)\in \bbA$. We consider the equivariant localisation by action of the fractional cocharacter $(\bfx, \eta/2m) = (2\lambda_0, 2m, \eta)/2m\in \bfX_*(T_{\diamond}\times \Cq)$. Notice that $(\frakl^{\sigma})^{(\bfx, \eta/2m)} = (\frakl^{\sigma})^{(2\lambda_0, 2m, \eta)}= \frakl^{\sigma}_\eta$. The constructions of~\autoref{subsec:rappel} can thus be applied~\footnote{Our group is now $L^{\sigma}_{\diamond}$ with maximal torus $\Tt$. The $\bfZ$-grading on $\frakl^{\sigma}$ is given by $\lambda_{\diamond} = (\lambda_0, m)$ with $\delta(\lambdat) = m$. Notice that the completion is taken at $(\bfx, \eta/2m)$ instead of $(\lambda_{\diamond},\eta/2)$. It is fine because they are proportional.} to the datum $(L^{\sigma}_{\diamond}, \frakl^{\sigma}_{\diamond,*}, \xi)$. In particular, there is a homomorphism $\Phi_{\sigma}$ defined by means of equivariant localisation
\begin{equation}\begin{aligned}
	\begin{tikzcd}
		\bfH_{\xi, J}\arrow{d}{\cong}\arrow{r}{\Phi_{\sigma}}&\bigoplus_{\ubar\nu,\ubar \nu'\in \ubar\Xi^{\sigma}}\Ext^*_{L^{\sigma}_{\diamond,0,q}}\left(\Ind^{\sigma}_{\nu'}\scrC_{\eta} , \Ind^{\sigma}_{\nu}\scrC_{\eta} \right)_{0}\;= \ha\calH_{\sigma}\\
		\Ext^*_{L^{\sigma}_{\diamond,q}}(\Ind^{\frakl^{\sigma}}_{\frakp^{\sigma\le\nu_0}}\scrC, \Ind^{\frakl^{\sigma}}_{\frakp^{\sigma\le\nu_0}}\scrC)\arrow[hookrightarrow]{r}& \Ext^*_{L^{\sigma}_{\diamond,q}}(\Ind^{\frakl^{\sigma}}_{\frakp^{\sigma\le\nu_0}}\scrC, \Ind^{\frakl^{\sigma}}_{\frakp^{\sigma\le\nu_0}}\scrC)_{(\bfx, \eta/2m)}\arrow{u}{\cong}[swap]{\mathrm{loc}}
	\end{tikzcd}
\end{aligned}\end{equation}
\index{Ph@$\Phi_\sigma$}
which is independent of the choice of $\nu_0$, see the proof of~\autoref{prop:compatibleZ}.

We set 
\begin{equation*}\begin{aligned}
	\Phi_{J} = \left( \Phi_{\sigma} \right)_{\sigma\in \Xi_J}: \bfH_{\xi, J}\to \prod_{\ubar\sigma\in \ubar\Xi_J}\ha\calH_{\sigma},\quad \text{where $\ubar\Xi_J = W_{\xi, \bfx}\backslash \Xi_J$}.
\end{aligned}\end{equation*}\index{Ph@$\Phi_J$}

\subsection{Construction of \texorpdfstring{$\Phi$}{Φ}}\label{subsec:conPhi}

Let $K\subseteq J \subsetneq \Delta_{\xi}$. Given $\sigma\in \Xi_J$, we let $\Xi_K^{\sigma} = \left\{ \tau\in \Xi_K\;;\; \partial_J \tau = \sigma \right\}$.\index{Xi@$\Xi_J^{\sigma}$} For each $\sigma\in \Xi_J$, $\tau\in \Xi^{\sigma}_K$ and $\nu,\nu'\in \Xi^{\tau}$, there is a homomorphism given by the functoriality of $\Ind^{\sigma}_{\tau}$:
\begin{equation*}\begin{aligned}
	\Ext^*_{L^{\tau}_{0,\diamond,q}}\left(\Ind^{\tau}_{\nu'}\scrC , \Ind^{\tau}_{\nu}\scrC \right)_{0}\to  \Ext^*_{L^{\sigma}_{0,\diamond,q}}\left(\Ind^{\sigma}_{\nu'}\scrC , \Ind^{\sigma}_{\nu}\scrC \right)_{0}.
\end{aligned}\end{equation*}
Summing over $\ubar\nu,\ubar\nu'\in \ubar\Xi^{\tau}$, we get
\begin{equation*}\begin{aligned}
	\ha\calH_{\tau} = \bigoplus_{\ubar\nu,\ubar\nu'\in\ubar\Xi^{\tau}}\Ext^*_{L^{\tau}_{0,\diamond,q}}\left(\Ind^{\tau}_{\nu'}\scrC , \Ind^{\tau}_{\nu}\scrC \right)_{0}\to  \bigoplus_{\ubar\nu,\ubar\nu'\in\ubar\Xi^{\tau}}\Ext^*_{L^{\sigma}_{0,\diamond,q}}\left(\Ind^{\sigma}_{\nu'}\scrC , \Ind^{\sigma}_{\nu}\scrC \right)_{0}.
\end{aligned}\end{equation*}
On the other hand, the partition $\Xi^{\sigma} = \bigsqcup_{\tau\in \Xi^{\sigma}_K}\Xi^{\tau}$ yields an inclusion of subring
\begin{equation*}\begin{aligned}
	\prod_{\ubar\tau\in \ubar\Xi^{\sigma}_K}\bigoplus_{\ubar\nu,\ubar\nu'\in \ubar\Xi^{\tau}}\Ext^*_{L^{\sigma}_{0,\diamond,q}}\left(\Ind^{\sigma}_{\nu'}\scrC , \Ind^{\sigma}_{\nu}\scrC \right)_{0}\subset \ha\calH_{\sigma}.
\end{aligned}\end{equation*}
Combining the above two maps, we obtain an ring homomorphism
\begin{equation*}\begin{aligned}
	\psi^{\sigma}_K:\prod_{\tau\in \ubar\Xi^{\sigma}_K}\ha\calH_\tau\to \ha\calH_\sigma.
\end{aligned}\end{equation*}\index{ps@$\psi^{\sigma}_K$}
\autoref{prop:compatibleZ} can be reformulated as follows:
\begin{lemm}\label{lemm:PhiJPhiK}
	The following diagram commutes
	\[
		\begin{tikzcd}[column sep = 50pt]
			\bfH_{\xi,K}\arrow{r}{\left(\Phi_{\tau}\right)_{\ubar\tau\in\ubar\Xi^{\sigma}_K}}\arrow[hookrightarrow]{d} & \prod_{\ubar\tau\in\ubar\Xi^{\sigma}_K}\ha\calH_{\tau} \arrow{d}{\psi^{\sigma}_K} \\
			\bfH_{\xi,J}\arrow{r}{\Phi_{\sigma}} & \ha\calH_\sigma
		\end{tikzcd}.
	\]
	\hfill\qedsymbol
\end{lemm}
Taking product over $\ubar\sigma\in\ubar\Xi_J$, we obtain a commutative square of rings
\[
	\begin{tikzcd}[column sep = 50pt]
		\bfH_{\xi,K}\arrow{r}{\Phi_K}\arrow[hookrightarrow]{d} & \prod_{\ubar\tau\in\ubar\Xi_K}\ha\calH_{\tau} \arrow{d}{\prod_{\ubar\sigma}\psi^{\sigma}_K} \\
		\bfH_{\xi,J}\arrow{r}{\Phi_{J}}  & \prod_{\ubar\sigma\in\ubar\Xi_J}\ha\calH_{\sigma} 
	\end{tikzcd}.
\]
Similarly, using the functoriality of the spiral induction $\Ind_{\sigma}$ for $\ubar\sigma\in \ubar\Xi_J$, we have a ring homomorphism
\begin{equation*}\begin{aligned}
	\psi_J: \prod_{\ubar\sigma\in \ubar\Xi_J}\ha\calH_{\sigma}\to \ha\calH.
\end{aligned}\end{equation*}\index{ps@$\psi_J$}
The transitivity of spiral induction with parabolic induction~\eqref{equa:transindsp} implies that $\psi_K = \psi_J\circ \left(\prod_{\ubar\sigma\in \ubar\Xi_J}\psi^{\sigma}_K\right)$ for $K\subset J$. 

\begin{theo}\label{theo:Phi}
	There is a unique ring homomorphism
	\begin{equation*}\begin{aligned}
		\Phi: \bfH_{\xi}\to \ha\calH
	\end{aligned}\end{equation*}\index{Ph@$\Phi$}
	such that for each $J\subsetneq \Delta_{\xi}$, the following diagram commutes
	\[
		\begin{tikzcd}[column sep = 50pt]
			\bfH_{\xi,J}\arrow{r}{\Phi_J}\arrow[hookrightarrow]{d} & \prod_{\ubar\sigma\in\ubar\Xi_J}\ha\calH_{\sigma} \arrow{d}{\psi_J} \\
			\bfH_{\xi}\arrow{r}{\Phi}  & \ha\calH
		\end{tikzcd}.
	\]
\end{theo}
\begin{proof}
	Taking inductive limit over the partially ordered set $\left( \left\{ J\subsetneq \Delta_{\xi} \right\}, \subseteq \right)$, we obtain by~\autoref{lemm:PhiJPhiK} a ring homomorphism
	\[
		\varinjlim_{J}\bfH_{\xi, J}\xrightarrow{(\psi_J\circ \Phi_J)_{J}}\ha\calH.
	\]
	It suffices to show that the natural ring homomorphism induced by inclusion of parabolic subalgebras
	\[
		\iota: \varinjlim_{J\subsetneq \Delta_{\xi}}\bfH_{\xi, J}\to \bfH_{\xi}
	\]
	is a ring isomorphism. The map $\iota$ is surjective because each generator of $\bfH_{\xi}$ lies in a parabolic subalgebra. Observe the following: 
	\begin{enumerate}
		\item
			Each of the defining relations of $\bfH_{\xi}$ involves at most two elements from the set of simple reflections $\left\{ s_{a} \right\}_{a\in \Delta_{\xi}}$.
		\item
			In the case where $\#\Delta_{\xi} = 2$, there is no braid relation between the two simple reflections.
	\end{enumerate}
	Each defining relation of $\bfH_{\xi}$ must therefore lie in the parabolic subalgebra $\bfH_{\xi, J}$ for some $J\subsetneq \Delta_{\xi}$. It follows that $\iota$ is also injective.
\end{proof}
So far, we have very little information about the maps $\psi_J$ and $\Phi$. We will show in~\autoref{prop:XZpara} that $\psi_J$ is injective and we will determine its image.

\section{Geometry of the Steinberg varieties}\label{sec:geomZ}
We have explained in~\autoref{subsec:fixedcomp} that $\dot\frakg^{\lambda_{\diamond,q}}_{\aff}\cong \bigsqcup_{\ubar\nu\in \ubar \Xi}\calT^{\nu}$ is an infinite disjoint union of smooth algebraic varieties. In this section, we study the fixed points $\ddot\frakg^{\lambda_{\diamond,q}}_{\aff}$ of the Steinberg variety $\ddot\frakg_{\aff} = \dot\frakg_{\aff}\times_{\frakg_{\aff}}\dot\frakg_{\aff}$. We describe the stratification of $\ddot\frakg_{\aff}^{\lambda_{\diamond,q}}$ by the relative position of pairs of spirals. The exposition below will be given purely in terms of finite dimensional varieties and the ind-schemes such as $\dot\frakg_{\aff}$ and $\ddot \frakg_{\aff}$ will only serve heuristically. \par

The main result~\autoref{prop:strZ} states that the filtration by Bruhat order on the Steinberg varieties is a filtration by closed subvarieties and that the convolution product on the Steinberg variety respects the filtration by Bruhat order. This will allow us to define a filtration on the convolution algebra in the next section~\autoref{sec:convolution}. \par

We keep the assumptions of the previous sections. 
\subsection{Double quotient} \label{subsec:relpos}
Recall that $I = I_\xi$ is the type of $\bbE^M$-chambers. We recollect some basic combinatorial properties about the double quotient $\calW_I\backslash \calW / \calW_I$. 
\begin{prop}\label{prop:doublequo}
	Recall that $W_{\xi}$ acts on $\Xi$ and $W$ acts on $\frakF_I$.
	\begin{enumerate}[label=(\roman*)]
		\item\label{prop:doublequo-i}
			There is a canonical inclusion $\calW_{\xi}\hookrightarrow \calW_{I}\backslash \calW /\calW_{I}$. 
		\item\label{prop:doublequo-ii}
			There is a canonical bijection
			\begin{equation*}\begin{aligned}
				W\backslash \left(\frakF_{I} \times \frakF_{I}\right)&\cong \calW_{I}\backslash \calW / \calW_{I} \\
			\end{aligned}\end{equation*}
			which induces a bijection
			\begin{equation*}\begin{aligned}
				W_{\xi}\backslash \left(\Xi  \times \Xi\right)&\cong \calW_{\xi}.\\
			\end{aligned}\end{equation*}
	\end{enumerate}
\end{prop}
\begin{proof}
	The inclusion in assertion \ref{prop:doublequo-i} is from~\autoref{prop:relW}, given by the shortest representatives. Choosing any $\kappa\in \frakA$, we define
	\begin{equation}\begin{aligned}\label{equa:relpos}
		\calW_{I}\backslash \calW / \calW_{I}&\to W\backslash \left(\frakF_{I} \times \frakF_{I}\right) \\
		\calW_{I}w\calW_{I} &\mapsto W\left(\partial_{I}\kappa , w^{\kappa}\partial_{I}\kappa\right).
	\end{aligned}\end{equation}
	Since $\frakF_I\cong \frakA / \calW_I$ as $W$-set, the above map is a bijection and does not depend on the choice of $\kappa$. The second statement follows, since both $W_{\xi}$ and $\calW_{\xi}$ act simply transitively on $\Xi$.
\end{proof}

\begin{defi}
	The image of any pair $(\nu, \nu')\in \frakF_{I}\times \frakF_{I}$ under the map
	\begin{equation*}\begin{aligned}
		\frakF_{I}\times \frakF_{I} \to W\backslash\left(\frakF_{I}\times \frakF_{I} \right) \cong \calW_{I} \backslash \calW / \calW_{I}
	\end{aligned}\end{equation*}
	is called the {\bf relative position} of $\nu$ and $\nu'$. The relative position is called {\bf good} if it is in the image of the inclusion $\calW_{\xi}\hookrightarrow \calW_{I} \backslash \calW / \calW_{I}$ and {\bf bad} otherwise. 
\end{defi}

We introduce a partial order on the double quotient $\calW_I \backslash \calW / \calW_I$.  For $w\in \calW$ we denote by $[w] = \calW_{I}w\calW_{I} \in \calW_{I}\backslash \calW / \calW_{I}$\index{w@$[w]$} the double coset containing $w$. 
\begin{prop}\label{prop:ordre}
	The following statements hold:
	\begin{enumerate}
		\item
			Each double coset $[w]\in \calW_I\backslash \calW / \calW_I$ has a maximal (resp. minimal) representative in $\calW$ with respect to the Bruhat order, denoted by $\max([w])\in \calW$ (resp. $\min([w])$). The map $w\mapsto \max([w])$ (resp. $w\mapsto \min([w])$ is order-preserving.
		\item
			Define partial orders $\le_{\max}$ and $\le_{\min}$ on $\calW_I\backslash \calW / \calW_I$ by 
			\begin{equation*}\begin{aligned}
				[w] \le_{\max} [w'] \Leftrightarrow \max([w]) \le \max([w']) \\
				[w] \le_{\min} [w'] \Leftrightarrow \min([w]) \le \min([w']). \\
			\end{aligned}\end{equation*}
			Then $\le_{\max}$ and $\le_{\min}$ coincide.
		\item
			The restriction of $\le_{\max}$ (resp. $\le_{\min}$) on $\calW_{\xi}\subset \calW_{I}\backslash \calW / \calW_{I}$ refines the Bruhat order on $\calW_{\xi}$. 
	\end{enumerate}
\end{prop}
\begin{proof}
	For $w\in \calW$, let $p(w)\in \calW$ be the maximal element in $\calW_I w$ and let $q(w)\in \calW$ be the maximal element in $w\calW_I$. It is clear that $p:\calW\to \calW$ and $q:\calW\to \calW$ preserve the order. \par
	Fix $[w] \in \calW_I\backslash \calW / \calW_I$. Replacing $w$ with $q(p(w))$, we may suppose that $p(w) = w = q(w)$. Given any $w'\in [w]$, pick $v\in \calW_Iw' \cap w\calW_I \neq \emptyset$. It follows that $v \le w$ so $w'\le p(w') = p(v) \le p(w) = w$. Hence $w$ is maximal in $[w]$. Moreover, we see that $\max([w]) = q(p(w))$. Similarly, the minimal element exists in $[w]$.  \par
	Since $p$ and $q$ preserve the Bruhat order, so does the map $w\mapsto \max([w])$. Similarly, the map $w\mapsto \min([w])$ also preserves the Bruhat order. Hence $\le_{\min}$ and $\le_{\max}$ coincides.  \par
	The last statement follows from~\autoref{prop:relW}.
\end{proof}

We say that a subset $C$ of a partially ordered set $(S, \le)$ is an {\itshape ideal} if $x\in C$ and $y\le x$ implies $y\in C$. The combinatorial results below follow from~\autoref{prop:ordre} and are easy to prove.

\begin{coro}\label{lemm:CC}
	The following statements hold
	\begin{enumerate}
		\item
			If $\calI\subset \calW$ satisfies $\calI = \calW_I \calI \calW_I$, then $\calI$ is an ideal of $(\calW, \le)$ if and only if $\calW_I \backslash \calI / \calW_I$ is an ideal of $\left(\calW_I\backslash \calW / \calW_I, \le_{\max}\right)$.
		\item
			If $\calI$ and $\calJ$ are ideals of $(\calW,\le)$, then so is $\calI\calJ$.
		\item
			If $\calI$ and $\calJ$ are ideals of $\left(\calW_I\backslash \calW / \calW_I, \le_{\max}\right)$, then so is the product 
			\[
				\calI\calJ = \left\{ [ww']\in \calW_I\backslash\calW / \calW_I\;;\; [w]\in \calI\,,\; [w']\in \calJ \right\}.
			\]
	\end{enumerate}\hfill\qed
\end{coro}

\subsection{The fixed-point components \texorpdfstring{$\calX^{\nu,\nu'}$}{X}}\label{subsec:X}
 Consider the space $\calP$ of spirals which are $\Gz$-conjugate to $\frakp^{\nu}_*$ for some $\nu\in \Xi$. This space is isomorphic to an infinite disjoint union of partial flag varieties of $G_{\ubar 0}$:
\begin{equation*}\begin{aligned}
	\calP \cong \bigsqcup_{\ubar\nu\in \ubar\Xi} \calP^{\nu},\quad \calP^{\nu} = G_{\ubar 0} / P^{\nu}_0&\hookrightarrow \calP
\end{aligned}\end{equation*}\index{P@$\calP^{\nu}$} \par
given by $G_{\ubar 0} / P^{\nu}_0\ni gP^{\nu}_0\mapsto \Ad_g \frakp^{\nu}_*\in \calP$. \par
	For $\nu, \nu'\in \Xi$, we set
\begin{equation*}\begin{aligned}
	\calX^{\nu, \nu'} = \calP^{\nu} \times  \calP^{\nu'}.
\end{aligned}\end{equation*} \index{X@$\calX^{\nu, \nu'}$}
Recall that by~\autoref{coro:WGz}, the Weyl group $W(\Gz,T)$ is isomorphic to the stabiliser $W_{\bfx} = \Stab_W(\bfx)$. Let $\nu, \nu'\in \Xi$. By the Bruhat decomposition, the orbits of the $G_{\ubar 0}$-action on $\calX^{\nu, \nu'}$ are in canonical bijection with the orbits of diagonal left translation of $W_{\bfx}$ on $\left(W_{\bfx} / W_{\bfx, \nu}\right) \times \left(W_{\bfx} / W_{\bfx,\nu'}\right)$, where $W_{\bfx, \nu} = W_{\bfx}\cap W_{\nu}$. \par

Define the map of {\itshape relative position} $\Pi_{\nu, \nu'} : \calX^{\nu, \nu'}\to \calW_{I}\backslash \calW/\calW_{I}$ by \index{P@$\Pi_{\nu, \nu'}$}
	\begin{equation*}\begin{aligned}
		G_{\ubar 0}\backslash\calX^{\nu, \nu'}\cong W_{\bfx}\backslash \left(\left(W_{\bfx} / W_{\bfx, \nu}\right) \times \left(W_{\bfx} / W_{\bfx, \nu'}\right)\right) &\rightarrow W\backslash \left(\frakF_{I} \times \frakF_{I}\right) \cong \calW_{I}\backslash \calW /\calW_{I} \\
		W_{\bfx}\cdot(u W_{\bfx, \nu}, v W_{\bfx, \nu'}) &\mapsto W\cdot (u \nu, v\nu'),
	\end{aligned}\end{equation*}
	the last isomorphism given by~\autorefitem{prop:doublequo}{ii}. We define for each $w\in \calW$ and each $\nu, \nu'\in \Xi$ the locally closed subvariety
\begin{equation*}\begin{aligned}
	\calX^{\nu, \nu'}_{w} = \Pi^{-1}_{\nu,\nu'}([w])\subset \calX^{\nu,\nu'},
\end{aligned}\end{equation*} \index{X@$\calX^{\nu, \nu'}_w$}
equipped with the reduced subscheme structure. Obviously, $\calX^{\nu, \nu'}_w = \calX^{\nu, \nu'}_{w'}$ if $[w] = [w']$. Similarly, for each ideal $\calI$ of $(\calW_I\backslash \calW / \calW_I, \le_{\max})$, we put 
\[
	\calX^{\nu,\nu'}_{\calI} = \Pi_{\nu,\nu'}^{-1}(\calI) = \bigcup_{[y]\in \calI}\calX^{\nu,\nu'}_y.
\]\index{X@$\calX^{\nu,\nu'}_{\calI}$}
\par

The following statements are analogues of standard results about $(B,N)$-pairs and convolution on flag manifolds. The arguments are similar to the classical ones using the Demazure desingularisation.
\begin{lemm}\label{lemm:XI}
	The following statements hold:
	\begin{enumerate}
		\item\label{lemm:XI-i}
			If $\calI$ is an ideal of $(\calW_I\backslash \calW / \calW_I, \le_{\max})$, then $\calX^{\nu,\nu'}_{\le \calI}$ is closed in $\calX^{\nu,\nu'}$.
		\item\label{lemm:XI-ii}
			If $\calI$ and $\calJ$ are ideals of $\calW_I\backslash \calW / \calW_I$, then the image of the projection
			\[
				\calX^{\nu,\nu'}_{\calI}\times_{\calP^{\nu'}}\calX^{\nu',\nu''}_{\calJ}\to \calX^{\nu',\nu''},\quad  (g P^{\nu}_0, g' P^{\nu'}_0, g'' P^{\nu''}_0)\mapsto (gP^{\nu}_0, g''P^{\nu''}_0).
			\]
			is contained in $\calX^{\nu,\nu''}_{\calI\calJ}$, where $\calI\calJ\subset \calW_I\backslash \calW / \calW_I$ is the product of $\calI$ and $\calJ$.
	\end{enumerate}
\end{lemm}
\begin{proof}
	Consider first the scheme $\calB$ which parametrises minimal spirals. For each minimal spiral $\frakp_*$, we can find $g\in G_{\ubar 0}$ such that $T\subset g P_0g^{-1}$. Since by~\autoref{coro:facet-spiral}, the minimal $T$-stable spirals are parametrised by $\frakA$, it follows that $\calB$ is isomorphic to an infinite disjoint union of flag varieties of $\Gz$:
	\begin{equation*}\begin{aligned}
		\calB \cong \bigsqcup_{\ubar\kappa \in \ubar \frakA} \calB^{\kappa},\quad  \calB^{\kappa} = G_{\ubar 0} / P^{\kappa}_0,\quad \ubar\frakA = W_{\bfx}\backslash\frakA,
	\end{aligned}\end{equation*}
	given by $\calB^{\kappa}\ni gP^{\kappa}_0 \mapsto \Ad_g \frakp^{\kappa}_*\in \calB$. Define $\calY^{\kappa,\kappa'} = \calB^{\kappa}\times \calB^{\kappa'}$. We can define as above a map of relative position $\Pi_{\kappa,\kappa'}: \calY^{\kappa,\kappa'}\to \calW$. \par
	\begin{enumerate}
		\item[\ul{Step 1}.]
			Let $\kappa\in \frakA$ be an alcove and $w,s\in \calW$ with $\ell(s) = 1$. Denote $\kappa' = w^{-1}\kappa$ and $\kappa'' = s^{-1} \kappa'$. We consider the image of the map 
			\begin{equation}\begin{aligned}\label{equa:convY}
				\calY^{\kappa, \kappa'}_{w}\times_{\calB^{\kappa'}} \calY^{\kappa', \kappa''}_s \to \calY^{\kappa,\kappa''}.
			\end{aligned}\end{equation}
			\begin{enumerate}
				\item[Case 1.] 
					Suppose that $\ell(ws) = \ell(w) + 1$. Then we have $P^{\kappa'}_0 = \left( P^{\kappa}_0\cap P^{\kappa'}_0\right)\left(P^{\kappa'}_0\cap P^{\kappa}_0  \right)$ and hence 
					\[
						\calY^{\kappa, \kappa'}_w\times_{\calB^{\kappa'}}\calY^{\kappa', \kappa''}_s \cong \calY^{\kappa, \kappa''}_{ws}. 
					\]
				\item[Case 2.] 
					Suppose that $\ell(ws) = \ell(w) - 1$. By the previous case, we have $\calY^{\kappa,\kappa''}_{ws}\times_{\calB^{\kappa''}}\calY^{\kappa'',\kappa'}_{s}\cong \calY^{\kappa, \kappa'}_{w}$. Hence
					\begin{equation*}\begin{aligned}
						\calY^{\kappa,\kappa'}_w\times_{\calB^{\kappa'}}\calY^{\kappa',\kappa''}_s \cong \calY^{\kappa,\kappa''}_{ws}\times_{\calB^{\kappa''}}\calY^{\kappa'',\kappa'}_s\times_{\calB^{\kappa'}}\calY^{\kappa',\kappa''}_s
					\end{aligned}\end{equation*}
					Using the fact that image of $\calY^{\kappa'',\kappa'}_s\times_{\calB^{\kappa'}}\calY^{\kappa',\kappa''}_s \to \calY^{\kappa'',\kappa''}$ lies in $\calY^{\kappa'',\kappa''}_e\cup\calY^{\kappa'',\kappa''}_s$, we deduce 
					\begin{equation*}\begin{aligned}
						\calY^{\kappa,\kappa''}_{ws}\times_{\calB^{\kappa''}}\calY^{\kappa'',\kappa'}_s\times_{\calB^{\kappa'}}\calY^{\kappa',\kappa''}_s\to \calY^{\kappa,\kappa''}_{ws}\times_{\calB^{\kappa''}}\left(\calY^{\kappa'',\kappa''}_e\cup \calY^{\kappa'',\kappa''}_s\right)\cong \calY^{\kappa,\kappa''}_{ws}\cup \calY^{\kappa,\kappa''}_{w}.
					\end{aligned}\end{equation*}
			\end{enumerate}
			Thus, the image of~\eqref{equa:convY} lies in $\calY^{\kappa,\kappa''}_{ws}\cup \calY^{\kappa,\kappa''}_w$ in both cases.
	\end{enumerate}
	Now we show that if $\calI\subset \calW$ is an ideal and $\kappa, \kappa'\in \frakA$ are alcoves, then $\calY^{\kappa,\kappa'}_{\calI} = \Pi_{\kappa,\kappa'}^{-1}(\calI)$ is a proper algebraic variety. Let $y\in \calI$ be such that $\ubar \kappa' = y^{-1}\ubar \kappa$ and pick a reduced decomposition $y = s_1\cdots s_r$. For each $i$, put $\kappa_i = s_{i} \cdots s_1 \kappa\in \frakA$ and consider $D_i = \calY^{\kappa_{i-1},\kappa_{i}}_{e}\cup \calY^{\kappa_{i-1},\kappa_{i}}_{s_i}$.
	\begin{enumerate}
		\item[\ul{Step 2}.]
			We show that $D_i$ is a proper variety for each $i$. 
			\begin{enumerate}
				\item[Case 1.]
					If $\ubar \kappa_{i-1} \neq \ubar \kappa_i$, then $\calY_e^{\kappa_{i-1},\kappa_{i}} = \emptyset$ and 
					\[
						\calY^{\kappa_{i-1},\kappa_{i}}_{s_i} = \Gz\cdot \left( P^{\kappa_{i-1}}_0, P^{\kappa_{i}}_0 \right) \subset \calY^{\kappa_{i-1},\kappa_{i}}.
					\]
					However, since $\kappa_{i-1}$ and $\kappa_{i}$ are adjacent in $\bbA$, the condition $\ubar\kappa_{i-1} \neq \ubar\kappa_{i}$ implies that the hyperplane separating $\kappa_{i-1}$ and $\kappa_{i}$ does not contain $\bfx\in \bbA$. Since by~\autoref{coro:WGz}, the root system $R(\Gz, T)$ can be identified with the set of affine roots which vanish at $\bfx$, we see that $P^{\kappa_{i-1}}_0 = P^{\kappa_i}_0$ as subgroup of $G_{\ubar 0}$. Therefore, we have $\Stab_{G_{\ubar 0}}\left( P^{\kappa_{i-1}}_0, P^{\kappa_{i}}_0 \right) = P^{\kappa_{i-1}}_0$. It follows that $D_i = \calY^{\kappa_{i-1}, \kappa_{i}}_{s_i}\cong G_{\ubar 0} / P^{\kappa_{i-1}}_0$ is a proper variety. 
				\item[Case 2.]
					If $\ubar \kappa_{i-1} = \ubar \kappa_i$, then $s_{i}^{\kappa_{i-1}}$ lies in $W_{\bfx}$, so there is an isomorphism $\calY^{\kappa_{i-1}, \kappa_i}\cong \calY^{\kappa_{i}, \kappa_i}$. Let $\dot s^{\kappa_i}_i\in N_{\Gz}(T)$ be a lifting of $s_i^{\kappa_{i-1}}$. Then $D_i$ can be described by
					\begin{equation*}\begin{aligned}
						D_i \cong \calY^{\kappa_i, \kappa_i}_{e}\cup \calY^{\kappa_i, \kappa_i}_{s_i} = \Gz\cdot \left( P^{\kappa_i}_0, P^{\kappa_i}_0 \right)\cup \Gz\cdot \left( P^{\kappa_i}_0, \dot s^{\kappa_i}_i P^{\kappa_i}_0 \right).
					\end{aligned}\end{equation*}
					We see that the projection $D_i \xrightarrow{\pr_1} \calB^{\kappa_i}$ is a $\bfP^1$-bundle. Hence $D_i$ is proper. 
			\end{enumerate}
		\item[\ul{Step 3}.]
			Now consider the convolution product 
			\begin{equation*}\begin{aligned}
				D_1\times_{\calB^{\kappa_1}}\cdots \times_{\calB^{\kappa_{r-1}}}D_r &\to \calY^{\kappa_0,\kappa_r} \\
				\left( g_1P^{\kappa_0}_0,  g_2P^{\kappa_1}_0, \ldots, g_{r+1}P^{\kappa_r}_0 \right) &\mapsto \left( g_1P^{\kappa_0}_0, g_{r+1}P^{\kappa_r}_0 \right)
			\end{aligned}\end{equation*}
			We show by induction on $i$ that the image of $D_1\times_{\calB^{\kappa_1}}\cdots\times_{\calB^{\kappa_{i-1}}}D_i\to \calY^{\kappa_0,\kappa_{i}}$ contains $\calY^{\kappa_0,\kappa_{i}}_{s_1\cdots s_i}$ and lies in $\calY^{\kappa_0,\kappa_{i}}_{\le s_1\cdots s_i}$. It is trivial for $i=0$. Suppose $i > 0$ and that the statement is proven for $i-1$, so that we have $D_1\times_{\calB^{\kappa_1}}\cdots \times_{\calB^{\kappa_{i-2}}}D_{i-1}\to \calY^{\kappa_0,\kappa_{i-1}}_{\le s_1 \cdots s_{i-1}}$. Applying Step 1, we have 
			\[
				\calY^{\kappa_0,\kappa_{i-1}}_{\le s_1 \cdots s_{i-1}}\times_{\calB^{\kappa_{i-1}}}D_i = \bigcup_{w'\le s_1 \cdots s_{i-1}}\calY^{\kappa_0,\kappa_{i-1}}_{w'}\times_{\calB^{\kappa_{i-1}}}D_i\to \bigcup_{w'\le s_1 \cdots s_{i-1}}(\calY^{\kappa_0,\kappa_{i}}_{w'}\cup \calY^{\kappa_0,\kappa_{i}}_{w's_i}).
			\]
			Since $w'\le s_1 \cdots s_{i-1}$ implies $w' \le s_1 \cdots s_i$ and $w's_i \le s_1 \cdots s_i$, the image of the composition
			\[
				D_1\times_{\calB^{\kappa_1}}\cdots \times_{\calB^{\kappa_{i-1}}}D_i \to \calY^{\kappa_0,\kappa_{i-1}}_{\le s_1 \cdots s_{i-1}}\times_{\calB^{\kappa_{i-1}}}D_i\to \calY^{\kappa_0,\kappa_i}
			\]
			lies in $\calY^{\kappa_0,\kappa_i}_{\le s_1 \cdots s_i}$. The case 1 of Step 1 also implies that $\calY^{\kappa_0, \kappa_i}_{s_1 \cdots s_i}$ lies in the image.
		\item[\ul{Step 4}.]
			Since $D_1\times_{\calB^{\kappa_1}}\cdots \times_{\calB^{\kappa_{r-1}}}D_r$ is a proper variety, its image in $\calY^{\kappa_0, \kappa_r}_{\le s_1\cdots s_r}$ is also proper. We see that, $\calY^{\kappa_0,\kappa_r}_y$ is contained in a $G_{\ubar 0}$-invariant proper subvariety of $\calY^{\kappa_0,\kappa_r}_{\le y}$. Letting $y$ run over all elements of $\calI$ with $y\ubar\kappa' = \ubar \kappa$, we see that $\calY^{\kappa,\kappa'}_{\calI}$ is covered by a finite number of proper subvarieties of $\calY^{\kappa,\kappa'}$. Hence $\calY^{\kappa,\kappa'}_{\calI}$ is proper.  \par
		\item[\ul{Step 5}.]
			Applying Step 1, standard arguments of $(B, N)$-pairs show that the image of 
			\[
				\calY^{\kappa,\kappa'}_{\calI}\times_{\calB^{\kappa'}}\times \calY^{\kappa',\kappa''}_{\calJ}\to \calY^{\kappa,\kappa''}
			\]
			is contained in $\calY^{\kappa,\kappa'}_{\calI\calJ}$. 
		\item[\ul{Step 6}.]
			We prove the statement~\ref{lemm:XI-i}. For $\nu,\nu'\in \Xi$, pick alcoves $\kappa, \kappa'\in \frakA$ such that $\partial_I \kappa = \nu$ and $\partial_I \kappa' = \nu'$. The projection $f:\calY^{\kappa,\kappa'}\to \calX^{\nu,\nu'}$ satisfies the property that $f^{-1}(\calX^{\nu,\nu'}_w) = \bigcup_{y\in [w]}\calY^{\kappa,\kappa'}_y$ for each $[w]\in \calW_I\backslash\calW / \calW_I$. Since $\calI$ is an ideal of $(\calW_I\backslash \calW / \calW_I, \le_{\max})$, by~\autoref{lemm:CC}, the set $\til\calI = \left\{ w\in \calW\;;\; [w]\in \calI\right\}$ is an ideal of $\calW$. Thus  $f^{-1}(\calX^{\nu,\nu'}_{\calI}) = \calY^{\kappa,\kappa'}_{\til\calI}$ is proper by Step 4 and so is the image $\calX^{\nu,\nu'}_{\calI} = f(\calY^{\kappa,\kappa'}_{\til\calI})$ proper. Hence $\calX^{\nu,\nu'}_{\calI}$ is closed in $\calX^{\nu,\nu'}$. \par
			Similarly, one can prove~\ref{lemm:XI-ii} by descending the corresponding statement about $\calY$ in Step 5. 
	\end{enumerate}

\end{proof}

\subsection{Good strata of \texorpdfstring{$\calX^{\nu, \nu'}$}{Xνν'}}
For general $[w]\in \calW_I\backslash\calW / \calW_I$, the stratum $\calX^{\nu,\nu'}_w = \Pi^{-1}_{\nu,\nu'}[w]$ may not be a single $G_{\ubar 0}$-orbit. However, the following lemma shows that it is the case for those strata $\calX^{\nu,\nu'}_w$ with $w\in \calW_{\xi}$ and $w\ubar\nu' = \ubar \nu$. As we will see in~\autoref{sec:convolution}, only these strata will have contribution to the extension algebra $\ha\calH$. 

\begin{lemm}\label{lemm:uniqueorbit}
	Let $[w]\in \calW_{I}\backslash \calW / \calW_{I}$ and $\nu,\nu'\in \Xi$.
	\begin{enumerate}[label=(\roman*)]
		\item\label{lemm:uniqueorbit-i}
			The stratum $\calX^{\nu,\nu'}_w$ contains a finite number of $G_{\ubar 0}$-orbits. 
		\item\label{lemm:uniqueorbit-ii}
			If $w\in \calW_{\xi}$, then we have $\#\left(G_{\ubar 0}\backslash \Pi_{\nu, \nu'}^{-1}\left[ w \right] \right)\le 1 $. Moreover, the equality holds if and only if $\ubar{\nu} = w\ubar{\nu'}$.
	\end{enumerate}
\end{lemm}
\begin{proof}
	The assertion \ref{lemm:uniqueorbit-i} follows from the fact that $G_{\ubar 0}\backslash \calX^{\nu,\nu'}$ is finite. \par
	We prove~\ref{lemm:uniqueorbit-ii}. Suppose that $w\in \calW_{\xi}$ and $\nu,\nu'\in \Xi$ such that $\ubar\nu = w\ubar\nu'$. We may assume that $\nu =  w\nu'$ since $\calX^{\nu, \nu'}$ and $\Pi_{\nu, \nu'}$ depends only on the $W_{\xi,\bfx}$-orbits $\ubar\nu$ and $\ubar\nu'$. Let $w^{\nu}\in W_{\xi}$ be the image of $w$ under the isomorphism $\left( \calW_{\xi}, \Delta_{\xi} \right)\cong\left( W_{\xi}, \Delta_{\xi}^{\nu} \right)$ so that $\nu' = w^{\nu}\nu$ \cfauto{subsec:relpos}. It suffices to show that the map 
	\begin{equation*}\begin{aligned}
		\varphi: W_{\bfx}\backslash \left(\left(W_{\bfx} / W_{\bfx, \nu}\right) \times \left(W_{\bfx} / W_{\bfx, \nu'}\right)\right) &\to W\backslash \left(\frakF_{I} \times \frakF_{I}\right)\\
		W_{\bfx}\cdot (u, u') &\mapsto W\cdot (u \nu, u'w^{\nu}\nu)
	\end{aligned}\end{equation*}
	is injective over $W\cdot\left(\nu, w^{\nu}\nu  \right)\in W\backslash \left(\frakF_I \times \frakF_I\right)$.
	Assume we have $u, u'\in W_{\bfx}$ such that $\left( uW_{\bfx, \nu}, u'W_{\bfx, \nu'} \right)$ is sent to $W\cdot\left(\nu, w^{\nu}\nu  \right)$, in other words
	\begin{equation*}\begin{aligned}
		W\cdot\left( u\nu, u'w^{\nu}\nu \right) =  W\cdot\left( \nu, w^{\nu} \nu \right).
	\end{aligned}\end{equation*}
	We shall prove that $W_{\bfx}\cdot\left( u\nu, u'w^{\nu}\nu \right)= W_{\bfx}\cdot\left( \nu, w^{\nu}\nu \right)$. \par

	Indeed, as
	\begin{equation*}\begin{aligned}
		W\cdot\left( \nu, w^{\nu} \nu \right)= W\cdot\left( u\nu, u'w^{\nu} \nu \right)= W\cdot\left( \nu, u^{-1}u'w^{\nu} \nu \right)
	\end{aligned}\end{equation*}
	there exists $q\in W_{\nu} = W_M$ such that $qw^{\nu}\nu = u^{-1}u'w^{\nu}\nu$, or equivalently 
	\begin{equation*}\begin{aligned}
		\left(w^{\nu}\right)^{-1}q^{-1}u^{-1}u'w^{\nu}\in W_{\nu}.
	\end{aligned}\end{equation*}
	Since $w^{\nu}$ normalises $W_{\nu}$, we obtain $(w^{\nu})^{-1}u^{-1}u'w^{\nu}\in W_{\nu}$ and hence 
	\begin{equation*}\begin{aligned}
		W_{\bfx}\cdot\left(u\nu, u' w^{\nu}\nu  \right)&= W_{\bfx}\cdot\left(\nu, u^{-1}u' w^{\nu}\nu  \right) = W_{\bfx}\cdot\left(\nu, w^{\nu}((w^{\nu})^{-1}u^{-1}u'w^{\nu})\nu  \right) = W_{\bfx}\cdot(\nu,w^{\nu}\nu),
	\end{aligned}\end{equation*}
	which concludes the proof.
\end{proof}

\begin{lemm} \label{prop:strX}
	For each $\nu, \nu'\in \Xi$ and each $w\in \calW_{\xi}$ such that $w\ubar{\nu'} = \ubar{\nu}$, there is an isomorphism of $\Gztq$-schemes
				\begin{equation*}\begin{aligned}
								\begin{tikzcd}[row sep=.5em]
									G_{\ubar 0} / P_0^{\nu}\cap P_0^{w^{-1}\nu} \arrow{r}{\cong} &  \calX_w^{\nu, w^{-1}\nu}\arrow{r}{\cong} & \calX_w^{\nu, \nu'}  \\
									g\cdot\left(P_0^{\nu}\cap P_0^{ w^{-1}\nu}\right) \arrow[mapsto]{r}& \left(gP_0^{\nu}, gP_0^{w^{-1}\nu}\right) & 
								\end{tikzcd}.
				\end{aligned}\end{equation*}
\end{lemm}
\begin{proof}
By~\autoref{lemm:uniqueorbit}, $\calX_w^{\nu, \nu'}$ is a $G_{\ubar 0}$-orbit. The proposition results from the fact that the $\Gz$-action on $\calX^{\nu,w^{-1}\nu}_w$ satisfies
	\[
		\Stab_{G_{\ubar 0}} \left( P_0^{\nu}, P_0^{w^{-1}\nu} \right) = P_0^{\nu}\cap P_0^{w^{-1}\nu }.\qedhere
	\]
\end{proof}
There is a canonical bijection 
\[
	P^{\nu}_0\backslash G_{\ubar 0}/P^{\nu'}_0\cong G_{\ubar 0}\backslash \calX^{\nu, \nu'},\quad P^{\nu}_0gP^{\nu'}_0\mapsto \Gz\cdot(P^{\nu}_0, gP^{\nu'}_0).
\]
The following technical lemma will be used in the proof of~\autoref{prop:SS-bimod} to show that bad orbits do not contribute to the extension algebra $\ha\calH$.
\begin{prop}\label{prop:gooddoublecosets}
	Let $\Omega\in P^{\nu}_0\backslash G_{\ubar 0}/P^{\nu'}_0$ be a double coset and let $O_{\Omega}\in G_{\ubar 0}\backslash \calX^{\nu, \nu'}$ be the corresponding orbit. Then $\Pi_{\nu, \nu'}\left(O_{\Omega}\right) \in \calW_{\xi}$ if and only if for any $g\in \Omega$, the following natural inclusion 
	\begin{equation*}\begin{aligned}
		\left(\frakp^{\nu}_{N}\cap \Ad_{g}\frakp^{\nu'}_N\right) / \left(\fraku^{\nu}_{N}\cap \Ad_{g}\fraku^{\nu'}_N\right) \to \Ad_{g}\frakp^{\nu'}_N / \Ad_{g}\fraku^{\nu'}_N
	\end{aligned}\end{equation*}
	is an isomorphism for each $N\in \bfZ$. 
\end{prop}
\begin{proof}
	Let $g\in \Omega$. According to~\cite[5.1]{LYI}, there is a splitting $\frakl_*$ of $\frakp^{\nu}_*$ and a splitting $\frakl'_*$ of $\Ad_g\frakp^{\nu'}_*$ such that $L^{\nu}_0 = \exp\left( \frakl_0 \right)$ and $L^{\nu'}_0 = \exp\left( \frakl'_0 \right)$ contain a common maximal torus $T'$ of $G_{\ubar 0}$. \par
	Let $g'\in G_{\ubar 0}$ be such that $g'T'g'^{-1} = T$. Choose any alcove $\kappa\in \frakA$ such that $\partial_I\kappa = \nu$. Then $\Ad_{g'}\frakp^{\nu}_*$ and $\Ad_{g'g}\frakp^{\nu'}_*$ are in $\frakP_T$ and there exist $y, w\in \calW$ such that $\Ad_{g'}\frakp^{\nu}_* = \frakp^{y^{\kappa}\nu}_*$ and $\Ad_{g'g}\frakp^{\nu'}_* = \frakp^{y^{\kappa}w^{\kappa}\nu}_*$. From the definition of $\Pi_{\nu, \nu'}$, we see that $\Pi_{\nu, \nu'}\left( O_{\Omega} \right) = [w]$. Moreover,  $[w]\in \calW_{\xi}$ if and only if $y^{\kappa}\nu$ and $y^{\kappa}w^{\kappa}\nu$ span the same affine subspace, which is equivalent to that $\frakl^{y^{\kappa}\nu}_* = \frakl^{y^{\kappa}w^{\kappa}\nu}_*$. Using the Levi decomposition $\frakp^{y^{\kappa}\nu}_* = \frakl^{y^{\kappa}\nu}_*\oplus\fraku^{y^{\kappa}\nu}_*$ and $\frakp^{y^{\kappa}w^{\kappa}\nu}_* =  \frakl^{y^{\kappa}w^{\kappa}\nu}_*\oplus\fraku^{y^{\kappa}w^{\kappa}\nu}_*$, we see that $\frakl^{y^{\kappa}\nu}_* = \frakl^{y^{\kappa}w^{\kappa}\nu}_*$ if and only if  the obvious map
	\begin{equation}\begin{aligned}\label{equa:condition1}
		\left(\frakp^{y^{\kappa}\nu}_{N}\cap \frakp^{y^{\kappa}w^{\kappa}\nu}_N\right) / \left(\fraku^{y^{\kappa}\nu}_{N}\cap \fraku^{y^{\kappa}w^{\kappa}\nu}_N\right) \to \frakp^{y^{\kappa}w^{\kappa}\nu}_N / \fraku^{y^{\kappa}w^{\kappa}\nu}_N
	\end{aligned}\end{equation}
	is an isomorphism for all $N\in \bfZ$. Taking into account that $\Ad_{g'}\frakp^{\nu}_* = \frakp^{y^{\kappa}\nu}_*$ and $\Ad_{g'g}\frakp^{\nu'}_* = \frakp^{y^{\kappa}w^{\kappa}\nu}_*$, the condition that the map~\eqref{equa:condition1} is an isomorphism is equivalent to that the map
	\begin{equation}\begin{aligned}\label{equa:condition2}
		\left(\Ad_{g'}\frakp^{\nu}_{N}\cap \Ad_{g'g}\frakp^{\nu'}_N\right) / \left(\Ad_{g'}\fraku^{\nu}_{N}\cap \Ad_{g'g}\fraku^{\nu'}_N\right) \to \Ad_{g'g}\frakp^{\nu'}_N / \Ad_{g'g}\fraku^{\nu'}_N
	\end{aligned}\end{equation}
	is an isomorphism for all $N\in \bfZ$. We obtain the desired equivalent condition by applying $\Ad_{g'^{-1}}$ to the map~\eqref{equa:condition2}.
\end{proof}

\begin{rema}\label{rema:goodbad}
	In the terminology of~\cite[5.2]{LYI}, if a double coset $\Omega\in P_0^{\nu}\backslash G_{\ubar 0} / P_0^{\nu'}$ satisfies the condition of~\autoref{prop:gooddoublecosets}, it is called {\bf good}. It is called {\bf bad} otherwise. 
\end{rema}

\subsection{Steinberg varieties}\label{subsec:steinberg}
For each $\nu\in \Xi$, we introduce a variety
\begin{equation*}\begin{aligned}
	\calT^{\nu} = G_{\ubar 0}\times^{P_0^{\nu}} \left( \rmO_{\eta}\oplus \fraku^{\nu}_{\eta} \right).
\end{aligned}\end{equation*}
It is a smooth variety equipped with an action of $\Gztq$, where the factor $\Gz$ acts by multiplication on the left, and is, up to $\Gztq$-equivariant isomorphism, independent of the choice of $\nu$ in its $W_{\bfx}$-conjugacy class. There is a stack morphism
\[
	f:[\calT^{\nu} / \Gztq] \cong [  \rmO_{\eta}\oplus \fraku^{\nu}_{\eta} / P^{\nu}_{0,\diamond,q} ] \to \left[ \rmO_{\eta} / \Mztq \right].
\]
Let $\scrC_{\eta} = \scrC\mid_{\rmO_{\eta}}$. Define $\dot\scrC_{\nu} = f^*\scrC_{\eta}$. It is a $\Gztq$-equivariant local system on $\calT^{\nu}$. If we put 
\[
	a^{\nu}:\calT^{\nu} = G_{\ubar 0}\times^{P_0^{\nu}} \left( \rmO_{\eta}\oplus \fraku^{\nu}_{\eta} \right)\to \frakg_{\ubar\eta}, \qquad a^{\nu}(g, z) = \Ad_g z,
\]
then we recover the induced complex $a^{\nu}_* \dot\scrC_{\nu} = \Ind_{\nu}\scrC = \bfI^{\nu}$ introduced in~\autoref{subsec:indcusp}.  \par
\index{T@$\calT^{\nu}$}\index{C@$\dot\scrC_{\nu}$}
For $\nu, \nu'\in \Xi$, we take the fibred product of $a^{\nu}$ and $a^{\nu'}$:
\begin{equation*}\begin{aligned}
	\calZ^{\nu, \nu'} = \calT^{\nu} \times_{\frakg_{\ubar \eta}} \calT^{\nu'}.
\end{aligned}\end{equation*} \index{Z@$\calZ^{\nu, \nu'}$} 
There is a canonical $\Gztq$-equivariant map which forgets the Lie algebra component
\begin{equation*}\begin{aligned}
	g:\calZ^{\nu, \nu'} \to \calX^{\nu, \nu'}.
\end{aligned}\end{equation*}
For $w\in \calW$, we set $\calZ^{\nu, \nu'}_w = g^{-1}(\calX^{\nu,\nu'}_w)$. \index{Z@$\calZ^{\nu, \nu'}_w$} Similarly, for any ideal $\calI$ of $\left( \calW_I\backslash \calW / \calW_I, \le_{\max} \right)$, we put $\calZ^{\nu,\nu'}_{\calI} = g^{-1}(\calX^{\nu,\nu'}_{\calI})$.\index{Z@$\calZ^{\nu, \nu'}_{\calI}$} The following lemma summarises the geometric properties of the Steinberg varieties $\calZ^{\nu,\nu'}$ that are necessary for our analysis of the extension algebra $\ha\calH$. 
\begin{prop} \label{prop:strZ}
	Let $\nu, \nu',\nu''\in \Xi$.
	\begin{enumerate}
		\item\label{prop:strZ-i}
			For every ideal $\calI\subset \calW_I\backslash \calW / \calW_I$, the subvariety $\calZ^{\nu,\nu'}_{\calI}\subset \calZ^{\nu,\nu'}$ is closed.
		\item\label{prop:strZ-ii}
			For every pair of ideals $\calI,\calJ\subset \calW_I\backslash \calW / \calW_I$, the image of the projection
			\begin{equation*}\begin{aligned}
				\calZ^{\nu,\nu'}_{\calI}\times_{\calT^{\nu'}}\calZ^{\nu',\nu''}_{\calJ}\to \calZ^{\nu,\nu''}
			\end{aligned}\end{equation*}
			lies in $\calZ^{\nu,\nu''}_{\calI\calJ}$. 
		\item\label{prop:strZ-iii}
			For every $w\in \calW_{\xi}$ such that $w\ubar{\ubar\nu'} = \ubar{\ubar\nu}$, there is an equivariant isomorphism of $\Gztq$-schemes
				\begin{equation*}\begin{aligned}
								\begin{tikzcd}[row sep=.5em]
									G_{\ubar 0} \times^{P_0^{\nu}\cap P_0^{w^{-1}\nu}}\left(\rmO_{\eta}\oplus \left(\fraku_{\eta}^{\nu}  \cap \fraku_{\eta}^{w^{-1}\nu}\right)\right) \arrow{r}{\cong} &  \calZ_w^{\nu, w^{-1}\nu}\arrow{r}{\cong} & \calZ_w^{\nu, \nu'}  \\
												(g, x) \arrow[mapsto]{r}& \left( \left( g , x\right), \left(g, x  \right) \right) & 
								\end{tikzcd}.
				\end{aligned}\end{equation*}
	\end{enumerate}
\end{prop}
\begin{proof}
	The first two statements follow from~\autoref{lemm:XI} and the last one follows from~\autoref{prop:strX}.
\end{proof}

\section{Convolution algebra}\label{sec:convolution}
We keep the assumptions of the previous sections. \par

In this section, we study the extension algebra $\ha\calH$ in terms of convolution algebras. We deduce some properties of convolution algebras from the geometry of Steinberg varieties $\calZ^{\nu,\nu'}$ studied in the previous section.

\subsection{Convolution product}\label{subsec:convolution}
Let $\nu, \nu'\in \Xi$. We have two projections $q_1:\calZ^{\nu, \nu'}\to \calT^{\nu}$ and $q_2:\calZ^{\nu, \nu'}\to \calT^{\nu'}$. We set
\begin{equation*}\begin{aligned}
				\calH^{\nu, \nu'} = \Ext^*_{\Gztq}\left( \bfI^{\nu'}, \bfI^{\nu} \right).
\end{aligned}\end{equation*}
\index{H@$\calH^{\nu, \nu'}$} We put $\scrK = \scrHom(q_2^*\dot\scrC_{\nu'}, q_1^!\dot\scrC_{\nu})$. The Verdier duality yields
\begin{equation*}\begin{aligned}
	\calH^{\nu,\nu'}\cong \Ext^*_{\Gztq}\left(q_2^* \dot\scrC_{\nu'}, q_1^! \dot\scrC_{\nu}  \right) \cong  \rmH^*_{\Gztq}(\calZ^{\nu,\nu'}, \scrK).
\end{aligned}\end{equation*}
\index{K@$\scrK$} For the reason explained in~\autoref{subsec:indcusp}, the definition of $\calH^{\nu, \nu'}$ is, up to canonical isomorphism, independent of $\nu$ and $\nu'$ in their respective $W_{\bfx}$-conjugacy classes. \par
By the formalism of~\cite{lusztig95} and~\cite{EM}, for $\nu, \nu', \nu''\in \Xi$, the Yoneda product
\begin{equation*}\begin{aligned}
				\calH^{\nu, \nu'}\otimes \calH^{\nu', \nu''} \to \calH^{\nu, \nu''}
\end{aligned}\end{equation*}
can be described in terms of $\calZ^{\nu,\nu'}$ and $\scrK$. It is called the {\itshape convolution product}. One subtlety of it is that the maps $q_1$ and $q_2$ are not proper and {\itshape a priori} cohomological convolution product cannot be defined for non-proper maps. \par

We recall the construction of this convolution product. Since the orbit $\rmO$ is distinguished, $\rmO_{\eta}$ must be open and dense in $\frakm_{\eta}$ {\it cf.}~\cite[4.4a]{lusztig95b}. We set $\ha\calT^{\nu} = G_{\ubar 0}\times^{P_0^{\nu}}\frakp^{\nu}_{\eta}$ and $\ha\calZ^{\nu, \nu'} = \ha\calT^{\nu}\times_{\frakg_{\ubar \eta}} \ha\calT^{\nu}$. Let $u^{\nu}:\calT^{\nu}\hookrightarrow\ha\calT^{\nu}$ and $u^{\nu,\nu'}:\calZ^{\nu,\nu'}\hookrightarrow\ha\calZ^{\nu,\nu'}$ be the open embeddings. We remark that $\ha\calT^{\nu}$ is smooth and is proper over $\frakg_{\ubar \eta}$. Let $q_1: \calZ^{\nu, \nu'}\to \calT^{\nu}$, $q_2: \calZ^{\nu, \nu'}\to \calT^{\nu'}$, $\ha q_1: \ha\calZ^{\nu, \nu'}\to \ha\calT^{\nu}$ and $\ha q_2: \ha\calZ^{\nu, \nu'}\to \ha\calT^{\nu'}$ be the canonical projections. \par
By the cleanness of $\scrC_{\eta}$ (\cite[\S4]{lusztig95b}), we have $u^{\nu}_!\dot\scrC_{\nu} = u^{\nu}_*\dot\scrC_{\nu}$. Put $\ha\scrK = \scrHom\left( \ha q_2^* u^{\nu'}_*\dot\scrC_{\nu'}, \ha q_1^! u^{\nu}_*\dot\scrC_{\nu} \right)$, which is a complex on $\ha\calZ^{\nu,\nu'}$. We have
\begin{equation*}\begin{aligned}
	\ha\scrK &\cong \scrHom\left( \ha q_2^* u^{\nu'}_!\dot\scrC_{\nu'}, \ha q_1^! u^{\nu}_*\dot\scrC_{\nu} \right) \cong\scrHom\left( u^{\nu,\nu'}_!q_2^*\dot\scrC_{\nu'}, u^{\nu,\nu'}_*q_1^! \dot\scrC_{\nu} \right)\\
	&\cong u^{\nu,\nu'}_*\scrHom\left( q_2^*\dot\scrC_{\nu'}, q_1^! \dot\scrC_{\nu} \right)\cong u^{\nu,\nu'}_*\scrK.
\end{aligned}\end{equation*}
Therefore 
\begin{equation*}\begin{aligned}
				\calH^{\nu, \nu'} = \rmH^*_{\Gztq}\left( \calZ^{\nu, \nu'}, \scrK \right)\cong \rmH^*_{\Gztq}\left( \ha\calZ^{\nu, \nu'}, \ha\scrK \right). 
\end{aligned}\end{equation*}
It follows that the cleanness of the cuspidal local system $\scrC$ assures that we can safely work on the non-proper version $\calT^{\nu}$, $\calZ^{\nu,\nu'}$ and $\scrK$. 
\par

 Now we describe the convolution product. Given $\nu, \nu', \nu''\in \Xi$, we consider the following diagram
\begin{equation*}\begin{aligned}
				\begin{tikzcd}[column sep=5em]
					\calZ^{\nu, \nu'} \times \calZ^{\nu', \nu''}\arrow{d}{t} \arrow[hookleftarrow]{r}&  \calZ^{\nu, \nu'} \times_{\calT^{\nu'}} \calZ^{\nu', \nu''} \arrow{r}{\mu} \arrow{d}{r} & \calZ^{\nu, \nu''}\arrow{d}{s} \\
					\left(\calT^{\nu}\times \calT^{\nu'}\right)\times\left(\calT^{\nu'}\times \calT^{\nu''}\right) & \calT^{\nu}\times \calT^{\nu'}\times \calT^{\nu''} \arrow[swap]{l}{\gamma} \arrow{r}{u} &  \calT^{\nu} \times \calT^{\nu''}
				\end{tikzcd}
\end{aligned}\end{equation*}
There is a sequence of maps
\begin{equation*}\begin{aligned}
	&\rmH^*_{\Gztq}\left( \calZ^{\nu, \nu'}, \scrK \right)\otimes \rmH^*_{\Gztq}\left( \calZ^{\nu', \nu''}, \scrK \right) \\
	&\to\rmH^*_{\Gztq}\left(\calZ^{\nu, \nu'}\times\calZ^{\nu', \nu''}, t^!\left(\dot\scrC_{\nu}\boxtimes\bfD\dot\scrC_{\nu'}\boxtimes\dot\scrC_{\nu'}\boxtimes\bfD\dot\scrC_{\nu''}\right) \right) \\
	& \xrightarrow{t^!\left(\id\to \gamma_*\gamma^*\right)}
	\rmH^*_{\Gztq}\left(\calZ^{\nu, \nu'}\times\calZ^{\nu', \nu''}, t^!\gamma_*\left(\dot\scrC_{\nu}\boxtimes\left(\bfD\dot\scrC_{\nu'}\otimes\dot\scrC_{\nu'}\right)\boxtimes\bfD\dot\scrC_{\nu''}\right) \right) \\
	&\cong \rmH^*_{\Gztq}\left(\calZ^{\nu, \nu'}\times_{\calT^{\nu'}}\calZ^{\nu', \nu''}, r^!\left(\dot\scrC_{\nu}\boxtimes\left(\bfD\dot\scrC_{\nu'}\otimes\dot\scrC_{\nu'}\right)\boxtimes\bfD\dot\scrC_{\nu''}\right) \right) \\
	&\xrightarrow{\bfD \dot\scrC_{\nu}\otimes \dot\scrC_{\nu}\to \bfD \bfk}\rmH^*_{\Gztq}\left(\calZ^{\nu, \nu'}\times_{\calT^{\nu'}}\calZ^{\nu', \nu''}, r^!\left(\dot\scrC_{\nu}\boxtimes\bfD\bfk\boxtimes\bfD\dot\scrC_{\nu''}\right) \right)\\
	&\cong \rmH^*_{\Gztq}\left(\calZ^{\nu, \nu'}\times_{\calT^{\nu'}}\calZ^{\nu', \nu''}, r^!u^!\left(\dot\scrC_{\nu}\boxtimes\bfD\dot\scrC_{\nu''}\right) \right) \\
	&\cong \rmH^*_{\Gztq}\left(\calZ^{\nu, \nu'}\times_{\calT^{\nu'}}\calZ^{\nu', \nu''}, \mu^!\scrK\right)\to \rmH^*_{\Gztq}\left( \calZ^{\nu, \nu''}, \scrK \right).
\end{aligned}\end{equation*}
The convolution product $\calH^{\nu, \nu'}\otimes\calH^{\nu', \nu''} \to\calH^{\nu, \nu''}$ is then defined to be the composite. \par

Later on, it will be crucial for us to be able to analyse the convolution algebra by the Bruhat filtration $\calZ^{\nu,\nu'}_{\calI}$ introduced in~\autoref{subsec:steinberg}. Let $\calI, \calJ$ be ideals of $\calW_{I}\backslash \calW/\calW_I$. By~\autorefitem{prop:strZ}{ii}, we have the map 
\[
	\calZ^{\nu,\nu'}_{\calI}\times_{\calT^{\nu'}}\calZ^{\nu',\nu''}_{\calJ}\to \calZ^{\nu,\nu''}_{\calI\calJ}.
\]
We define similarly the convolution product of them so that the following diagram commutes
\begin{equation*}\begin{aligned}
	\begin{tikzcd}
		\rmH^*_{\Gztq}\left( \calZ^{\nu,\nu'}_{\calI}, i_{\calI}^!\scrK \right) \otimes \rmH^*_{\Gztq}\left( \calZ^{\nu',\nu''}_{\calJ}, i_{\calJ}^!\scrK  \right) \arrow{r}\arrow{d}&  \rmH^*_{\Gztq}\left( \calZ^{\nu,\nu''}_{\calI\calJ}, i_{\calI\calJ}^!\scrK \right) \arrow{d}\\
		\calH^{\nu, \nu'}\otimes\calH^{\nu', \nu''}\arrow{r} & \calH^{\nu, \nu''} \\
	\end{tikzcd},
\end{aligned}\end{equation*}
where the vertical arrows are given by adjunction co-units $i_{\calI!}i_{\calI}^!\to \id$. 
\subsection{Polynomial action}
We define the polynomial algebra $\bfS^{M}= \rmH^*_{\Mztq}(\rmO_{\eta}, \bfk)$\index{S@$\bfS^{M}$}, graded by the cohomological degree. 
\begin{lemm}\label{lemm:SE}
	There is a graded ring isomorphism
	\[
		\bfS^{M}\cong \bfk[\bbE^M_{\diamond}]\otimes \bfk[u],
	\]
	where the linear functions $(\bbE^M_{\diamond})^*\subset \bfk[\bbE^M_{\diamond}]$ and $u\in \bfk[u]$ are of degree $2$.
\end{lemm}
\begin{proof}
	If we choose $e\in \rmO_{\eta}$ and complete it into a $\fraksl_2$-triple $\phi = (e, h, f)$ as in~\autoref{subsec:JM}, then we have
	\[
		\bfS^{M}= \rmH^*_{\Mztq}(\rmO_{\eta}, \bfk)\cong \rmH^*_{Z_{\Mztq}(e)}.
	\]
	As the orbit $\rmO$ support a cuspidal local system, by~\cite[4.2,~4.3]{lusztig95b} we have $Z_{\Mtq}(e)= Z_{\Mztq}(e)$. With the group $Z^q_{\Mt}(\phi)$ as in~\eqref{equa:Mphi}, we have an isomorphism
	\[
		\iota: Z_{\Mt}(\phi)\times \bfC^{\times}\cong Z^q_{\Mt}(\phi),\quad \iota(g, q) = (gq^h, q).
	\]
	Since $Z_{\Mtq}(e)\subset Z^q_{\Mt}(\phi)$ is a maximal reductive subgroup~\cite[2.1]{lusztig88}, we deduce that 
	\[
		\rmH^*_{Z_{\Mtq}(e)}\cong \rmH^*_{Z^q_{\Mt}(\phi)}\cong \rmH^*_{Z_{\Mt}\times \bfC^{\times}}\cong \bfk[\bbE^M_{\diamond}]\otimes \bfk[u]. \qedhere
	\]
\end{proof}
Now we define a $(\bfS^M, \bfS^M)$-bimodule structure on various cohomology groups. For any $\nu\in \Xi$, the stack morphism $[\calT^{\nu}/ \Gztq]\to [\rmO_{\eta}/\Mztq]$ induces by pull-back an isomorphism of equivariant cohomology rings:
\[
	\bfS^M\cong \rmH^*_{\Pztq^{\nu}}(\rmO_{\eta}, \bfk)\cong \rmH^*_{\Pztq^{\nu}}(\rmO_{\eta}\oplus \fraku^{\nu}_{\eta}, \bfk)\cong \rmH^*_{\Gztq}(\calT^{\nu}, \bfk). 
\]
Let $\nu,\nu'\in \Xi$ and suppose that $i_V:V\hookrightarrow \calZ^{\nu,\nu'}$ is a $\Gztq$-stable subvariety. Then the two projections $q_1:\calZ^{\nu,\nu'}\to \calT^{\nu}$ and $q_2:\calZ^{\nu,\nu'}\to \calT^{\nu'}$ induce two stack morphisms
\[
	[\calT^{\nu}/\Gztq]\xleftarrow{\ba q_1}[V/\Gztq]\xrightarrow{\ba q_2}[\calT^{\nu'}/\Gztq].
\]
The diagonal $(\ba q_1, \ba q_2)$ yields a ring homomorphism
\[
	\bfS^M\otimes \bfS^M\cong \rmH^*_{\Gztq\times\Gztq}(\calT^{\nu}\times\calT^{\nu},\bfk)\xrightarrow{(\ba q_1, \ba q_2)^*}\rmH^*_{\Gztq}(V, \bfk),
\]
which defines a graded $(\bfS^M, \bfS^M)$-bimodule structure on $\rmH^*_{\Gztq}(V, i^!_V\scrK)$ by cup product via $(\ba q_1, \ba q_2)^*$.

\subsection{Cohomology of strata}\label{subsec:cohom}
We put for $\nu, \nu'\in \Xi$ and $[w]\in  \calW_{I} \backslash \calW/ \calW_{I}$ 
\begin{equation*}\begin{aligned}
				\calH^{\nu, \nu'}_w = \rmH^*_{\Gztq}\left(\calZ^{\nu, \nu'}_w, i_w^!\scrK \right),
\end{aligned}\end{equation*}
\index{H@$\calH^{\nu, \nu'}_w$}
where $i_w: \calZ_w^{\nu, \nu'}\to \calZ^{\nu, \nu'}$ is the locally closed immersion. The following lemma determines the ``size'' of the extension space $\calH^{\nu,\nu'}$ and is crucial in our calculation of extension algebra.

\begin{prop}\label{prop:SS-bimod}
	Let $\nu, \nu'\in \Xi$ and $[w]\in \calW_{I}\backslash \calW / \calW_{I}$.
	\begin{enumerate}[label=(\roman*)]
		\item\label{prop:SS-bimod-ii}
			If $[w]$ is bad, i.e. $[w]\notin \calW_{\xi}$, then $\calH^{\nu, \nu'}_w = 0$.
		\item\label{prop:SS-bimod-iii}
			If $[w]$ is good, i.e. $[w]\in \calW_{\xi}$, and if $w\ubar{\nu'} = \ubar{\nu}$, then $\calH^{\nu, \nu'}_w$ is a free of rank $1$ as left (resp. right) graded $\bfS^M$-module, vanishing in odd degrees.
		\item\label{prop:SS-bimod-iv}
			If $[w]$ is good but $w\ubar{\nu'} \neq \ubar{\nu}$, then $\calH^{\nu, \nu'}_w = 0$.
	\end{enumerate}
\end{prop}
\begin{proof}
	We prove~\ref{prop:SS-bimod-ii}. Suppose that $[w]\notin \calW_{\xi}$. Recall that there is a left $G_{\ubar 0}$-action on $\calX^{\nu, \nu'}_w$ with a finite number of orbits. Let
	\begin{equation*}\begin{aligned}
		\calX^{\nu, \nu'}_w = \bigsqcup_{\Omega} \calX_{\Omega}
	\end{aligned}\end{equation*}
	be the stratification by $G_{\ubar 0}$-orbits, where $\Omega$ is taken over a finite subset of $P^{\nu}_0\backslash G_{\ubar 0}/P^{\nu'}_{0}$. We define $\calZ_\Omega\subseteq \calZ^{\nu, \nu'}$ to be the pre-image of $\calX_\Omega$ under the projection $\calZ^{\nu, \nu'}\to \calX^{\nu, \nu'}$. Notice that all double cosets $\Omega$ that appear in this stratification are {\it bad} in the sense of~\cite[5.2]{LYI} since $[w]$ is not in $\calW_{\xi}$.\par
	Put
	\begin{equation*}\begin{aligned}
		\dot\frakp^{\nu}_{\eta} = \left(\rmO^{\nu}_{\eta}\oplus \fraku^{\nu}_{\eta}\right)\times_{\frakg_{\ubar \eta}} \calT^{\nu'} 
	\end{aligned}\end{equation*}
	so that $\calZ^{\nu, \nu'} \cong G_{\ubar 0}\times^{P_0^{\nu''}}\dot\frakp^{\nu'}_{\eta}$. Moreover, for each $\Omega$ we put 
	\begin{equation*}\begin{aligned}
		\dot\frakp^{\nu}_{\eta,\Omega} = \left(\rmO^{\nu}_{\eta}\oplus \fraku^{\nu}_{\eta}\right)\times_{\frakg_{\ubar \eta}}\left(\Omega\times^{P^{\nu'}_0}\left(\rmO^{\nu'}_{\eta}\oplus \fraku^{\nu'}_{\eta}\right)\right)\subseteq\dot\frakp^{\nu}_{\eta}
	\end{aligned}\end{equation*}
	then there is a diagram of Cartesian squares
	\begin{equation*}\begin{aligned}
		\begin{tikzcd}
			\dot\frakp^{\nu}_{\eta,\Omega}\arrow{r}{r'}\arrow{d}{q_1'} & \calZ_{\Omega} \arrow{r}{q_2}\arrow{d}{q_1}& \calT^{\nu'} \arrow{d}\\
			\rmO_{\eta}\oplus\fraku^{\nu}_{\eta} \arrow{r}{r} & \calT^{\nu} \arrow{r}& \frakg_{\ubar \eta}
		\end{tikzcd}
	\end{aligned}\end{equation*}
	Let $i:\calZ_{\Omega}\hookrightarrow \calZ^{\nu, \nu'}$ be the inclusion. Now 
	\begin{equation*}\begin{aligned}
		\rmH^*_{\Gztq}\left( \calZ_{\Omega}, i^!\scrK \right) \cong \Ext^*_{\Gztq}\left( q_2^*\dot\scrC_{\nu'}, q_1^!\dot\scrC_{\nu}\right)\cong\Ext^*_{\Gztq}\left( q_{1!}q_{2}^*\dot\scrC_{\nu'}, \dot\scrC_{\nu}\right) \\
		\cong\Ext^*_{\Pztq^{\nu}}\left( r^*q_{1!}q_{2}^*\dot\scrC_{\nu'}, r^*\dot\scrC_{\nu}\right)\cong\Ext^*_{\Pztq^{\nu}}\left( q'_{1!}r'^*q_{2}^*\dot\scrC_{\nu'}, r^*\dot\scrC_{\nu}\right).
	\end{aligned}\end{equation*}
	Taking into account the cuspidality of $\scrC_\eta$ and the fact that $\Omega$ is bad,~\cfauto{prop:gooddoublecosets} and~\autoref{rema:goodbad}, it is shown in~\cite[5.3]{LYI} that $q'_{1!}r'^*q_{2}^*\dot\scrC_{\nu'} = 0$. Thus $\rmH^*_{\Gztq}\left( \calZ_{\Omega}, i^!\scrK \right)  = 0$ for each $\Omega$. By an argument of long exact sequence of cohomology, we conclude that 
	\begin{equation*}\begin{aligned}
		\calH^{\nu, \nu'}_w = \rmH^*_{\Gztq}\left( \calZ^{\nu, \nu'}_{w}, i_w^!\scrK \right) = 0.
	\end{aligned}\end{equation*}
	This proves~\ref{prop:SS-bimod-ii}.
	\par

	We turn to~\ref{prop:SS-bimod-iii}. The arguments follow~\cite[4.2]{lusztig88}. Suppose that $w\nu' = \nu$ with $w\in \calW_{\xi}$. \autoref{prop:strZ} gives an isomorphism
	\[
		G_{\ubar 0}\times^{P^{\nu}_0\cap P^{\nu'}_0}(\rmO_{\eta}\oplus \fraku^{\nu}_{\eta}\cap \fraku^{\nu'}_{\eta})\cong\calZ^{\nu,\nu'}_w. 
	\]
	Let 
	\[
		p_1: G_{\ubar 0}\times^{P^{\nu}_0\cap P^{\nu'}_0}(\rmO_{\eta}\oplus \fraku^{\nu}_{\eta}\cap \fraku^{\nu'}_{\eta})\to \calT^{\nu} ,\quad p_2: G_{\ubar 0}\times^{P^{\nu}_0\cap P^{\nu'}_0}(\rmO_{\eta}\oplus \fraku^{\nu}_{\eta}\cap \fraku^{\nu'}_{\eta})\to \calT^{\nu'}
	\]
	be the two projections. Since $\dot\scrC_{\nu}$ is a local system and since $\calT^{\nu}$ and $G_{\ubar 0}\times^{P^{\nu}_0\cap P^{\nu'}_0}(\rmO_{\eta}\oplus \fraku^{\nu}_{\eta}\cap \fraku^{\nu'}_{\eta})$ are smooth varieties, we have $p^!_1\dot\scrC_{\nu} \cong p_1^*\dot\scrC_{\nu}[2d_{\nu,\nu'}]$ by the Verdier duality, where
	\[
		d_{\nu,\nu'} = \dim P^{\nu'}_{0} / P^{\nu}_{0}\cap P^{\nu'}_{0} - \dim \fraku^{\nu'}_{\eta} / \fraku^{\nu}_{\eta} \cap \fraku^{\nu'}_{\eta}\in \bfZ.
	\]
	Denote by $f:\rmO_{\eta}\oplus \fraku^{\nu}_{\eta}\cap \fraku^{\nu'}_{\eta}\to \rmO_{\eta}$ the projection. Then 
	\[
		\calH^{\nu, \nu'}_w\cong \Ext^{*+2d_{\nu,\nu'}}_{\Gztq}(p_2^*\dot\scrC_{\nu'}, p_1^*\dot\scrC_{\nu}) \cong \Ext^{*+2d_{\nu,\nu'}}_{(P^{\nu}_0\cap P^{\nu'}_0)_{\diamond,q}}(f^*\scrC_{\eta}, f^*\scrC_{\eta})\cong \Ext^{*+2d_{\nu,\nu'}}_{\Mztq}(\scrC_{\eta}, \scrC_{\eta}).
	\]
	Since $\scrC_{\eta}$ is irreducible, the last term is isomorphic to $\rmH^{*+2d_{\nu,\nu'}}_{\Mztq}(\rmO_{\eta}, \bfk)$, which is a graded-free left (resp. right) $\bfS^M$-module of rank $1$ shifted by an even degree. By~\autoref{lemm:SE}, it vanishes in odd degrees, whence~\ref{prop:SS-bimod-iii}.
	\par
	Finally, if $w\in \calW_{\xi}$ but $w\ubar\nu' \neq \ubar\nu$, then $\calX^{\nu, \nu'}_w = \emptyset$ and $\calZ^{\nu, \nu'}_w = \emptyset$ by~\autorefitem{lemm:uniqueorbit}{ii}, so $\calH^{\nu, \nu'}_w = 0$, whence~\ref{prop:SS-bimod-iv}. 
\end{proof}
\begin{rema}
	The special feature of convolution algebra with a cuspidal local system as coefficient is that not every Bruhat cell on the partial flag variety contributes to the cohomology, but only those corresponding to the subset $\calW_{\xi}\subset \calW_I \backslash \calW / \calW_I$ do. \autoref{prop:SS-bimod} roughly says that the convolution algebra $\ha\calH$ and the dDAHA $\bfH_{\xi}$ are of the same {\itshape size}. The proof is adapted from~\cite[4.7]{lusztig88}.
\end{rema} 
\subsection{Filtration by Bruhat order}\label{subsec:filtration}

Recall that by~\autoref{prop:strZ}, $\calZ^{\nu,\nu'}_{\calI}$ is a closed subvariety of $\calZ^{\nu,\nu'}$. Denote by $i_{\calI}:\calZ^{\nu,\nu'}_{\calI}\hookrightarrow \calZ^{\nu,\nu'}$ the inclusion.

\begin{prop}\label{prop:HI}
	For each ideal $\calI\subset\calW_I\backslash \calW / \calW_I$, the space 
	\[
		\calH^{\nu,\nu'}_{\calI} = \rmH^*_{\Gztq}\left( \calZ^{\nu,\nu'}_{\calI}, i_{\calI}^!\scrK \right)
	\]
	is a graded $(\bfS^M,\bfS^M)$-bimodule which vanishes in odd degrees and is graded-free as one-sided $\bfS^M$-module. Moreover, if $\calJ\subset \calI$ is a sub-ideal, then the inclusion $\calZ^{\nu,\nu'}_{\calJ}\subseteq \calZ^{\nu,\nu'}_{\calI}$ induces an injective map of graded $(\bfS^M,\bfS^M)$-bimodules
	\begin{equation*}\begin{aligned}
		\calH^{\nu,\nu'}_{\calJ}\hookrightarrow \calH^{\nu,\nu'}_{\calI}.
	\end{aligned}\end{equation*}
\end{prop}\index{H@$\calH^{\nu,\nu'}_{\calI}$}
\begin{proof}
	We may assume that $\calI$ is finite because there is only a finite number of $[w]\in \calW_I\backslash \calW / \calW_I$ such that $\calZ^{\nu,\nu'}_{w}\neq \emptyset$. The first statement is proven by induction on $\# \calI$ and~\autoref{prop:SS-bimod}. The second statement is a direct consequence of the first one.
\end{proof}
Consequently, we may view $\calH_{\calJ}^{\nu,\nu'}$ as subspace of $\calH^{\nu,\nu'}$ by taking $\calI = \calW_I \backslash \calW / \calW_I$.
\begin{prop}\label{prop:Hquot}
	For any pair of ideals $\calI,\calJ\subset\calW_I \backslash \calW / \calW_I$, let $\calI\calJ \subset\calW_I \backslash \calW / \calW_I$ be the product ideal introduced in~\autoref{lemm:CC}. Then the convolution product
	\begin{equation*}\begin{aligned}
		\calH^{\nu,\nu'}_{\calI}\otimes\calH^{\nu',\nu''}_{\calJ} \to \calH^{\nu,\nu''}
	\end{aligned}\end{equation*}
	factorises through the subspace $\calH^{\nu,\nu''}_{\calI\calJ}\subset\calH^{\nu,\nu''}$.
\end{prop}
\begin{proof}
	It follows from~\autoref{lemm:XI} that the projection
	\[
		\calX^{\nu,\nu'}_{\calI}\times_{\calP^{\nu'}}\calX^{\nu',\nu''}_{\calJ} \to\calX^{\nu,\nu''}
	\]
	lies in $\calX^{\nu,\nu''}_{\calI\calJ}$. It follows that the image of
	\[
		\calZ^{\nu,\nu'}_{\calI}\times_{\calT^{\nu'}}\calZ^{\nu',\nu''}_{\calJ} \to\calZ^{\nu,\nu''}
	\]
	lies in $\calZ^{\nu,\nu''}_{\calI\calJ}$.
\end{proof}

\subsection{Convolution product over good strata}\label{subsec:grconv}

For each $y\in \calW_{\xi}$, we have two ideals $\calI_y = \left\{ [w]\;;\; w\le y \right\}$ and $\calJ_y = \calI_y\setminus [y]$ of $\calW_I\backslash \calW / \calW_I$. Denote 
\begin{equation*}\begin{aligned}
	\calH^{\nu,\nu'}_{\le y} = \calH^{\nu,\nu'}_{\calI_y},\quad  \calH^{\nu,\nu'}_{< y} = \calH^{\nu,\nu'}_{\calJ_y}. 
\end{aligned}\end{equation*}
There is a short exact sequence of $(\bfS^M,\bfS^M)$-bimodules
\begin{equation*}\begin{aligned}
	0\to \calH^{\nu,\nu'}_{< y}\to \calH^{\nu,\nu'}_{\le y}\to \calH^{\nu,\nu'}_{y} \to 0.
\end{aligned}\end{equation*}

If $w,w'\in \calW_{\xi}$ are such that $\ell_{\xi}(ww') = \ell_{\xi}(w) + \ell_{\xi}(w')$, then $\ell(ww') = \ell(w) + \ell(w')$ by~\autorefitem{prop:relW}{iii} and therefore there are inclusions
\[
	\calI_w\cdot \calI_{w'} \subset \calI_{ww'},\quad \calI_w\cdot \calJ_{w'} \subset \calJ_{ww'},\quad \calJ_w\cdot \calI_{w'} \subset\calJ_{ww'}.
\]
Thus~\autoref{prop:Hquot} yields a map
\begin{equation}\begin{aligned}\label{equa:conv}
	\calH^{\nu,\nu'}_{w}\otimes \calH^{\nu',\nu''}_{w'}\cong (\calH^{\nu,\nu'}_{\le w} / \calH^{\nu,\nu'}_{< w})\otimes (\calH^{\nu',\nu''}_{\le w'} / \calH^{\nu',\nu''}_{< w'}) \to \calH^{\nu,\nu''}_{\le ww'} / \calH^{\nu,\nu''}_{< ww'} = \calH^{\nu,\nu''}_{ww'}.
\end{aligned}\end{equation}

\begin{prop}\label{lemm:isostr}
	Given $\nu\in \Xi$, $w,w'\in \calW_{\xi}$ such that $\ell_{\xi}(ww') = \ell_{\xi}(w) + \ell_{\xi}(w')$, let $\nu' =  w^{-1}\nu$ and $\nu'' = w'^{-1}\nu'$. 
	\begin{enumerate}[label=(\roman*)]
		\item\label{lemm:isostr-iii}
			The map
			\begin{equation*}\begin{aligned}
				\calH^{\nu, \nu'}_{w}\otimes \calH^{\nu', \nu''}_{w'} \to \calH^{\nu, \nu''}_{ww'}
			\end{aligned}\end{equation*} from~\eqref{equa:conv} is surjective.
		\item \label{lemm:isostr-iv}
			There is an isomorphism of algebras $\bfS^M\cong\calH^{\nu, \nu}_e$ which is at the same time an $(\bfS^M,\bfS^M)$-bimodule isomorphism. 
	\end{enumerate}
\end{prop}
\begin{rema}
	More generally, one can show that the associated graded algebra of the Bruhat filtration on the sum $\bigoplus_{\ubar\nu,\ubar\nu'}\calH^{\nu,\nu'}$ is isomorphic to a certain skew tensor product of the nil-Hecke algebra of $\calW_{\xi}$ with the ``polynomial part'' $\bigoplus_{\ubar\nu}\calH^{\nu,\nu}_e$. However, we will not need it.
\end{rema}
\begin{proof}
	We proceed in steps. 
	\begin{enumerate}
		\item[\ul{Step 1}.]
			We show that the projection $\calZ^{\nu, \nu'}\times_{\calT^{\nu'}} \calZ^{\nu', \nu''}\to \calZ^{\nu,\nu''}$ induces an isomorphism 
			\[
				\ba\mu:\calZ^{\nu, \nu'}_{w}\times_{\calT^{\nu'}} \calZ^{\nu', \nu''}_{w'}\cong  \calZ^{\nu, \nu''}_{ww'}.  
			\]
			Let $\kappa\in \frakA$ be an alcove such that $\partial_{I_{\xi}}\kappa = \nu$. We view $w$ and $w'$ as elements of $\calW$ via the splitting of shortest representative $\calW_\xi\hookrightarrow N_{\calW}(\calW_{I})$ of~\autorefitem{prop:relW}{ii}. From~\autoref{prop:Ppoids}, we deduce the following formula
			\begin{equation*}\begin{aligned}
				\sum_{n\in \bfZ} \dim (\frakp^{\nu}_n/ \frakp^{\nu}_n\cap \frakp^{w^{-1}\nu}_n) = \#\left(w\Delta^{\kappa}\cap -\Delta^{\kappa}\right) = \ell(w),
			\end{aligned}\end{equation*}
			from which we get inequalities
			\begin{equation*}\begin{aligned}
				\ell(ww') &= \sum_{n\in \bfZ} \dim (\frakp^{\nu}_n/ \frakp^{\nu}_n\cap \frakp^{\nu''}_n) \le \sum_{n\in \bfZ}\dim (\frakp^{\nu}_n/ \frakp^{\nu}_n\cap \frakp^{\nu'}_n\cap \frakp^{\nu''}_n) \\
				&=  \sum_{n\in \bfZ}\dim (\frakp^{\nu}_n/ \frakp^{\nu}_n\cap \frakp^{\nu'}_n) + \sum_{n\in \bfZ}\dim (\frakp^{\nu}_n\cap \frakp^{\nu'}_n/ \frakp^{\nu}_n\cap \frakp^{\nu'}_n\cap \frakp^{\nu''}_n) \le \ell(w) + \ell(w').
			\end{aligned}\end{equation*}
			By~\autorefitem{prop:relW}{iii}, we have $\ell(ww') = \ell(w) + \ell(w')$, from which we deduce $P^{\nu}_0\cap P^{\nu'}_0\cap P^{\nu''}_0 = P^{\nu}_0\cap P^{\nu''}_0$ as well as $\frakp^{\nu}_\eta\cap \frakp^{\nu'}_\eta\cap \frakp^{\nu''}_{\eta} = \frakp^{\nu}_{\eta}\cap \frakp^{\nu''}_{\eta}$. By~\autoref{prop:strZ}, we see that
			\begin{equation*}\begin{aligned}
				\calZ^{\nu,\nu'}_{w}\times_{\calT^{\nu'}}\calZ^{\nu',\nu''}_{w'} &\cong G_{\ubar 0}\times^{P^{\nu}_0\cap P^{\nu'}_0\cap P^{\nu''}_0}\left( \rmO_{\eta}\oplus\fraku^{\nu}_{\eta}\cap\fraku^{\nu'}_{\eta}\cap \fraku^{\nu''}_{\eta} \right) \\
				&\cong G_{\ubar 0}\times^{P^{\nu}_0\cap P^{\nu''}_0}\left(\rmO_{\eta}\oplus \fraku^{\nu}_{\eta}\cap \fraku^{\nu''}_{\eta} \right)\cong \calZ^{\nu,\nu''}_{ww'},
			\end{aligned}\end{equation*}
			which proves the claim. \par
		\item[\ul{Step 2}.]
			We show that the map~\eqref{equa:conv} agrees with the {\it Gysin map} of the closed embedding 
			\begin{equation*}\begin{aligned}
				\calZ^{\nu, \nu''}_{ww'}\cong \calZ^{\nu, \nu'}_{w}\times_{\calT^{\nu'}} \calZ^{\nu', \nu''}_{w'} \hookrightarrow \calZ^{\nu, \nu'}_{w}\times \calZ^{\nu', \nu''}_{w'}.
			\end{aligned}\end{equation*}
			It is a consequence of the transversality proven in Step 1\footnote{It is a cohomological analogue of the well-known property in the intersection theory that ``refined Gysin map = Gysin map'' in the case of transversal intersection. Indeed, the convolution product that we have defined makes use of the refined Gysin map from the diagonal embedding of $\calT^{\nu'}$, which is a regular embedding, see~\cite[6.3.2]{fulton98}}, see also the proof of~\cite[7.6.12]{CG}. Consider the Cartesian square
			\begin{equation*}\begin{aligned}
				\begin{tikzcd}
					\calZ^{\nu, \nu'}_w \times\calZ^{\nu', \nu''}_{w'}\arrow[swap]{d}{k = (q_1,q_2)\times(q_1,q_2)} & \calZ^{\nu, \nu'}_w\times_{\calT^{\nu'}}\calZ^{\nu', \nu''}_{w'} \arrow[swap]{l}{\ba \gamma}\arrow{d}{k'}\arrow{r}{\ba\mu}[swap]{\cong}  & \calZ^{\nu,\nu''}_{ww'} \\
					\left(\calT^{\nu}\times \calT^{\nu'}\right) \times \left(\calT^{\nu'}\times \calT^{\nu''}\right) & \calT^{\nu}\times \calT^{\nu'} \times \calT^{\nu''} \arrow{l}{\gamma}
				\end{tikzcd}.
			\end{aligned}\end{equation*}
			All the varieties appeared are smooth, all the morphisms are immersions. By dimension count, the immersions $k$ and $\gamma$ intersect transversally. It gives rise to a commutative triangle in $\Db_{\Gztq}(\calZ^{\nu, \nu'}_w \times\calZ^{\nu', \nu''}_{w'})$:
			\begin{equation}\begin{aligned}\label{equa:adjbas}
				\begin{tikzcd}[column sep=7em]
					& \ba \gamma_*\ba \gamma^* k^!\bfD \bfk \arrow{d}{\cong} \\
					k^!\bfD \bfk \arrow[swap]{r}{k^!\left(\id \to  \gamma_*\gamma^*\right)}\arrow{ur}{\id \to  \ba \gamma_*\ba \gamma^*} & k^!\gamma_*\gamma^*\bfD \bfk
				\end{tikzcd}.
			\end{aligned}\end{equation}
			Let $\scrG = \dot\scrC_{\nu}\boxtimes \bfD \dot\scrC_{\nu'}\boxtimes \dot\scrC_{\nu'} \boxtimes \bfD \dot\scrC_{\nu''}$ be on $\calT^{\nu}\times\calT^{\nu'}\times \calT^{\nu'}\times \calT^{\nu''}$. Applying $\scrHom(k^*\bfD \scrG, \relbar)$ to the commutative triangle~\eqref{equa:adjbas}, we obtain another commutative triangle
			\begin{equation}\begin{aligned}\label{equa:adjbasii}
				\begin{tikzcd}[column sep=7em]
					& \ba \gamma_*\ba \gamma^* k^!\scrG \arrow{d}{\cong} \\
					k^!\scrG \arrow[swap]{r}{k^!\left(\id \to  \gamma_*\gamma^*\right)}\arrow{ur}{\id \to  \ba \gamma_*\ba \gamma^*} & k^!\gamma_*\gamma^*\scrG
				\end{tikzcd}.
			\end{aligned}\end{equation}
			We apply the functor $\rmH^*_{\Gztq\times \Gztq}(\calZ^{\nu,\nu'}_w\times \calZ^{\nu,\nu'}_{w'}, \relbar)$ on~\eqref{equa:adjbasii}. The morphism $k^!(\id\to \gamma_*\gamma^*)$ yields the convolution product introduced in~\autoref{sec:convolution} and the morphism $\id\to \ba\gamma_*\ba\gamma^*$ yields the Gysin map. The commutativity of~\eqref{equa:adjbasii} implies that these two products coincide, whence the claim. \par
		\item[\ul{Step 3}.]
			We prove~\ref{lemm:isostr-iii}. Step 2 allows us to compute the map~\eqref{equa:conv} on graded pieces with the following commutative diagram:
			\begin{equation*}\begin{aligned}
				\begin{tikzcd}[column sep=5em]
					{\left[\calZ^{\nu, \nu'}_w / \Gztq\right]}\times {[\calZ^{\nu', \nu''}_{w'} / \Gztq]} \arrow{d} & {[\calZ^{\nu, \nu''}_{ww'} / \Gztq]} \arrow{l}[swap]{\ba\gamma\ba\mu^{-1}}\arrow{d} \\
					{\left[\calO_{\eta} / \Mztq\right]}\times {[\calO_{\eta} / \Mztq]} & {[\calO_{\eta} / \Mztq]} \arrow{l}[swap]{\Delta}
				\end{tikzcd}.
			\end{aligned}\end{equation*}
			The arguments in the proof of~\autoref{prop:SS-bimod}~\ref{prop:SS-bimod-iii} shows that the vertical arrows induce isomorphisms on cohomology groups, so that we have:
			\begin{equation}\begin{aligned}\label{equa:SetH}
				\begin{tikzcd}[column sep=7em]
					\rmH^*_{\Gztq}\left(\calZ^{\nu, \nu'}_w, i_w^!\scrK\right)\otimes \rmH^*_{\Gztq}\left(\calZ^{\nu', \nu''}_{w'}, i_{w'}^!\scrK\right) \arrow{r}{\id\to (\ba\gamma\ba\mu^{-1})_*(\ba\gamma\ba\mu^{-1})^*}& \rmH^*_{\Gztq}\left(\calZ^{\nu, \nu''}_{ww'}, i_{ww'}^!\scrK\right) \\
					\rmH^{*+2d_{\nu,\nu'}}_{\Mztq}\left(\calO_{\eta}, \bfk\right)\otimes \rmH^{*+2d_{\nu',\nu''}}_{\Mztq}\left(\calO_{\eta}, \bfk\right) \arrow{r}{\id\to \Delta_*\Delta^*} \arrow{u}{\cong}& \rmH^{*+2d_{\nu,\nu''}}_{\Mztq}\left(\calO_{\eta}, \bfk\right)  \arrow{u}{\cong}
				\end{tikzcd},
			\end{aligned}\end{equation}
			where the bottom row is identified with the cup product of $\bfS^M =\rmH^*_{\Mztq}\left(\calO_{\eta}, \bfk\right)$ shifted by appropriate degrees, which is surjective. This proves~\ref{lemm:isostr-iii}. Setting $\nu = \nu' = \nu''$ in~\eqref{equa:SetH}, we get~\ref{lemm:isostr-iv}.\qedhere
	\end{enumerate}
\end{proof}

\subsection{Parabolic subalgebras}\label{subsec:parabH}
In~\autoref{subsec:affinehecke}, we have introduced an algebra 
\[
	\ha\calH_{\sigma} = \bigoplus_{\ubar\nu,\ubar\nu'\in \ubar\Xi^{\sigma}} \Ext^*_{\Lztq^{\sigma}}(\Ind^{\sigma}_{\nu'}\scrC, \Ind^{\sigma}_{\nu}\scrC)_{0}
\]
for each $\bbE^M$-facet $\sigma$. This algebra can be thought of as parabolic subalgebra of $\ha\calH$ associated to the facet $\sigma$. The analysis of $\calZ^{\nu,\nu'}$ and $\calH^{\nu,\nu'}$ in~\autoref{sec:geomZ} and in~\autoref{sec:convolution} can equally be done for $\sigma$. In~\autoref{prop:imagepsi}, we will use the geometry of Steinberg varieties to study the map $\psi_J$ appearing in~\autoref{theo:Phi}. \par

Let $\sigma\in \frakF(\bbE^M)$ be an $\bbE^M$-facet and let $J\subset \Delta$ denote its type (\autoref{subsec:facetcox}) so that $\sigma\in \frakF_J$. By~\autoref{prop:Wxi}~\ref{prop:Wxi-iv}, we have $I = I_\xi\subset J$. 
For $\nu,\nu'\in \Xi^{\sigma} = \left\{ \nu\in \Xi\;;\; \sigma\le \nu \right\}$, put
\begin{equation*}\begin{aligned}
	\calX^{\nu,\nu'}_{\sigma} &= (L^{\sigma}_0 / P^{\nu\le \sigma}_0)\times ( L^{\sigma}_0 / P^{\nu'\le \sigma}_0 ), \\
	\calT_{\sigma}^{\nu} &= L^{\sigma}_0\times^{P^{\sigma\le \nu}_0}(\rmO_{\eta}\oplus \fraku^{\sigma\le \nu}_\eta)\to \frakl^{\sigma}_{\eta},&  \calZ_{\sigma}^{\nu,\nu'} &= \calT_{\sigma}^{\nu}\times_{\frakl^\sigma_{\eta}}\calT_{\sigma}^{\nu'}.
\end{aligned}\end{equation*}
There are $L^{\sigma}_0$-stable stratifications 
\[
	\calX^{\nu,\nu'}_{\sigma} = \bigcup_{[w]\in \calW_I\backslash \calW_J /\calW_I} \calX^{\nu,\nu'}_{\sigma, w},\quad \calZ^{\nu,\nu'}_{\sigma} = \bigcup_{[w]\in \calW_I\backslash \calW_J /\calW_I} \calZ^{\nu,\nu'}_{\sigma, w}
\]
so that $(P^{\sigma\le \nu}_0, P^{\sigma\le \nu'}_0)\in \calX^{\nu,\nu'}_{\sigma, w}$ for $w\in N_{\calW_J}(\calW_I)$ such that $w\nu' = \nu$. 
\begin{prop}\label{prop:XZpara}
	The statements below hold for $\nu,\nu',\nu''\in \Xi^{\sigma}$ and for ideals $\calI,\calJ\subset \calW_I\backslash \calW_J/\calW_I$. 
	\begin{enumerate}
		\item \label{prop:XZpara-i}
			The subvarieties $\calX^{\nu,\nu'}_{\sigma,\calI}\subset \calX^{\nu,\nu'}_{\sigma}$ and $\calZ^{\nu,\nu'}_{\sigma,\calI}\subset \calZ^{\nu,\nu'}_{\sigma}$ are closed.
		\item \label{prop:XZpara-ii}
			We have 
			\begin{equation*}\begin{aligned}
				\calX^{\nu,\nu'}_{\sigma,\calI}\times_{L^{\sigma}_0 / P^{\sigma\le \nu'}_0}\calX^{\nu',\nu''}_{\sigma,\calJ}\to \calX^{\nu',\nu''}_{\sigma,\calI\calJ},\quad \calZ^{\nu,\nu'}_{\sigma,\calI}\times_{L^{\sigma}_0 / P^{\sigma\le \nu'}_0}\calZ^{\nu',\nu''}_{\sigma,\calJ}\to \calZ^{\nu',\nu''}_{\sigma,\calI\calJ}.
			\end{aligned}\end{equation*}
		\item \label{prop:XZpara-iii}
			For each $w\in \calW_{\xi,J}$ such that $w\ubar\nu' = \ubar \nu$, there are isomorphisms
			\begin{equation*}\begin{aligned}
				\calX^{\nu,\nu'}_{\sigma,w} \cong L^{\sigma}_0/ P^{\sigma\le \nu}_0\cap P^{\sigma\le w^{-1}\nu}_0 ,\quad
				\calZ^{\nu,\nu'}_{\sigma,w} \cong L^{\sigma}_0\times^{P^{\sigma\le \nu}_0\cap P^{\sigma\le w^{-1}\nu}_0}(\rmO_{\eta}\oplus \fraku^{\sigma\le \nu}_{\eta}\cap \fraku^{\sigma\le w^{-1}\nu}_{\eta}).
			\end{aligned}\end{equation*}
	\end{enumerate}
\end{prop}
\begin{proof}
	See~\autoref{lemm:XI} and~\autoref{prop:strZ}.
\end{proof}
For $\nu\in \Xi^{\sigma}$, the inverse image of $\scrC_{\eta}$ under the stack morphism $f: [\calT^{\nu}_{\sigma} / \Lztq^{\sigma}]\to [\rmO_{\eta}/\Mztq]$ is denoted by $\dot\scrC$. We have similarly a complex $\scrK_{\sigma}$ on $\calZ^{\nu,\nu'}_{\sigma}$ such that 
\[
	\calH^{\nu,\nu'}_{\sigma} := \Ext^*_{\Lztq^{\sigma}}(\Ind^\sigma_{\nu'}\scrC, \Ind^\sigma_\nu\scrC) \cong \rmH^*_{\Lztq^{\sigma}}(\calZ^{\nu,\nu'}_{\sigma}, \scrK_{\sigma}). 
\]
As in~\autoref{subsec:filtration}, we put a filtration on $\calH^{\nu,\nu'}_{\sigma}$ by defining for each ideal $\calI\subset \calW_I\backslash \calW_J / \calW_I$ the following space
\[
	\calH^{\nu,\nu'}_{\sigma, \calI} = \rmH^*_{\Lztq^{\sigma}}(\calZ^{\nu,\nu'}_{\sigma,\calI}, i^!_{\calI}\scrK_{\sigma}),\quad i_{\calI}:\calZ^{\nu,\nu'}_{\sigma,\calI}\hookrightarrow \calZ^{\nu,\nu'}_{\sigma}.
\]
For each $[w]\in\calW_I\backslash \calW_J / \calW_I$, we set
\[
	\rmH^*_{\Lztq^{\sigma}}(\calZ^{\nu,\nu'}_{\sigma,w}, i^!_{w}\scrK_{\sigma}),\quad i_{w}:\calZ^{\nu,\nu'}_{\sigma,w}\hookrightarrow \calZ^{\nu,\nu'}_{\sigma}.
\]
Then~\autoref{prop:SS-bimod} and~\autoref{lemm:isostr} also have analogues for $\calH^{\nu,\nu'}_{\sigma,w}$ and $\calH^{\nu,\nu'}_{\sigma,\calI}$. \par

Now we consider the spiral induction $\Ind_{\sigma}$, the functoriality of which yields a map $\psi_{\sigma}$\index{ps@$\psi_{\sigma}$} for $\nu,\nu'\in \Xi^{\sigma}$:
\begin{equation}\begin{aligned}\label{equa:indsigma}
	\calH^{\nu,\nu'}_{\sigma} = \Ext^*_{\Lztq^{\sigma}}(\Ind^{\sigma}_{\nu'}\scrC, \Ind^{\sigma}_{\nu'}\scrC)\xrightarrow{\psi_{\sigma}} \Ext^*_{\Gztq}(\Ind_{\sigma}\Ind^{\sigma}_{\nu'}\scrC, \Ind_{\sigma}\Ind^{\sigma}_{\nu'}\scrC)\cong \calH^{\nu,\nu'}.
\end{aligned}\end{equation}
\begin{prop}\label{prop:imagepsi}
	The map $\psi_{\sigma}$ is injective and for each ideal $\calI\subset \calW_I\backslash \calW_J / \calW_I$, we have 
	\[
		\psi_{\sigma}(\calH^{\nu,\nu'}_{\sigma,\calI}) = \calH^{\nu,\nu'}_{\calI}.
	\]
	In particular, $\psi_{\sigma}$ induces an isomorphism $\calH^{\nu,\nu'}_{\sigma} \cong \calH^{\nu,\nu'}_{\le w^J_0}$, where $w^J_0$ is the longest element of the finite Coxeter subgroup $(\calW_J, J)\subset(\calW, \Delta)$.
\end{prop}
\begin{proof}
			Set
			\begin{equation*}\begin{aligned}
				\pi^*\calT_{\sigma}^{\nu} &= L^{\sigma}_0\times^{P^{\sigma\le \nu}_0}(\rmO_{\eta}\oplus \fraku^{\nu}_\eta)\to \frakp^{\sigma}_{\eta},& \pi^*\calZ_{\sigma}^{\nu,\nu'} &= \calT_{\sigma}^{\nu}\times_{\frakp^\sigma_{\eta}}\calT_{\sigma}^{\nu'}.
			\end{aligned}\end{equation*}
			We have the following two Cartesian squares 
			\begin{equation}\begin{aligned}\label{equa:cartesiens}
				\begin{tikzcd}
					\calZ^{\nu,\nu'}_{\sigma}\arrow{d} & \pi^*\calZ^{\nu,\nu'}_\sigma \arrow{l}{\til\pi}\arrow{d} \\
					\frakl^{\sigma}_{\eta} & \frakp^{\sigma}_{\eta}\arrow{l}{\pi}
				\end{tikzcd}\quad 
				\begin{tikzcd}
					\Gz\times^{P^{\sigma}_0}\pi^*\calZ^{\nu,\nu'}_{\sigma}\arrow{d}\arrow{r}{\mu} & \calZ^{\nu,\nu'}\arrow{d}\\
					\Gz\times^{P^{\sigma}_0}\calX^{\nu,\nu'}_{\sigma}\arrow{r}{\ba\mu}  & \calX^{\nu,\nu'}
				\end{tikzcd},
			\end{aligned}\end{equation}
			where
			\begin{equation*}\begin{aligned}
				\ba\mu:G_{\ubar 0}\times^{P^{\sigma}_0}\calX^{\nu,\nu'}_{\sigma} &\to \calX^{\nu,\nu'} \\
				(h, (gP^{\sigma\le\nu}_0, g'P^{\sigma\le\nu'}_0))&\mapsto (hg P^{\nu}_0, hg' P^{\nu'}_0).
			\end{aligned}\end{equation*}
			Observe the following:
	\begin{enumerate}
		\item
			The complex $\til\pi^*\scrK_{\sigma}$ on the stack $[\pi^*\calZ_{\sigma} / \Pztq^{\sigma}]\cong [\Gz\times^{P^{\sigma}_0}\pi^*\calZ_{\sigma} / \Gztq]$ is canonically isomorphic to $\mu^!\scrK$. 
		\item
			The map $\psi_{\sigma}$ can be described in terms of the $\calZ$-varieties:
			\begin{equation*}\begin{aligned}
				\calH^{\nu,\nu'}_{\sigma} &= \rmH^*_{\Lztq^{\sigma}}(\calZ^{\nu,\nu'}_{\sigma}, \scrK_{\sigma})\xrightarrow{\til\pi^*} \rmH^*_{\Pztq^{\sigma}}(\pi^*\calZ^{\nu,\nu'}_{\sigma}, \til\pi^*\scrK_{\sigma}) \\
				&\cong \rmH^*_{\Gztq}(\Gz\times^{P^{\sigma}_0}\pi^*\calZ^{\nu,\nu'}_{\sigma}, \mu^! \scrK) \xrightarrow{\mu_!} \rmH^*_{\Gztq}(\calZ^{\nu,\nu'}, \scrK) = \calH^{\nu,\nu'}.
			\end{aligned}\end{equation*}
	\end{enumerate}
	The map $\til\pi^*$ is an isomorphism since $\til\pi$ is a vector bundle. Thus it remains to examine the image of map $\mu_!$. Observe also that: 
	\begin{enumerate}[resume]
		\item
			The map $\ba\mu$ preserves the stratification on $\calX^{\nu,\nu'}_{\sigma}$ and $\calX^{\nu,\nu'}$. In other words, for each $[w]\in \calW_I\backslash \calW_J / \calW_I$, we have
			\[
				\ba\mu(\Gz\times^{P^{\sigma}_0}\calX^{\nu,\nu'}_{\sigma,w}) \subseteq \calX^{\nu,\nu'}_w. 
			\]
			Using the Cartesian square in the right of~\eqref{equa:cartesiens}, we deduce that
			\[
				\mu(\Gz\times^{P^{\sigma}_0}\calZ^{\nu,\nu'}_{\sigma,w}) \subseteq \calZ^{\nu,\nu'}_w.
			\]
			Consequently, we see that $\psi_{\sigma}(\calH^{\nu,\nu'}_{\sigma, \calI})\subset \calH^{\nu,\nu'}_{\calI}$.
		\item
			For each $w\in \calW_{\xi,J}$, the map $\mu$ induces an isomorphism by restriction
			\[
				\Gz\times^{P^{\sigma}_0}\pi^*\calZ^{\nu,\nu'}_{\sigma,w} = \mu^{-1}(\calZ^{\nu,\nu'}_{w})\xrightarrow{\cong} \calZ^{\nu,\nu'}_{w}. 
			\]
			Indeed, it follows immediately from the descriptions~\autoref{prop:strZ}~\ref{prop:strZ-iii} and~\autoref{prop:XZpara}~\ref{prop:XZpara-iii}. 	
	\end{enumerate}
	For each ideal $\calI\subset \calW_I\backslash \calW_J / \calW_I$, we show that $\mu_!$ induces an isomorphism
	\begin{equation*}\begin{aligned}
		\calH^{\nu,\nu'}_{\sigma,\calI}= \rmH^*_{\Gztq}(\Gz\times^{P^{\sigma}_0}\pi^*\calZ^{\nu,\nu'}_{\sigma,\calI}, \mu^!i^!_{\calI} \scrK) \xrightarrow{\cong} \rmH^*_{\Gztq}(\calZ^{\nu,\nu'}_{\calI}, i^!_{\calI}\scrK)= \calH^{\nu,\nu'}_{\calI}.
	\end{aligned}\end{equation*}
	We may suppose that $\calI$ is finite and prove this by induction on $\#\calI$. Let $\calJ\subset \calI$ be a subideal such that $\#(\calI \setminus\calJ) = 1$. Denote $[w] = \calI \setminus\calJ$. There are two cases:
	\begin{itemize}
		\item 
			if $[w]\notin \calW_{\xi,J}$, then $\calH^{\nu,\nu'}_{\sigma,w} = 0 = \calH^{\nu,\nu'}_{w}$ by~\autorefitem{prop:SS-bimod}{ii} and its analogue for $\calH^{\nu,\nu'}_{\sigma}$;
		\item 
			if $[w]\in \calW_{\xi,J}$, then by (iv) above, $\mu_!$ yields an isomorphism
			\[
				\rmH^*_{\Gztq}(G_{\ubar 0}\times^{P^{\sigma}_0}\calZ^{\nu,\nu'}_{\sigma, w}, \mu^!i^!_w\scrK_{\sigma}) \xrightarrow{\cong} \rmH^*_{\Gztq}(\calZ^{\nu,\nu'}_{w}, i^!_w\scrK_{\sigma}).
			\]
	\end{itemize}
	Hence $\psi_{\sigma}$ induces an isomorphism on the graded pieces 
	\[
		\calH^{\nu,\nu'}_{\calI} / \calH^{\nu,\nu'}_{\calJ} \cong\calH^{\nu,\nu'}_{w},\quad \calH^{\nu,\nu'}_{\sigma, \calI} / \calH^{\nu,\nu'}_{\sigma,\calJ} \cong\calH^{\nu,\nu'}_{\sigma,w}.
	\]
	The claim follows from this and the induction hypothesis. This completes the proof.
\end{proof}

\section{Density of the image of \texorpdfstring{$\Phi$}{Φ}}\label{sec:density}

Recall that we have defined the following completed extension algebra in~\autoref{subsec:indcusp}:
\[
	\ha\calH = \prod_{\ubar\nu'\in \ubar\Xi}\bigoplus_{\ubar\nu\in \ubar\Xi} \ha\calH^{\nu,\nu'} ,\quad \ha\calH^{\nu,\nu'} = \Ext^*_{\Gztq}(\Ind_{\nu'}\scrC, \Ind_{\nu}\scrC)_{0}
\]
and we have studied the space $\calH^{\nu,\nu'} = \Ext^*_{\Gztq}(\Ind_{\nu'}\scrC, \Ind_{\nu}\scrC)$ in~\autoref{sec:convolution}. The completion $\ha\calH^{\nu,\nu'}$ is equipped with the $\frako$-adic topology, where $\frako = \rmH^{>0}_{\Gztq}$, and $\ha\calH$ is a topological ring with the product (with respect to $\ubar\nu'$) of the coproduct (with respect to $\ubar\nu$) topology of the factors $\ha\calH^{\nu,\nu'}$.\par
Following the lines of~\cite[4.8, 4.9]{vasserot05}, we prove the following result in this section, which is the key instrument in the proof of the classification theorem~\autoref{theo:classification}:
\begin{theo}\label{theo:density}
	The map $\Phi:\bfH_{\xi} \to \ha\calH$ defined in~\autoref{theo:Phi} is injective with dense image.
\end{theo}

\par
Recall that for each ideal $\calI\subset \calW_I\backslash \calW/\calW_I$ and each $\nu,\nu'\in \Xi$, we have defined in~\autoref{subsec:filtration} a subspace $\calH^{\nu,\nu'}_{\calI}\subset\calH^{\nu,\nu'}$. Define the $\frako$-adic completion:
\[
	\ha\calH^{\nu,\nu'}_{\calI} = \calH^{\nu,\nu'}_{\calI}\otimes_{\rmH^*_{\Gztq}}\rmH^*_{\Gztq, 0},\quad \ha\calH_{\calI} = \prod_{\nu'\in \ubar\Xi}\bigoplus_{\nu\in \ubar \Xi} \ha\calH^{\nu,\nu'}_{\calI},
\]
so that $\ha\calH_{\calI}\subset \ha\calH$ is a closed subspace. The parabolic version $\ha\calH^{\nu,\nu'}_{\sigma, \calI}$ and $\ha\calH_{\sigma, \calI}$ are defined in the same way.

\subsection{Preparatory lemmas}\label{subsec:prepa}
We first examine the restriction of $\Phi:\bfH_{\xi}\to \ha\calH$ to the polynomial subalgebra $\bfS_{\xi}$. \par

Recall the point $\bfx^M\in \bbE^M$ is defined in~\autoref{subsec:conjxi} as the image of $\bfx$ under the orthogonal projection $\bbA\to \bbE^M$. For each $\nu\in \Xi$, let $\bfx_{\nu}\in \bbE_{\xi}$\index{x@$\bfx_{\nu}$} be the image of $\bfx^M$ under the canonical isomorphism $(\bbE^M_{\diamond}, \Delta^{\nu}_{\xi})\cong (\bbE_{\xi,\diamond}, \Delta_{\xi})$ from~\autoref{subsec:canrelweyl}. Let $(\bfS_{\xi})_{(\bfx_{\nu}, \eta/2m)}$ denote the completion of $\bfS_{\xi} = \bfk[\bbE_{\xi,\diamond}]\otimes \bfk[u]$ at the point $(\bfx_{\nu},\eta/2m)\in \bbE_{\diamond}\times \bfQ$.

\begin{lemm}\label{lemm:vp}
	For each $\nu\in \Xi$. The map $\Phi_{\nu}:\bfS_{\xi}\to \ha\calH_{\nu}$ factors through the completion $\bfS_{\xi}\hookrightarrow(\bfS_{\xi})_{(\bfx_{\nu},\eta/2m)}$ and induces an isomorphism of topological rings $(\bfS_{\xi})_{(\bfx_{\nu}, \eta/2m)}\cong \ha\calH_{\nu}$.
\end{lemm}
\begin{proof}
	By the definition of $\Phi_{\nu}$ in~\autoref{subsec:affinehecke} in the case $J = \emptyset$, we should examine the composition
	\begin{equation}\begin{aligned}\label{equa:Phinu}
		\bfS_{\xi}\xrightarrow{\cong}\Ext^*_{\Mtq}(\scrC, \scrC)\hookrightarrow  \Ext^*_{\Mtq}(\scrC, \scrC)_{(\bfx, \eta/2m)}.
	\end{aligned}\end{equation}
	Let $\phi = (e, h, f)$ be an $\bfZ$-graded $\fraksl_2$-triple in $\frakm_*$ with $e\in \rmO_{\eta}$.  The arguments in the proof of \autoref{lemm:SE} yields
	\begin{equation}\begin{aligned}\label{equa:CE}
		\Ext^*_{\Mtq}(\scrC, \scrC)\cong \rmH^*_{\Mtq}(\rmO, \bfk)\cong \rmH^*_{Z^q_{\Mt}(e)}\cong \rmH^*_{Z_{\Mt}(e)}\otimes \bfk[u] \cong \bfk[\bbE^M_{\diamond}]\otimes \bfk[u].
	\end{aligned}\end{equation}
	Let $\varphi = \exp(h)\in \bfX_*(M_0)$. The second-to-last isomorphism in~\eqref{equa:CE} is induced by the following group isomorphism from~\eqref{equa:Mphi}:
	\[
		\iota: Z_{\Mt}\times \bfC^{\times}\xrightarrow{\cong} Z^q_{\Mt}(\phi) ,\quad  (g, q)\mapsto (g \varphi(q), q)
	\]
	and under this isomorphism, the cocharacter $(\bfx, \eta / 2m)\in \bfX_*(Z^q_{\Mt}(\phi))_{\bfQ}$ is sent to
	\[
		\iota^*(\bfx, \eta / 2m) = (\bfx - (\eta/2m)\varphi, \eta/2m)\in \bfX_*(Z_{\Mt})_{\bfQ}\oplus \bfQ.
	\]
	We see that $\bfx - (\eta/2m)\varphi = \bfx^M$ because $(\eta/2m)\varphi$ is orthogonal to $\bbE^{M}_{\diamond}$ --- indeed, it is because $\bfX_*(\Tt\cap [M_{\diamond}, M_{\diamond}])_{\bfQ}$ and $\bbE^M_{\diamond} = \bfX_*(Z_{\Mt})_{\bfQ}$ are orthogonal with respect to the Killing form (by the $W_{\Mt}$-invariance) and because $\varphi$ lies in $\bfX_*(\Tt\cap [M_{\diamond}, M_{\diamond}])$. Hence the isomorphisms in~\eqref{equa:CE} induce 
	\[
		\Ext^*_{\Mtq}(\scrC, \scrC)_{(\bfx, \eta/2m)}\cong (\bfk[\bbE^M_{\diamond}]\otimes \bfk[u])_{(\bfx^M, \eta / 2m)}. 
	\]

	On the other hand, it follows from the construction of polynomial action on graded AHAs of Lusztig~\cite[\S 4]{lusztig88} that the first isomorphism in~\eqref{equa:Phinu} is given by the composition of~\eqref{equa:CE} with the isomorphism $\bfk[\bbE^M_{\diamond}]\otimes \bfk[u]\cong \bfk[\bbE_{\xi,\diamond}]\otimes \bfk[u]$ induced by the canonical isomorphism $(\bbE_{\xi,\diamond}, \Delta_{\xi})\cong(\bbE^M_{\diamond}, \Delta^{\nu}_{\xi})$. Therefore~\eqref{equa:Phinu} induces 
	\[
		(\bfS_{\xi})_{(\bfx_{\nu},\eta/2m)}\cong (\bfk[\bbE^M_\diamond]\otimes\bfk[u])_{(\bfx^M, \eta/2m)}\cong  \Ext^*_{\Mtq}(\scrC, \scrC)_{(\bfx, \eta/2m)} = \ha\calH_{\nu},
	\]
	which concludes the proof.
\end{proof}
\begin{lemm}\label{prop:dense}
	The restriction $\Phi\mid_{\bfS_{\xi}}$ is injective and its image is a dense subring of $\ha\calH_e = \prod_{\ubar\nu\in \ubar\Xi}\ha\calH^{\nu,\nu}_e\subset \ha\calH$.
\end{lemm}
\begin{proof}
	By~\autoref{theo:Phi}, the restriction $\Phi\mid_{\bfS_{\xi}}$ coincides with the composition
	\[
		\bfS_{\xi}\xrightarrow{\Phi_{\emptyset}}\prod_{\ubar\nu\in \ubar\Xi}\Ext^*_{\Mztq}(\scrC_{\eta}, \scrC_{\eta})_{0}\xrightarrow{\prod_{\ubar\nu\in \ubar\Xi}\psi_{\nu}}\prod_{\ubar\nu\in \ubar\Xi}\Ext^*_{\Gztq}(\Ind_{\nu}\scrC_{\eta}, \Ind_{\nu}\scrC_{\eta})_{0}\subset \ha\calH,
	\]
	where $\psi_{\nu}$ as defined in~\eqref{equa:indsigma} by the functoriality of $\Ind_{\nu}$. Applying~\autoref{prop:imagepsi} to the case $\calI = \{[e]\}$, we obtain an isomorphism
	\[
		\prod_{\ubar\nu\in \ubar\Xi}\psi_{\nu}  : \prod_{\nu\in \ubar\Xi}\Ext^*_{\Mztq}(\scrC_{\eta}, \scrC_{\eta})_{0} \cong \ha\calH_e.
	\]
	Thus it suffices to show that $\Phi_{\emptyset}$ is injective with dense image. \par
	By~\autoref{lemm:vp}, the map $\Phi_{\nu}$ can be factorised as
	\[
		\bfS_{\xi}\hookrightarrow (\bfS_{\xi})_{(\bfx_{\nu},\eta/2m)}\xrightarrow{\cong} \ha\calH_{\nu}.
	\]  
	Therefore the map $\Phi_{\nu}$ has dense image for each $\nu\in \Xi$. \par
	For every pair $\nu,\nu'\in \Xi$, the quality $\bfx_{\nu} = \bfx_{\nu'}$ holds if and only if $\nu$ and $\nu'$ are $W_{\xi,\bfx}$-conjugate, which means that $\ubar\nu = \ubar \nu'\in \ubar\Xi$. Therefore for any finite subset $\Sigma\subset \ubar\Xi$, the map
	\[
		(\Phi_{\nu})_{\ubar \nu\in \Sigma}:\bfS_{\xi} \to \prod_{\ubar\nu\in \Sigma}(\bfS_{\xi})_{(\bfx_{\nu},\eta/2m)}
	\]
	is injective and its image is dense by the Chinese remainder theorem. Hence, by the definition of product topology, the image of $\Phi_{\emptyset} = (\Phi_{\nu})_{\ubar\nu\in \ubar\Xi}$ is dense.
\end{proof}

\begin{lemm}\label{prop:generator}
	Let $w\in \calW_{\xi}$. The following statements hold:
	\begin{enumerate}
		\item\label{prop:generator-i}
			The element $\Phi(w)$ lies in the subspace $\ha\calH_{\le w}$.
		\item\label{prop:generator-ii}
			The image of $\Phi(w)$ in the quotient $\ha\calH_{w} \cong \ha\calH_{\le w} / \ha\calH_{< w}$ generates the latter as free left (resp. right) $\ha\calH_{e}$-module. \par
	\end{enumerate}
\end{lemm}
\begin{proof}
	Notice that $\calH^{\nu,\nu'}_w$ is a free left $\calH^{\nu,\nu}_e$-module of rank one by~\autoref{lemm:isostr}~\ref{lemm:isostr-iv} and~\autoref{prop:SS-bimod}~\ref{prop:SS-bimod-iii}, and so is $\ha\calH_w$ a free left $\ha\calH_e$-module of rank one. We prove the assertions by induction on $\ell_{\xi}(w)$, where $\ell_{\xi}:\calW_{\xi}\to \bfN$ is the length function of $(\calW_{\xi}, \Delta)$, see~\autoref{prop:relW}. \par
	When $\ell_{\xi}(w) = 0$, the assertions are trivial. When $\ell_{\xi}(w) = 1$, we have $w = s_{\alpha}\in \calW_{\xi}$ for some $\alpha\in \Delta_{\xi}$. Let $J = \{ \alpha\}\subset \Delta_{\xi}$ and denote $s = s_{\alpha}$. Let $\sigma\in \Xi_J$. The isomorphism of~\cite{EM} and~\cite{lusztig95} reads:
	\[
		\bfH_{\xi,J} \cong \Ext^*_{L^{\sigma}_{\diamond,q}}\left(\Ind^{\frakl^{\sigma}_{\diamond}}_{\frakm_{\diamond}}\scrC, \Ind^{\frakl^{\sigma}_{\diamond}}_{\frakm_{\diamond}}\scrC \right).
	\]
	The map $\Phi_{\sigma}:\bfH_{\xi, J}\to \ha\calH_{\sigma}$ (\autoref{subsec:affinehecke}), obtained from this isomorphism by passing to completion at $(\bfx,\eta/2m)$, is injective and has dense image. By the decomposition $\bfH_{\xi,J} = \bfS_{\xi}\oplus \bfS_{\xi}s$ as left (resp. right) $\bfS_{\xi}$-module and the fact that $\Phi_{\sigma}(\bfS_{\xi})\subset \ha\calH_{\sigma,e}$, we have
	\[
		\Phi_{\sigma}(\bfH_{\xi,J}) = \Phi_{\sigma}(\bfS_{\xi}\oplus \bfS_{\xi}s) = \Phi_{\sigma}(\bfS_{\xi})+ \Phi_{\sigma}(\bfS_{\xi})\Phi_{\sigma}(s)\subseteq \ha\calH_{\sigma,e} + \ha\calH_{\sigma,e}\Phi_{\sigma}(s).
	\]
	The short exact sequence
	\[
		0\to \ha\calH_{\sigma,e}\to \ha\calH_{\sigma}\xrightarrow{p} \ha\calH_{\sigma, s}\to 0
	\]
	shows that $p(\Phi_{\sigma}(s))\in \ha\calH_{\sigma,s}$ must be a $\ha\calH_{\sigma,e}$-module generator, since otherwise the image $\Phi_{\sigma}(\bfH_{\xi,J})$ would not be dense in $\ha\calH_{\sigma}$. Taking product over $\ubar\sigma\in \ubar\Xi_{J}$, we see that the image of $\Phi_J(s)$ generates $\prod_{\ubar\sigma}\ha\calH_{\sigma, s}$ over $\prod_{\ubar\sigma}\ha\calH_{\sigma,e}$. \autoref{prop:imagepsi} implies that the image of $\psi_J:\ha\calH_J\to \ha\calH$ is equal to $\ha\calH_{\le s}$. We see that $\Phi(s) = (\psi_J\circ\Phi_J)(s)$ lies in $\ha\calH_{\le s}$ and generates the quotient $\ha\calH_{s}$. \par

	Suppose now that $\ell_{\xi}(w) \ge 2$. We choose $s\in \calW_{\xi}$ such that $\ell_{\xi}(s) = 1$ and $\ell_{\xi}(ws) = \ell_{\xi}(w) - 1$. By induction hypothesis, the element $\Phi(ws)$ (resp. $\Phi(s)$) lies in $\ha\calH_{\le ws}$ (resp. $\ha\calH_{\le s}$). By~\autoref{prop:Hquot} and the fact that $\ell_{\xi}(ws) + \ell_{\xi}(s) = \ell_{\xi}(w)$, the product $\Phi(w) = \Phi(ws)\Phi(s)$ lies in~$\ha\calH_{\le w}$. Since the image of $\Phi(ws)$ in $\ha\calH_{ws}$ and the image of $\Phi(s)$ in $\ha\calH_{s}$ are free generators by induction hypothesis, the surjectivity~\autorefitem{lemm:isostr}{i} implies that $\Phi(w) = \Phi(ws)\Phi(s)$ is in turn a generator of $\ha\calH_{w}$ as free left (resp. right) $\ha\calH_e$-module. 
\end{proof}

\subsection{Proof of~\autoref{theo:density}}

\begin{proof}
	For each finite ideal $\calI\subset \calW_{I}\backslash \calW / \calW_I$, we put 
	\[
		\bfH_{\xi,\calI} =  \bigoplus_{w\in \calW_{\xi}\cap \calI}\bfS_{\xi} w \subset \bfH_{\xi}.
	\]
	We show that 
	\begin{enumerate}
		\item 
			$\Phi$ restricts to an injective map $\Phi\mid_{\bfH_{\xi, \calI}}: \bfH_{\xi, \calI}\hookrightarrow \ha\calH_{\calI}$ and
		\item 
			the subspace $\Phi(\bfH_{\xi, \calI})$ is dense in $\ha\calH_{\calI}$. 
	\end{enumerate}
	We proceed by induction on $\#\calI$. When $\calI = \emptyset$, it is trivial. Suppose that $\#\calI \ge 1$. Let $\calJ\subset \calI$ be a sub-ideal such that $\#\left( \calI \setminus\calJ \right) = 1$ and write $\calI \setminus\calJ = \left\{ [w] \right\}$ for some $w\in \calW_\xi$. If $[w]$ is bad (i.e. $[w]\notin \calW_{\xi}$), then $\bfH_{\xi,\calI} = \bfH_{\xi, \calJ}$ and~\autoref{prop:SS-bimod}~\ref{prop:SS-bimod-ii} implies that $\ha\calH_w = 0$, so $\ha\calH_{\calI} = \ha\calH_{\calJ}$ and the claim follows from induction hypothesis. \par
	Suppose therefore that $[w]\in \calW_{\xi}$. We let $w = \min([w])$ be the minimal representative of $[w]$ in $\calW$ (see~\autoref{prop:ordre}). By~\autoref{prop:generator}~\ref{prop:generator-i}, the image $\Phi(\bfH_{\xi, \calI})$ lies in $\ha\calH_{\xi,\calI}$. By induction hypothesis, $\Phi(\bfH_{\xi,\calJ})$ is dense in $\ha\calH_{\calJ}$. Consider following diagram of short exact sequences:
	\begin{equation*}\begin{aligned}
		\begin{tikzcd}
			0 \arrow{r} & \bfH_{\xi,\calJ} \arrow{r}\arrow{d}{\Phi\mid_{\bfH_{\xi,\calJ}}} &\bfH_{\xi,\calI} \arrow{r}\arrow{d}{\Phi\mid_{\bfH_{\xi,\calI}}} & \bfS_{\xi}w\arrow{r} \arrow{d}& 0 \\
			0 \arrow{r} & \ha\calH_{\calJ}\arrow{r} & \ha\calH_{\calI}\arrow{r} & \ha\calH_{w} \arrow{r} & 0
		\end{tikzcd}.
	\end{aligned}\end{equation*}
	The image of the right vertical arrow, denoted by $V$, is a $\Phi(\bfS_{\xi})$-submodule. By~\autoref{prop:dense}, the image $\Phi(\bfS_{\xi})\subseteq \ha\calH_{e}$ is dense, so the closure $\ba V$ is a $\ha\calH_{e}$-submodule of $\ha\calH_{w}$. By~\autoref{prop:generator}~\ref{prop:generator-ii}, we have $\ba V = \ha\calH_{w}$ and the right vertical arrow is injective. Since the left arrow is also injective with dense image by induction hypothesis, so is the middle arrow. The claim is proven. \par

	The injectivity of $\Phi$ follows immediately from the claim because $\bfH_{\xi} = \bigcup_{\calI}\bfH_{\xi, \calI}$. As the union $\bigcup_{\calI}\ha\calH_{\calI}$ is dense in $\ha\calH$, so is the image of $\Phi$. This completes the proof.
\end{proof}

\section{Simple and proper standard modules}\label{sec:simple}

\subsection{Specialisation \texorpdfstring{$\delta=1$}{δ=1}}

Recall the linear function $\delta\in \bbE^{*}_{\xi,\diamond}$ (\autoref{subsec:canrelweyl}) is such that $\delta^{-1}(1) = \bbE_{\xi}$. Notice that $\delta$ is central in $\bfH_{\xi}$. Set $\bfH'_{\xi} = \bfH_{\xi} / (\delta-1)$. There is a vector space decomposition : 
\[
	\Hxid =\bfk\calW_{\xi} \otimes\bfk[\bbE_{\xi}]\otimes \bfk[u]. 
\]\index{H@$\Hxid$}
\par

On the other hand, recall the extension algebra $\ha\calH$ defined in~\autoref{subsec:indcusp}. We can get rid of the redundant $\Ctm$-equivariance: let $\gamma = (2\lambda_0, 2m, \eta)\in \bfX_*(T\times \Ct\times \Cq)$. Then $\gamma$ acts trivially on $\frakg_{\ubar \eta}$ and $\Gz$ according to the definition of $\Ct$-action (\autoref{subsec:extratori}). Since the complexes $\bfI^{\nu}$ are semisimple, making use of the isomorphism $\Gzq\times \bfC^{\times}_{\gamma}\cong \Gztq$, we have
\begin{equation*}\begin{aligned}
	\Ext^*_{\Gztq}(\bfI^{\nu'},\bfI^{\nu})\cong \Ext^*_{\Gzq}(\bfI^{\nu'},\bfI^{\nu})\otimes \rmH^*_{\bfC^{\times}_{\gamma}}.
\end{aligned}\end{equation*}
Set 
\[
	\ha\calH' = \ha\calH / \ha\calH \cdot \rmH^{>0}_{\bfC^{\times}_{\gamma}} \cong \prod_{\ubar\nu'\in \ubar \Xi}\bigoplus_{\ubar\nu\in \ubar \Xi}\Ext^*_{\Gzq}(\bfI^{\nu'}, \bfI^{\nu})_0.
\]\index{H@$\ha\calH'$}
We define similarly $\ha\calH'^{\nu,\nu'}$ and $\ha\calH'^{\nu,\nu'}_w$ by forgetting the $\Ctm$-equivariance.  \par
Arguing as in the proof of~\autoref{lemm:vp}, we can show that the image $\Phi(\delta)\in \ha\calH_e$ generates the same ideal as $\rmH^{>0}_{\bfC^{\times}_\gamma}$ does and therefore the map $\Phi:\bfH_{\xi}\to \ha\calH$ descends to a map $\Phi':\Hxid\to \ha\calH'$.\index{Ph@$\Phi'$} We obtain the following corollary of~\autoref{theo:density}:
\begin{coro}\label{coro:Phibarre}
	The map $\Phi'$ is injective with dense image.\hfill\qedsymbol
\end{coro}

\begin{rema}
	The case $\delta = 0$ cannot be treated with our approach. In the case where $\delta \neq 0$, the fraction $\eta/2m$, called \emph{slope}, has capital importance in the behaviour of the block $\calO_{x, \eta/2m}(\Hxid)$ defined below.  In the case $\delta\neq 0$, it is conventional to set $\delta=1$. 
\end{rema}

\subsection{The category \texorpdfstring{$\calO(\Hxid)$}{O(H)}}\label{subsec:OH}
We refer to~\cite[2.2]{VV09} for an exposition of the category $\calO$ of the dDAHA $\Hxid$. \par

Let $\bfS'_{\xi} = \bfk[\bbE_{\xi}]\otimes \bfk[u]$ be the polynomial part of $\Hxid$. For each point $x\in\bbE_{\xi,\bfk} = \bbE_{\xi}\otimes_{\bfQ} \bfk$ and each $r\in \bfk$, let $\frako_{x,r}\subset \bfS'_{\xi}$ be the maximal ideal generated by $u - r$ and $f - f(x)$ for all $f\in \bfk[\bbE_{\xi}]$. Given any module $\scrM\in \bfH'_{\xi}\Mod$, consider for each $x\in \bbE_{\xi,\bfk}$ and each $r\in \bfk$, the generalised $(x,r)$-weight space in $\scrM$:
\begin{equation*}\begin{aligned}
	\scrM_{x,r} = \bigcup_{N\ge 0}\left\{ a\in \scrM\;;\;\frako_{x,r}^N a = 0\right\}.
\end{aligned}\end{equation*}\index{M@$\scrM_{x,r}$}
The category of {\bf integrable $\Hxid$-modules} $\calO(\Hxid)\subset \Hxid\mof$ is defined to be the full subcategory of finitely generated $\Hxid$-modules $\scrM$ which satisfy the following condition:
\begin{equation*}\begin{aligned}
	\scrM = \bigoplus_{r\in \bfk}\bigoplus_{x\in\bbE_{\xi, \bfk}} \scrM_{x,r}.
\end{aligned}\end{equation*}
It is known that if $\scrM\in \calO(\Hxid)$, then $\dim_{\bfk}\scrM_{x, r} < \infty$ for each $x$ and $r$, see~\cite[2.1.5(b)]{VV09}.
\par

For any $x\in \bbE_{\xi,\bfk}$ and $r\in \bfk$, we define $\calO_{x,r}( \Hxid )\subset \calO(\Hxid)$\index{O@$\Ox$} to be the full subcategory consisting of those modules $\scrM\in\Hxid\mof$ satisfying
\begin{equation*}\begin{aligned}
	\scrM = \bigoplus_{x'\in \calW_{\xi}x} \scrM_{x,r}.
\end{aligned}\end{equation*}
In other words, the polynomial subalgebra $\bfk[\bbE_{\xi}]$ acts locally finitely on $\scrM$ with eigenvalues in the $\calW_{\xi}$-orbit of $x\in \bbE_{\xi, \bfk}$ and $u$ acts with eigenvalue $r\in \bfk$. These subcategories form the blocks of $\calO(\Hxid)$, so that we have
\begin{equation*}\begin{aligned}
	\calO(\Hxid) = \bigoplus_{\substack{r\in \bfk \\ \calW_{\xi}x \in \bbE_{\xi,\bfk} / \calW_{\xi}}}\calO_{x,r}(\Hxid).
\end{aligned}\end{equation*}

\subsection{Simple \texorpdfstring{$\ha\calH$}{H hat}-modules}\label{subsec:special}
We recall the notion of smooth modules:
\begin{defi}
	Let $A$ be a topological ring. A left $A$-module $\scrM$ is called {\bf smooth} if the action of $A$ on $\scrM$ is continuous when $\scrM$ is equipped with the discrete topology. Equivalently, $\scrM$ is smooth if for each $m\in \scrM$, the annihilator 
	\begin{equation*}\begin{aligned}
		\ann_{A}(m) = \left\{x\in A\;;\; xm = 0  \right\}
	\end{aligned}\end{equation*}
	is open in $A$.
\end{defi}

Finitely generated smooth $\cHd$-modules form a Serre subcategory of $\cHd\mof$, denoted by $\cHd\mof^{\sm}$.\index{H@$\cHd\mof^{\sm}$}\par

Recall that the Lusztig sheaf $\bfI^{\nu} = \Ind_{\nu}\scrC$ is a semisimple complex by~\autoref{prop:BBDG}~\ref{prop:BBDG-i}. Let $\bfI = \bigoplus_{\nu\in \ubar \Xi}\bfI^{\nu}$ be the (infinite) sum.  
\begin{prop}\label{lemm:simples}
	There is a canonical bijection between the set of isomorphism classes of smooth simple modules of $\cHd$ and the set of simple constituents of $\pH$ as follows: given $\scrF\in \Irr\Perv_{\Gz}(\frakg_{\ubar \eta})$, the associated simple smooth $\ha\calH'$-module is given by the multiplicity space:\index{0@$[\bfI:\scrF]$}
	\[
		[\bfI:\scrF] = \bigoplus_{\ubar\nu\in \ubar\Xi}\bigoplus_{k\in \bfZ}\Hom_{\Perv_{\Gz}(\frakg_{\ubar\eta})}(\scrF,\pH^k\bfI^{\nu}).
	\]
\end{prop}
\begin{proof}
	This is a standard property of extension algebras modulo the fact that $\bfI$ is an infinite sum. See~\cite[6.1]{vasserot05} for a proof in this case. 
\end{proof}
Notice that the perverse sheaves $\pH^k\bfI^{\nu}$ are $\Gzq$-equivariant. We regard them as $\Gz$-equivariant perverse sheaves by forgetting the $\Cq$-equivariance. The multiplicity space is not influenced by the change of equivariance because the forgetful functor
\[
	\Perv_{\Gzq}(\frakg_{\ubar\eta})\to \Perv_{\Gz}(\frakg_{\ubar\eta})
\]
is fully faithful. \par

In particular, for the sign $\varepsilon = \eta / |\eta|$, the simple constituents of $\pH \bfI = \bigoplus_{\ubar\nu\in \ubar\Xi}\bigoplus_{k\in \bfZ}\pH^k\bfI^{\nu}$ coincide with the series $\Irr\Perv_{\Gz}(\frakg^{\nil}_{\ubar \eta})_{\xi}$~(\ref{subsec:block}) because $\bfI$ is, by definition, the sum of all possible spiral inductions from $\xi$ modulo the $W_{\xi, \bfx}$-conjugation. Hence we have the following corollary:
\begin{coro}\label{lemm:HP}
	When $\varepsilon = \eta / |\eta|$. The assignment $\scrF\mapsto [\bfI:\scrF]$ yields a bijection
	\[
		\pushQED{\qed} 
		\Irr \ha\calH'\mof^{\sm} \xlongrightarrow{\sim} \Irr\Perv_{\Gz}(\frakg^{\nil}_{\ubar\eta})_{\xi}.\qedhere
		\popQED
	\]
\end{coro}

\subsection{Comparison}
Denote $\Ox = \calO_{\bfx_{\nu}, \eta/2m}(\bfH'_{\xi})$ for any choice of $\nu\in \Xi$. It is the block of the category $\calO(\Hxid)$ associated to the $\calW_{\xi}$-orbits $\{\bfx_{\nu}\}_{\ubar\nu\in \ubar\Xi}$ in $\bbE_{\xi}$, as defined in~\autoref{subsec:OH}. \par
For each $\nu\in \Xi$, let $\bfe_{\nu}\in \cHd$ be the identity element of the factor $\cHd^{\nu,\nu}$. Then the set $\{\bfe_{\nu}\}_{\ubar\nu\in \ubar\Xi}$ satisfies $\bfe_{\nu}\bfe_{\nu'} = \delta_{\ubar\nu,\ubar \nu'}\bfe_{\nu}$.  The following theorem allows us to transfer informations about smooth $\cHd$-modules to the category $\Ox$.
\begin{theo}\label{theo:modHH}
	The pull-back via the homomorphism $\Phi'$ defined in~\autoref{sec:simple} yields an equivalence of category
	\begin{equation*}\begin{aligned}
		\Phi'^*: \cHd\mof^{\text{sm}}\cong \Ox.
	\end{aligned}\end{equation*}
	Moreover, we have
	\[
		(\Phi'^*\scrM)_{\bfx_{\nu}, \eta/2m} = \bfe_{\nu}\scrM,\quad \forall\ubar\nu\in \ubar\Xi.
	\]
\end{theo}
\begin{proof}
	See also~\cite[7.6]{vasserot05}. 
	\begin{enumerate}
		\item[\ul{Step 1}.] 
			For $\scrM\in \cHd\mof^{\sm}$, we show that $\Phi'^*\scrM\in \Ox$. \par
			We first show that $\Phi'^*\scrM$ is finitely generated. Since $\scrM$ is finitely generated, there is a surjective $\cHd$-module morphism
			\[
				(\cHd)^{\oplus r}\xrightarrow{f=(f_j)_{j=1}^r} \scrM,\quad f_j:\cHd\to \scrM.
			\]
			The smoothness of $\scrM$ implies that $f$ is continuous. Consequently, \autoref{coro:Phibarre} implies that the map 
			\[
				(f_j\circ\Phi')_{j=1}^r:(\bfH'_{\xi})^{\oplus r}\to \scrM
			\]
			has dense image, which implies that it is surjective because $\scrM$ is discrete. Therefore $\Phi'^*\scrM$ is finitely generated. \par
			Next, we show that the polynomial part $\bfS'_{\xi} = \bfk[\bbE_{\xi}]\otimes \bfk[u]$ acts with eigenvalues in the set $\left\{ (\bfx_{\nu}, \eta/2m) \right\}_{\ubar\nu\in \ubar\Xi}$. The smoothness of $\scrM$ implies that
			\[
				\scrM = \bigoplus_{\ubar\nu\in \ubar\Xi}\bfe_{\nu}\scrM.
			\]
			By~\autoref{lemm:vp}, it follows that via $\Phi_\nu$, the action of $\bfS'_{\xi}$ on $\bfe_{\nu}\scrM$ factorises through the completion $(\bfS'_{\xi})_{(\bfx_{\nu},\eta/2m)}$. By the smoothness of $\scrM$, this action must factorise through a quotient of $(\bfS'_{\xi})_{(\bfx_{\nu},\eta/2m)}$ of finite length. Hence $\bfS'_{\xi}$ acts locally finitely on $\bfe_{\nu}\scrM$ with eigenvalue $(\bfx_{\nu}, \eta/2m)\in \bbE_{\xi,\bfk}\times \bfk$. It follows that $\Phi'^*\scrM\in \Ox$.
		\item[\ul{Step 2}.]
			We show that $\Phi'^*$ is essentially surjective. \par
			Observe first that by~\autoref{prop:generator}, there is a decomposition for each $\nu,\nu'\in \Xi$
			\begin{equation*}\begin{aligned}
				\cHd^{\nu,\nu'} = \bigoplus_{\substack{w\in \calW_{\xi} \\ \ubar\nu = w\ubar\nu'}}\cHd^{\nu,\nu}_e \bfe_{\nu}\Phi(w)\bfe_{\nu'}.
			\end{aligned}\end{equation*}
			 Let $\scrN\in \Ox$. There is a decomposition into generalised weight spaces of $\bfS'_{\xi}$:
			\[
				\scrN = \bigoplus_{\ubar\nu\in \ubar\Xi} \scrN_{\bfx_\nu, \eta/2m},\quad \dim_{\bfk} \scrN_{\bfx_\nu, \eta/2m} <\infty\quad \forall \ubar\nu\in \ubar\Xi.
			\]
			Denote $\scrN_{\nu} = \scrN_{\bfx_\nu, \eta/2m}$. For each $\ubar\nu\in \ubar\Xi$, the $\bfS'_{\xi}$-action on $\scrN_\nu$ extends to an $(\bfS'_{\xi})_{(\bfx_\nu,\eta/2m)}$-action. Let $\Phi^{-1}_{\nu}:\cHd_{\nu} \cong (\bfS'_{\xi})_{(\bfx_\nu,\eta/2m)}$ denote the inverse of the isomorphism of~\autoref{lemm:vp} and let $\psi_{\nu}^{-1}: \cHd^{\nu,\nu}\cong \cHd_{\nu}$ be the inverse of the isomorphism of~\autoref{prop:imagepsi} in the case $\sigma = \nu$. For each $\ubar\nu,\ubar\nu'\in \ubar\Xi$ and $w\in \calW_{\xi}$ such that $\ubar\nu = w\ubar \nu'$, we define a map
			\begin{equation*}\begin{aligned}
				\cHd^{\nu,\nu}_e \bfe_{\nu}\Phi(w)\bfe_{\nu'}\times \scrN_{\nu'}&\to \scrN_{\nu} \\
				(f\bfe_{\nu}\Phi(w) \bfe_{\nu'}, m) &\mapsto (\Phi^{-1}_{\nu}\psi_{\nu}^{-1}f)(wm)_{\nu},\quad f\in \cHd^{\nu,\nu}_e,\quad m\in \scrN_{\nu'},
			\end{aligned}\end{equation*}
			where $(wm)_{\nu}$ is the projection of $wm\in \scrN$ onto the generalised weight space $\scrN_{\nu}$. Taking summation over $w$, we obtain a $\bfk$-bilinear map $\cHd^{\nu,\nu'}\times \scrN_{\nu'}\to \scrN_{\nu}$, which turns into a $\bfk$-linear map $\cHd^{\nu,\nu'}\to \Hom_{\bfk}(\scrN_{\nu'}, \scrN_{\nu})$ by adjunction. Taking sum over $\nu$ and product over $\nu'$, we get a linear map
			\[
				\alpha_{\scrN}: \cHd\to \prod_{\ubar\nu'\in \ubar\Xi}\bigoplus_{\ubar\nu\in \ubar\Xi}\Hom_{\bfk}\left( \scrN_{\nu'}, \scrN_{\nu} \right).
			\]
			By the finite dimensionality of $\Hom_{\bfk}\left( \scrN_{\nu'}, \scrN_{\nu} \right)$, the map $\alpha_{\scrN}$ is continuous when the right-hand side is equipped with the product with respect to $\ubar\nu'$ of discrete topology. It is easy to see that the composition
			\[
				\alpha_{\scrN}\circ\Phi':\Hxid\to \prod_{\nu'\in \ubar\Xi}\bigoplus_{\nu\in \ubar\Xi}\Hom_{\bfk}\left( \scrN_{\nu'}, \scrN_{\nu} \right)
			\]
			recovers the $\Hxid$-module structure on $\scrN$. In particular $\alpha_{\scrN}\circ\Phi$ is a ring homomorphism. By continuity and the density of the image of $\Phi'$ (\autoref{coro:Phibarre}), the map $\alpha_{\scrN}$ is also a ring homomorphism and defines a $\cHd$-module structure on $\scrN$, which is obviously smooth. This proves the essential surjectivity of $\Phi'^*$.
		\item[\ul{Step 3}.]
			Finally, the injectivity and the density of the image of $\Phi'$ (\autoref{coro:Phibarre}) implies immediately that $\Phi'^*$ is fully faithful. This completes the proof.\qedhere
	\end{enumerate}
\end{proof}

\subsection{Geometric parametrisation of simple modules}
In the rest of~\autoref{sec:simple}, we suppose that the sign $\varepsilon$ is given by $\varepsilon = \eta / |\eta|$, so that~$\bfI^{\nu}$ is supported in the nilpotent cone $\frakg^{\nil}_{\ubar \eta}$ by~\autoref{prop:BBDG}~\ref{prop:BBDG-i}. \par
For $z\in \frakg^{\nil}_{\ubar\eta}$, put $G_{\ubar 0, z} = \Stab_{\Gz}(z)$. 
\begin{theo}\label{theo:classification}
	\begin{enumerate}[label=(\roman*)]
		\item\label{theo:class-ii}
			For any $z\in \frakg_{\ubar \eta}^{\nil}$ and $\chi\in \Irr\Rep\left( \pi_0\left( G_{\ubar 0,z} \right)\right)$\index{z@$(z, \chi)$}, the $\bfH'_{\xi}$-module
			\begin{equation*}\begin{aligned}
				\bfL_{z, \chi} = \Phi'^*\left[\bfI : \mathrm{IC}\left(\chi\right)\right]
			\end{aligned}\end{equation*}
			\index{L@$\bfL_{z, \chi}$}
			is simple if it is non-zero. This happens precisely when $\IC(\chi)\in \Perv_{\Gz}( \frakg^{\nil}_{\ubar\eta})_{\xi}$.
		\item\label{theo:class-iii}
			The simple objects in $\Ox$ are given by $\{\bfL_{z, \chi}\}_{(z, \chi)}$, where $(z, \chi)$ runs over the $G_{\ubar 0}$-conjugacy classes of pairs, where $z\in \frakg_{\ubar \eta}^{\nil}$ and $\chi$ is a irreducible $\Gz$-equivariant local system on the $G_{\ubar 0}$-orbit of $z$ such that $\IC(\chi)\in \Perv_{\Gz}(\frakg^{\nil}_{\ubar \eta})_{\xi}$.
		\item\label{theo:class-iv}
			For each parameter $(z, \chi)$ as above and each $\nu\in \ubar\Xi$, the generalised $(\bfx_{\nu},\eta/2m)$-weight space of $\bfL_{z, \chi}$ is given by:
			\[
				\left( \bfL_{z, \chi} \right)_{\bfx_{\nu}, \eta/2m} = \left[ \bfI^{\nu}: \IC(\chi)\right] = \bigoplus_{k\in \bfZ}\Hom_{\Perv_{\Gz}(\frakg_{\ubar\eta})}(\IC(\chi),\pH^k\bfI^{\nu}).
			\]
	\end{enumerate}
\end{theo}
\begin{rema}
	\begin{enumerate}[label=(\roman*)]
		\item
			The assertion \ref{theo:class-ii} confirms the multiplicity-one conjecture in~\cite{LYIII}.
		\item
			We have supposed that the points $\bfx_{\nu}$ are rational in $\bbE_{\xi}$. However, the non-rational case can be easily reduced to the rational case.
		\item
			The hypothesis that the grading on $\frakg$ is inner, made in~\autoref{subsec:grad}, can be removed, see~\autoref{sec:twisted}.
	\end{enumerate}
\end{rema}
\begin{proof}
	The assertion~\ref{theo:class-ii} follows from~\autoref{theo:modHH} together with~\autoref{lemm:HP}.\par
	Now let $\bfL\in\Ox$ be a simple object. Using (i), $\bfL$ can be equipped with a smooth $\cHd$-module structure, which is simple. By~\autoref{lemm:simples}, it must be isomorphic to the multiplicity space of some simple constituent of $\bfI$, thus one of the $\bfL_{z, \chi}$'s. This proves~\ref{theo:class-iii}.\par
	The assertion~\ref{theo:class-iv} follows from the second statement of~\autoref{theo:modHH}.
\end{proof}

\subsection{Proper standard modules}\label{subsec:standard}
We keep the assumption that $\varepsilon = \eta / |\eta|$. Let $z\in \frakg_{\ubar \eta}^{\nil}$ be a nilpotent element. For each $\nu\in \Xi$, let $\calT^{\nu}_z$ be the fibre of $\alpha^{\nu}: \calT^{\nu}\to \frakg_{\ubar \eta}$ (\autoref{subsec:steinberg}) at $z$ and let $i_z: \calT^{\nu}_z \to \calT^{\nu}$ be the closed inclusion. Consider the following {\it cohomology of Springer fibres}:
\begin{equation*}\begin{aligned}
	\ba\Delta_{z} &= \bigoplus_{\ubar \nu\in \ubar\Xi}\rmH^*\left( \calT^{\nu}_z, i_z^!\dot\scrC_{\nu} \right).\\
\end{aligned}\end{equation*}
Notice that the $\bfk$-vector space $\rmH^*\left( \calT^{\nu}_z, i_z^!\dot\scrC_{\nu} \right)$ is finite-dimensional for each $\ubar\nu\in \ubar\Xi$. By the formalism of convolution algebras, for each $\nu, \nu'\in \Xi$ there is a natural map
\begin{equation*}\begin{aligned}
	\cHd^{\nu, \nu'}\to  \Hom\left(\rmH^*\left( \calT^{\nu'}_z, i^!_z\dot\scrC_{\nu'} \right),\rmH^*\left( \calT^{\nu}_z, i^!_z\dot\scrC_\nu \right)\right).
\end{aligned}\end{equation*}
Taking summation over $\ubar\nu\in \ubar\Xi$ and product over $\ubar\nu'\in \ubar\Xi$, we obtain a smooth $\cHd$-module structure on $\ba\Delta_{z}$.
Besides, there is a natural $\pi_0\left( G_{\ubar 0,z} \right)$-action on $\ba\Delta_{z}$, which commutes with the $\cHd$-action. For any $\chi\in \Irr\Rep\left( \pi_0\left( G_{\ubar 0,z} \right) \right)$, we define
\begin{equation*}\begin{aligned}
	\ba\Delta_{z, \chi} &= \Hom_{\pi_0\left( G_{\ubar 0, z} \right)}\left(\chi, \ba\Delta_{z}\right)\
\end{aligned}\end{equation*}
\index{D@$\ba\Delta_{z, \chi}$}
to be the $\chi$-isotypic component, which is a $\cHd$-submodule of $\ba\Delta_z$. We view $\ba\Delta_z$ as a $\Hxid$-module via $\Phi':\Hxid\to \cHd$.
We call $\ba\Delta_{z, \chi}$ the {\bf proper standard module} of $\bfH_\xi$ with parameter $(z, \chi)$. 
\begin{theo}\label{theo:standard}
	For each pair $(z, \chi)$ as above, the following holds.
	\begin{enumerate}[label=(\roman*)]
		\item\label{theo:std-i}
			$\ba\Delta_{z, \chi}\in \Ox$.
		\item\label{theo:std-ii}
			$\ba\Delta_{z, \chi}\neq 0$ if and only if $\bfL_{z, \chi}\neq 0$. 
		\item\label{theo:std-iii}
			For any pair $(z',\chi')$ as above, the Jordan--H\"older multiplicity of $\bfL_{z', \chi'}$ in $\ba\Delta_{z, \chi}$ is given by
			\begin{equation*}\begin{aligned}
				\left[\ba\Delta_{z, \chi}: \bfL_{z', \chi'} \right] = \sum_k\dim \Hom_{\pi_0\left( G_{\ubar 0, z} \right)}\left(\chi,\rmH^k\left( z^!\mathrm{IC}(\chi') \right)\right).
			\end{aligned}\end{equation*}
	\end{enumerate}
\end{theo}
\begin{proof}
	The assertion~\ref{theo:std-i} results from the smoothness of the $\cHd$-action and~\autoref{theo:modHH}. \par
	The assertion~\ref{theo:std-ii} can be proven with the same arguments as~\cite[8.17]{lusztig95b}, using the orthogonal decomposition~\eqref{equa:dcpGztq} of the equivariant category of $\frakg^{\nil}_{\ubar \eta}$ (\autoref{subsec:block}).
	The assertion~\ref{theo:std-iii} is standard, see~\cite[8.6.23]{CG}.
\end{proof}
\begin{rema}
	The term of \og proper standard modules\fg comes from the theory of quasi-hereditary categories and its generalisations. Retrospectively speaking, the results of~\cite{kato17} show that the graded version of the completed extension algebra $\cHd$ is \og affine properly stratified\fg in the sense of~\cite{kleshchev15}. 
\end{rema}

\section{Example: \texorpdfstring{$G=\Sp(4)$}{G=Sp(4)} untwisted}\label{sec:exem}
We calculate an example to illustrate~\autoref{theo:classification}. Retain the notations of~\autoref{sec:grad} and~\autoref{sec:affineCox}. 
\subsection{Affine root system}
Put $V = \bfC^{4}$ with standard basis $\left\{ e_1, e_2, e_3, e_4 \right\}$. Let $\omega:V\times V\to \bfC$ be the symplectic form defined by
\[
	\left(\omega(e_i, e_j)\right)^4_{i,j=1} = \begin{pmatrix}0 & 0 & 0 & 1 \\ 0 & 0 & 1 & 0 \\ 0 & -1 & 0 & 0 \\ -1 & 0 & 0 & 0\end{pmatrix}.
\]
Let $\left\{ \theta_1, \theta_2, \theta_3, \theta_4 \right\}\subset V^*$ denote the dual basis. Let $G = \Sp(V, \omega)$ and $\frakg = \fraksp(V, \omega)$. Choose the maximal torus $T\subset G$ given by diagonal matrices
\[
	T = \left\{\begin{pmatrix}a & 0 & 0 & 0 \\ 0 & b & 0 & 0 \\ 0 & 0 & b^{-1} & 0 \\ 0 & 0 & 0 & a^{-1}\end{pmatrix}\;;\; a, b\in \bfC^{\times}\right\}.
\]
Let $\frakt_{\bfQ} = \bfX_*(T)_{\bfQ}$.  Denote $\varepsilon_i = e_i\otimes \theta_i - e_{5-i}\otimes \theta_{5-i}\in \frakt_{\bfQ}$ for $i \in \left\{ 1, 2\right\}$. Then $\{\varepsilon_1, \varepsilon_2\}$ is a basis for $\frakt_{\bfQ}$. We use 1/2 times the trace pairing on $\frakt^*_{\bfQ}$ to make the identification $\frakt^*_{\bfQ}\cong \frakt_{\bfQ}$, so that 
\[
	\langle \varepsilon_i, \varepsilon_j\rangle = \delta_{i,j}. 
\]
Then the finite root system is given by
\[
	R(G, T) = \left\{ \pm(\varepsilon_1 - \varepsilon_2), \pm(\varepsilon_1 + \varepsilon_2),\pm 2\varepsilon_1, \pm 2\varepsilon_2 \right\}.
\]
The set of real affine roots is 
\[
	R_{\aff} = \left\{ \pm(\varepsilon_1 - \varepsilon_2) + n\delta, \pm(\varepsilon_1 + \varepsilon_2) + n\delta, \pm 2\varepsilon_1 + n\delta, \pm 2\varepsilon_2 + n\delta \;;\; n\in \bfZ\right\}\subset \bbA_{\diamond}^* = \bfX^*(T_{\diamond})_{\bfQ}.
\]
Choose the base $\Delta^{\kappa_0} = \left\{\alpha_0 = \delta - 2\varepsilon_1 ,\;\alpha_1 = \varepsilon_1 - \varepsilon_2,\; \alpha_2 = 2\varepsilon_2\right\}$, which corresponds to the usual fundamental alcove $\kappa_0$. \par

\subsection{Pseudo-Levi and admissible systems}

Now consider the possible {\it admissible systems} (without $\bfZ$-grading) $\xi = (M, \rmO, \scrC)$, where $M$ is a pseudo-Levi subgroup containing $T$, $\rmO\subset \frakm^{\nil}$ is a nilpotent orbit and $\scrC$ is a cuspidal local system on $\rmO$. Let $I\subset \Delta^{\kappa_0}$ denote the parabolic type of $M$. Up to conjugation by the affine Weyl group, there are four possibilities for $M$:
\begin{enumerate}
	\item
		$I = \emptyset$; in this case $M = T$,\;$\bbE^M = \bbA$,\;$R_{\xi} = R_{\aff} \cong \til{C_2}$;
	\item
	$I = \{\alpha_2\}$; in this case 
	\begin{equation*}\begin{aligned}
		M &= \begin{pmatrix}* & 0 & 0 & 0 \\ 0 & * & * & 0 \\ 0 & * & * & 0 \\ 0 & 0 & 0 & *\end{pmatrix} = \bfC^{\times} \times \Sp(\bfC e_2 \oplus \bfC e_3),\quad \bbE^M = (\bfQ\,\varepsilon_1, 1) \subset \bbA,\\
		R'_{\xi} &= \left\{ \pm \varepsilon_1 + n\delta, \pm 2\varepsilon_1  + n\delta;\;;\; n\in \bfZ\right\}, R_{\xi} = \left\{ \pm \varepsilon_1 + n\delta, \pm 2\varepsilon_1  + (2n+1)\delta\;;\; n\in \bfZ\right\}\cong \til{BC_1};
	\end{aligned}\end{equation*}
	\item $I = \{\alpha_0\}$; in this case
		\begin{equation*}\begin{aligned}
			M &= \begin{pmatrix}* & 0 & 0 & * \\ 0 & * & 0 & 0 \\ 0 & 0 & * & 0 \\ * & 0 & 0 & *\end{pmatrix} = \Sp(\bfC e_1 \oplus \bfC e_4)= \bfC^{\times},\quad \bbE^M = ( (1/2)\varepsilon_1 + \bfQ\,\varepsilon_2, 1) \subset \bbA,\\
			R'_{\xi} &= \left\{ \pm \varepsilon_2 + n\delta, \pm 2\varepsilon_2  + n\delta;\;;\; n\in \bfZ\right\}, R_{\xi} = \left\{ \pm \varepsilon_2 + n\delta, \pm 2\varepsilon_2  + (2n+1)\delta\;;\; n\in \bfZ\right\}\cong \til{BC_1};
		\end{aligned}\end{equation*}
	\item $I = \{\alpha_0, \alpha_2\}$; in this case
		\begin{equation*}\begin{aligned}
			M &= \begin{pmatrix}* & 0 & 0 & * \\ 0 & * & * & 0 \\ 0 & * & * & 0 \\ * & 0 & 0 & *\end{pmatrix} = \Sp(\bfC e_1 \oplus \bfC e_4) \times \Sp(\bfC e_2 \oplus \bfC e_3),\quad \bbE^M = \left\{ (1/2)\varepsilon_1, 1 \right\},\quad R'_{\xi} = \emptyset.
		\end{aligned}\end{equation*}
\end{enumerate}
Case (i) is the principal series case, already treated in~\cite{vasserot05}. Case (iv) is empty. Case (ii) and Case (iii) are conjugate by the {\itshape extended affine Weyl group}. Thus we consider Case (ii) only. 

\subsection{Relative root system and dDAHA}

In Case (ii), $M = \bfC^{\times}\times \Sp(\bfC e_2\oplus \bfC e_3) = \bfC^{\times}\times \SL(\bfC e_2\oplus \bfC e_3)$, there is a unique cuspidal local system $\scrC_{\chi}$ on the regular nilpotent orbit $\frakm^{\nil,\reg}$ with non-trivial central character $\chi:\mu_2\hookrightarrow \bfk^{\times}$. There is a unique $\bbE^M$-alcove $\nu_0$ contained in $\ba\kappa_0\cap \bbE^{M}$, which yields a base for $R_{\xi}$:
\[
	\Delta^{\nu_0}_{\xi} = \left\{ \alpha_0^{\xi} := \delta - 2\varepsilon_1,\;\alpha_1^{\xi} := \varepsilon_1 \right\}\subset R_{\xi}.
\]
The subspace $\bbE^M$ is equipped with the Euclidean structure restricted from $\bbA$ so that $\|\varepsilon_1\|=1$. The relative affine Weyl group $W_{\xi}$ is generated by $s_0$ (reflection w.r.t. $( (1 /2)\varepsilon_1, 1)\in \bbE^M$) and $s_1$ (reflection w.r.t. $(0, 1)\in \bbE^M$).  We have the constants $c_0 = 2$, $c_1 = 3$ for the relative simple affine roots $\alpha_0^{\xi}$ and $\alpha_1^{\xi}$, respectively.
The dDAHA $\bfH_{\xi}$ attached to the based affine root system $(\bbE^M, \Delta^{\nu_0}_{\xi})$ is the generated by the set $\{x, s_0, s_1\}$ ($x^{\alpha_1^{\xi}}= x$ and $x^{\alpha_0^{\xi}} = \delta - 2x$ in the notation of~\autoref{subsec:daha}) over the polynomial ring $\bfk[u,\delta]$ modulo the following relations
\[
	s_0^2 = s_1^2 = 1,\quad s_1x + x s_1 = 6u ,\quad  s_0x - (\delta - x) s_0 = -2u.
\]
The most interesting cases are when $u$ acts by a rational number with denominator $4$ or $8$ (there are finite-dimensional simple modules in these two cases). 
\begin{enumerate}
	\item[Case (ii-1)]
		Suppose $r = -1/8$ (other numbers with denominator $8$ are similar). Consider the following block of integrable $\bfH_{\xi}/(\delta-1)$-modules
		\[
			\calO_{\bfx^M, -1/8}(\bfH_{\xi}/(\delta-1)),\quad \bfx^M\in \bbE^M,\quad x(\bfx^M) = 3/8.
		\]
		The orbit $W_{\xi} \bfx^M$ is equal to $\left\{ x = \pm 3/8 + k \;;\; k\in \bfZ\right\}$. There are three simple objects in this block:
		\begin{equation*}\begin{aligned}
			\bfL_0 &= \bfH_{\xi} / \bfH_{\xi}\cdot\langle \delta-1,\;u+1/8,\;x - 3/8,\;s_0 +1,\;s_1 + 1 \rangle,\quad \dim \bfL_0 =1 \\
			\bfL_+ &= \bfH_{\xi} / \bfH_{\xi}\cdot\langle \delta-1,\;u+1/8,\;x - 5/8,\;s_0 -1 \rangle, \quad \dim \bfL_{+} = \infty \\
			\bfL_- &= \bfH_{\xi} / \bfH_{\xi}\cdot\langle \delta-1,\;u+1/8,\;x + 3/8,\;s_1 -1 \rangle,\quad \dim \bfL_{-} = \infty.
		\end{aligned}\end{equation*}
	\item[Case (ii-2)]
		Suppose $r = -1/4$. Consider the block 
		\[
			\calO_{\bfx^M, -1/4}(\bfH_{\xi}/(\delta-1)),\quad \bfx^M\in \bbE^M,\quad x(\bfx^M) = 3/4.
		\]
		The orbit $W_{\xi} \bfx^M$ is equal to $\left\{ x = \pm 1/4 + k \;;\; k\in \bfZ\right\}$. There are three simple objects in this block:
		\begin{equation*}\begin{aligned}
			\bfL'_0 &= \bfH_{\xi} / \bfH_{\xi}\cdot\langle \delta-1,\;u+1/4,\;x - 3/4,\;s_0 -1,\;s_1 + 1 \rangle,\quad \dim \bfL'_0 =1 \\
			\bfL'_+ &= \bfH_{\xi} / \bfH_{\xi}\cdot\langle \delta-1,\;u+1/4,\;x - 1/4,\;s_0 +1 \rangle, \quad \dim \bfL'_{+} = \infty \\
			\bfL'_- &= \bfH_{\xi} / \bfH_{\xi}\cdot\langle \delta-1,\;u+1/4,\;x + 3/4,\;s_1 -1 \rangle,\quad \dim \bfL'_{-} = \infty.
		\end{aligned}\end{equation*}
\end{enumerate}

\subsection{Perverse sheaves}
We can choose the $\bfZ/m$-grading on $\frakg$ according to the block that we are interested in. We shall only work out Case (ii-1) in detail. Case (ii-2) is interesting but it would take longer paragraphs to analyse. \par
		Consider the adjoint action of $G = \Sp(V, \omega)$ on $\frakg = \fraksp(V, \omega)$. Set $m = 8$ and $\eta = -2$ so that $\eta/2m = -1/8$ is the eigenvalue of $u$. Fix the sign $\varepsilon = \eta / |\eta| = -1$. Following~\eqref{equa:Mphi}, the cocharacter $\lambda_0\in \bfX_*(T)$ is determined by the formula $\lambda_0 / m = \bfx = \bfx^M + (\eta/2m)\varphi$, where $\varphi = \exp(h)$ and $(e, h, f)$ forms an $\fraksl_2$-triple in $\frakm$ with $e\in \rmO$, $h\in \frakt$. We set
		\[
			\varphi(t) = \begin{pmatrix}0 & 0 & 0 & 0 \\ 0 & t & 0 & 0 \\ 0 & 0 & t^{-1} & 0 \\ 0 & 0 & 0 & 0 \end{pmatrix},\quad \lambda_0 = m\bfx^M + (\eta/2)\varphi,\quad \lambda_0(t) = \begin{pmatrix}t^3 & 0 & 0 & 0 \\ 0 & t^{-1} & 0 & 0 \\ 0 & 0 & t & 0 \\ 0 & 0 & 0 & t^{-3}\end{pmatrix}.
		\]
		It follows that $G_{\ubar 0} = G^{\lambda_0} = T$ and 
		\begin{equation*}\begin{aligned}
			\frakg^{\nil}_{-\ubar 2} = \left\{ \begin{pmatrix}0 & 0 & 0 & z_2 \\ 0 & 0 & z_1 & 0 \\ z_3 & 0 & 0 & 0 \\ 0 & -z_3 & 0 & 0\end{pmatrix}\;;\; z_1,z_2,z_3\in \bfC,\quad z_1z_2z_3 = 0\right\}.
		\end{aligned}\end{equation*}
		For $z_1, z_2, z_3\in \bfC$ with $z_1z_2z_3 = 0$, we let $\rmO_{(z_1,z_2,z_3)}\subset \frakg^{\nil}_{-\ubar 2}$ denote the corresponding $T$-orbit. There are $7$ nilpotent $T$-orbits : $\{\rmO_{(a,b,c)}\}_{a,b,c\in \left\{ 0,1 \right\},abc=0}$. The centralisers of orbits are $1, \mu_2$ or $\mu_2\times \mu_2$ depending on how many among $z_1$ and $z_2$ are non-zero. We can induce the cuspidal local system $\scrC_{\chi}$ on $\rmO_{(1, 0,0)}$ to $\rmO_{(1, 1, 0)}$ and $\rmO_{(1, 0, 1)}$ by pulling back via the obvious projection $\rmO_{(1, b, c)}\to \rmO_{(1, 0, 0)}$. Denote them by $\scrL_{(1, 1, 0)}$ and $\scrL_{(1, 0, 1)}$. Then the perverse sheaves 
		\[
			\IC(\scrC_{\chi}),\;\IC(\scrL_{(1, 1, 0)}),\; \IC(\scrL_{(1, 0, 1)})
		\]
		generate the block $\Db_{\Gz}\left( \frakg^{\nil}_{-\ubar 2} \right)_{\xi}$. The correspondence~\autorefitem{theo:classification}{ii} reads:
		\[
			\scrC_{\chi}\longleftrightarrow \bfL_0,\quad\scrL_{(1,1,0)}\longleftrightarrow \bfL_+,\quad \scrL_{(1,0,1)}\longleftrightarrow \bfL_-.
		\]
		This correspondence can be deduced from the description of generalised weight spaces~\autorefitem{theo:classification}{iii} by an examination of the spirals corresponding to $\bbE^M$-alcoves by~\autoref{prop:Ppoids}.

%\appendix
\section{Twisted case}\label{sec:twisted}
In~\autoref{subsec:grad}, we have assumed that the grading on $\frakg$ is inner for the sake of simplicity. In that case, the affine root system $R_{\aff}$ is untwisted. We explain here the modifications needed to accommodate the constructions in~\autoref{sec:grad} and~\autoref{sec:affineCox} to the twisted case. Those parts which are not mentioned are intact. The exposition below follows closely~\cite{LYIII}. 

\subsection{Choice of maximal torus}
We retain the setting of~\autoref{sec:grad} without the assumption that the image of $\theta$ lies in the adjoint group $G^{\ad}$. Fix a pinning $E= (B_0, T_0, U_0/[U_0,U_0]\to \bfC)$ for $G$. This yields an identification of $\Out(G) = \Aut(G) / G^{\ad}$ with $\Aut_E(G)$, the group of automorphisms of $G$ fixing $E$. It also yields a splitting $\Aut(G)= G^{\ad}\rtimes \Aut_E(G)$. \par
Consider the composition 
\[
	\theta_{\Out}:\mu_m\xrightarrow{\theta} \Aut(G)\to \Out(G)
\]
and suppose that $\mu_e = \mu_m / \ker\theta_{\Out}$. Then $\theta_{\Out}$ induces a monomorphism $\varsigma:\mu_e\hookrightarrow \Out(G)$. Since $\Out(G)$ is isomorphic to the automorphism group of the Dynkin diagram of $G$, we have $e \in \left\{ 1, 2, 3 \right\}$.  \par
Let $\mu_e^*\subset \mu_e$ be the set of primitive $e$-th roots of unity. The following statements are proven in~\cite[2.2.2]{LYIII}: 
\begin{enumerate}
	\item \label{LYIII-1}
		Every element $\tau\in G^{\ad}\rtimes \varsigma(\mu^*_e)$ can be conjugated to an element in $T^{\ad}_0 \times \varsigma(\mu^*_e)$, where $T^{\ad}_0 = T_0 / Z_G$. 
	\item \label{LYIII-2}
		If $M = G^{\tau}$ for any element $\tau\in T^{\ad}_0\times \varsigma(\mu^*_e)$, then $T_0^{\varsigma}$ is a maximal torus of $M$.
\end{enumerate}
We may choose $\zeta\in \mu^*_m$ and apply~\ref{LYIII-1} to the automorphism $\theta(\zeta)\in G^{\ad}\rtimes \varsigma(\mu^*_e)$. Therefore, we may suppose without loss of generality that $\theta:\mu_m\to T^{\ad}_0\times\varsigma(\mu^*_e)$. Let $T = T_0^{\varsigma}$ be the fixed points of $\varsigma(\mu_e)$. By~\ref{LYIII-2}, $T$ is a maximal torus of $G^{\varsigma}$. \par

\subsection{Twisted affine root system}

The action of $\varsigma$ yields a $\bfZ/e$-grading on $\frakg$:
\[
	\frakg = \bigoplus_{i\in \bfZ / e}\frakg^i,\quad\frakg^i = \left\{ z\in \frakg\;;\; \varsigma(\zeta) z = \zeta^iz,\;\forall \zeta\in \mu_e\right\}.
\]
The finite root system $R(G, T_0)$ acquires a $\bfZ/e$-grading:
\[
	\til R(G, T) = \left\{ (\alpha, i)\in R(G, T)\times \bfZ/e\;;\; \frakg^i_{\alpha}\neq 0 \right\}.
\]
Let $\Ct$ be a maximal torus with fundamental character $\delta$. The affine root system attached to $(G, T, \varsigma)$ is defined to be
\[
	R_{\aff} = \left\{ \alpha + (k/e)\delta\in \bfX^*(T\times \Ct)_{\bfQ}\;;\; (\alpha,k\!\mod e)\in \til R(G, T) \right\}.
\]
When $e\neq 1$, the affine root system $R_{\aff}$ is a twisted affine root system\footnote{They are denoted by $A^{(2)}_k, D^{(2)}_k, D^{(3)}_4, E^{(2)}_6$ in Kac's notation with the superscript $(e)$ for $e \in \{2,3\}$.}.  \par

The projection of $\theta:\mu_m\to T^{\ad} \times \varsigma(\mu_e)$ to the first factor is a homomorphism $\ba\theta: \mu_{m}\to T^{\ad}$, so that $\theta = (\ba\theta, \theta_{\Out})$. Up to composing $\theta$ with a finite cover $\mu_{m'm}\xrightarrow{[m']} \mu_m$ for some $m'\in \bfZ_{>0}$, there exists a cocharacter $\lambda_0\in \bfX_*(T)$ such that $\lambda_0 \mid_{\mu_m} = \ba\theta$. We may assume such $\lambda_0$ exists and we fix a choice for $\lambda_0$. 

\subsection{Action of \texorpdfstring{$\Ct$}{Ct}}
The action of $\Ct$ defined in~\autoref{subsec:extratori} in the untwisted case also need some modifications. We let $\Ct$ act on a root space $\frakg^i_{\alpha}$ in $\frakg_{\ubar 0}$ (resp. in $\frakg_{\ubar \eta}$) by weight $-e\langle \alpha, \lambda_0 \rangle/ m$ (resp. by weight $e(\eta - \langle \alpha,\lambda_0\rangle)/m$). Notice that these weights are integers.

Set $T_{\diamond} = T\times \Ctm$ and $\bbA_{\diamond} = \bfX_*(T_{\diamond})_{\bfQ}$ as before. The fundamental character of $\Ctm$ is still $(1/2m)\delta$. 

Given any spiral $\frakp_*$, we let $\Ct$ acts on $\frakp_n$ by $e(n - \lambda_0)/m$ for $n\in \bfZ$. \par

\subsection{Spirals and facets}
The point $\bfx$ is defined to be $(e\lambda_0/m, 1)\in \bbA$ so that $\delta(\bfx) = 1$. \par

The correspondence between spirals and facets is unchanged. Given a facet $\tau\in \frakF$, we choose a point $y\in \tau$. Set $\mu_y = m\varepsilon\left( \bfx - y \right) \in \bfX_*\left( T \right)_{\bfQ}$. Then the spiral attached to $\tau$ is $\pre{\varepsilon}\frakp^{\tau}_*$ given by 
\[
	\pre{\varepsilon}\frakp^{\mu_y}_n = \prescript{\mu_y}{\ge n\varepsilon}\frakg_{\ubar n},
\]
see \cite[3.4.4]{LYIII}. The characterisation~\autoref{prop:Ppoids} of $T_{\diamond}$-weights appearing in a spiral remains true. \\[10pt]

\printindex
\printbibliography

\end{document}